\documentclass[11pt]{amsart}

\usepackage{amssymb, amsmath, textcomp, amsthm, amsfonts, bm}
\usepackage{bbm,dsfont}
\usepackage{color}
\usepackage{hyperref}
\usepackage[left=1.03in, right=1.03in, top=1.05in, bottom=1.05in]{geometry}  

\newcommand{\ta}{\texttt{a}}
\newcommand{\bta}{\mbox{\small$\ta$}}
\newcommand{\tc}{\texttt{c}}
\newcommand{\btc}{\mbox{\small$\tc$}}
\newcommand{\stc}{\mbox{\scriptsize$\tc$}}
\newcommand{\sta}{\mbox{\scriptsize$\ta$}}

\usepackage{lipsum}
\usepackage{comment}

\usepackage{mathabx}

\title[Bochner-Riesz problems]{The Bochner-Riesz problem: An old approach revisited}
\author{Shaoming Guo}
\address[SG]{Department of Mathematics\\ University of Wisconsin Madison\\ USA}
\email{shaomingguo@math.wisc.edu}

\author{Changkeun Oh}
\address[CO]{Department of Mathematics\\ University of Wisconsin Madison\\ USA}
\email{coh28@wisc.edu}

\author{Hong Wang}
\address[HW]{School of Mathematics\\ Institute for Advanced Study\\ USA}
\email{Hong.Wang1991@gmail.com}

\author{Shukun Wu}
\address[SW]{Department of Mathematics\\
University of Illinois at Urbana-Champaign\\
Urbana, IL, 61801, USA}
\email{shukunw2@illinois.edu}

\author{Ruixiang Zhang}
\address[RZ]{Department of Mathematics, University of Wisconsin Madison\\and
School of Mathematics\\ Institute for Advanced Study\\ USA}
\email{rzhang@ias.edu}

\date{\today}

\makeatletter

\newcommand*\barredprod{%
  \DOTSB\mathop{%
      \@rodriguez@mathpalette \@rodriguez@overprint@bar \prod
    }\slimits@
}

\newcommand*\@rodriguez@mathpalette[2]{%
  \mathchoice
    {#1\displaystyle      \textfont         {#2}}%
    {#1\textstyle         \textfont         {#2}}%
    {#1\scriptstyle       \scriptfont       {#2}}%
    {#1\scriptscriptstyle \scriptscriptfont {#2}}%
}

\newcommand*\@rodriguez@overprint@bar[3]{%
  \sbox\z@{$#1#3$}%
  \dimen@   = \ht\z@   \advance \dimen@   \p@
  \dimen@ii = \dp\z@   \advance \dimen@ii \p@
  \dimen4 = 1.25\fontdimen 8 #2\thr@@ \relax
  \ooalign{
    \@rodriguez@bar \dimen@ \z@ \cr   
    $\m@th #1#3$\cr
    \@rodriguez@bar \z@ \dimen@ii \cr 
  }%
}
\newcommand*\@rodriguez@bar[2]{%
  \hidewidth \vrule \@width \dimen4 \@height #1\@depth #2\hidewidth
}

\makeatother

\makeatletter
\newcommand\avsuminner[2]{%
  {\sbox0{$\m@th#1\sum$}%
   \vphantom{\usebox0}%
   \ooalign{%
     \hidewidth
     \smash{\vrule height\dimexpr\ht0+1pt\relax depth\dimexpr\dp0+1pt\relax}%
     \hidewidth\cr
     $\m@th#1\sum$\cr
   }%
  }%
}
\makeatother

\def\R{\mathbb{R}}

\def\N{\mathbb{N}}
\def\C{\mathbb{C}}
\def\nint{\mathop{\diagup\kern-13.0pt\int}}
\def\Z{\mathbb{Z}}
\def\T{\mathbb{T}}

\def\W{\mathbb{W}}

\def\beq{\begin{equation}}
\def\eeq{\end{equation}}
\def\bg{\begin{gathered}}
\def\eg{\end{gathered}}

\def\mc{\mathcal}
\def\lesim{\lesssim}

\def\dist{\text{dist}}

\newcommand{\BLka}{\mathrm{BL}_{k, A}}

\newcommand{\bx}{{\bf x}}
\newcommand{\by}{{\bf y}}
\newcommand{\bz}{{\bf z}}
\newcommand{\bfn}{{\bf w}}

\newcommand{\operat}{H^{\lambda}}
\newcommand{\ang}{\measuredangle}
\newcommand{\supp}{\mathrm{supp}}
\newcommand{\bars}{\bar{S}}

\newcommand{\rapid}{\mathrm{RapDec}}
\newcommand{\pknown}{p_{\mathrm{restr}}}

\numberwithin{equation}{section}
\theoremstyle{plain}
\newtheorem{thm}{Theorem}[section]
\newtheorem{prop}[thm]{Proposition}

\newtheorem{lem}[thm]{Lemma}
\newtheorem{cor}[thm]{Corollary}
\newtheorem{defi}[thm]{Definition}

\newtheorem*{conj*}{Conjecture}
\newtheorem{conj}[thm]{Conjecture}
\newtheorem*{openproblem*}{Open Problem}
\newtheorem{remark}[thm]{Remark}

\newcommand{\al}{\alpha}
\newcommand{\be}{\beta}

\newcommand{\de}{\delta}

\newcommand{\e}{\varepsilon}

\newcommand{\ka}{\kappa}
\newcommand{\la}{\lambda}

\newcommand{\om}{\omega}

\newcommand{\wt}{\widetilde}
\newcommand{\wh}{\widehat}

\newcommand{\ZR}{\mathbb{R}}
\newcommand{\ZT}{\mathbb{T}}
\newcommand{\ZZ}{\mathbb{Z}}

\newcommand{\ti}{\tilde}
\newcommand{\Id}{{\bf 1}}

\newcommand{\cC}{{\mathcal C}}

\def\l{\ell}

\begin{document}
\maketitle

\begin{abstract}
We show that the recent techniques developed to study the Fourier restriction problem apply equally well to the Bochner-Riesz problem. This is achieved via applying a pseudo-conformal transformation and a two-parameter induction-on-scales argument. As a consequence, we improve the Bochner-Riesz problem to the best known range of the Fourier restriction problem in all high dimensions. 
\end{abstract}




\section{Introduction}
For $\alpha \geq 0$, the Bochner-Riesz multiplier of order $\alpha$ is  defined by 
\begin{equation}
    m^{\alpha}(\xi):=(1-|\xi|^2)^{\alpha}_{+},
\end{equation}
where $(1-|\xi|^2)_{+}$ is defined to be $1-|\xi|^2$ whenever $|\xi|\le 1$, and $0$ otherwise. 
We define the Bochner-Riesz operator $m^{\alpha}(D)(f)$ to be
\begin{equation}
    m^{\alpha}(D)f(x):=\int_{\R^n}e^{2\pi i \langle x,\xi \rangle }m^{\alpha}(\xi)\hat{f}(\xi)\,d\xi.
\end{equation}
The Bochner-Riesz conjecture is as follows.
\begin{conj}[Bochner-Riesz conjecture]\label{201204conj2_1}
For every $n\ge 2$ and $p \geq 2n/(n-1)$, it holds that 
\begin{equation}\label{210316e1_3}
\|m^{\alpha}(D)f\|_{L^p(\R^n)} \lesssim_{n, \alpha, p}\|f\|_{L^p(\R^n)} 
\end{equation}
whenever $\alpha>n(\frac12-\frac1p)-\frac12$.
\end{conj}
Let $S^{n-1}$ denote the unit sphere in $\R^n$. Let $d\sigma$ denote the surface measure of $S^{n-1}$. Take $f\in L^{\infty}(S^{n-1}, d\sigma)$. Define 
\begin{equation}
    \widehat{f d\sigma}(x):=\int_{S^{n-1}} e^{2\pi i\langle x, \omega\rangle} f(\omega) d\sigma(\omega), \ x\in \R^n.
\end{equation}
The (dual form of) Fourier restriction conjecture is as follows. \begin{conj}[Restriction conjecture]\label{201204conj2_2}
    Let $n\ge 2$. It holds that
    \begin{equation}
        \|\widehat{f d\sigma}\|_{L^p(\R^n)}\lesim_{n, p} \|f\|_{L^{\infty}(S^{n-1}, d\sigma)},
    \end{equation}
    whenever $p>2n/(n-1)$. 
\end{conj} 
Let $\mathfrak{R}f=\widehat{f}|_{S^{n-1}}$ be the sphere restriction operator. By duality and the factorization theory
of Maurey, Nikishin and Pisier (see Bourgain \cite{MR1097257}), the restriction conjecture is equivalent to that 
\begin{equation}\label{201204e2_6}
    \|\mathfrak{R} f\|_{L^p(S^{n-1})}\lesim_{n, p} \|f\|_{L^p(\R^n)},
\end{equation}
for every $1\le p< 2n/(n+1)$. Indeed, \eqref{201204e2_6} is the original Fourier restriction conjecture of Stein (see for instance page 345 of \cite{MR2827930}). Here we state the equivalent version as in Conjecture \ref{201204conj2_2} since it is closer to the operator under investigation in the current paper. \\

Tao \cite{MR1666558} proved that the Bochner-Riesz conjecture implies the restriction conjecture. Moreover, he mentioned in his paper that these two conjectures ``are widely believed to be at least heuristically equivalent". 
The information we would like to convey in the current paper is that, after applying the pseudo-conformal transformation (see \eqref{pseudo_conformal} below), the recently developed techniques in the Fourier restriction literature apply equally well to the Bochner-Riesz problem. These techniques include, but are not limited to, the broad-narrow analysis of Bourgain and Guth \cite{MR2860188}, the polynomial method of Guth in \cite{MR3454378, guth2018}, and the polynomial Wolff axioms obtained by Guth \cite{MR3454378}, Zahl \cite{MR3820441}, and Katz and Rogers \cite{MR3881832} that was applied in the Fourier restriction setting in Guth \cite{MR3454378}, Hickman and Rogers \cite{HR2019} and Hickman and Zahl \cite{hickman2020note}. To be slightly more precise, we will see that after the pseudo-conformal transformation, the above mentioned techniques do not see the differences between the Bochner-Riesz problem and the Fourier restriction problem. As a consequence, we show that the Bochner-Riesz conjecture holds for every $p$ for which the restriction conjecture has been verified in the above mentioned papers.\\

To state our result, let us recall what is known about the restriction conjecture. For $n=3$, Guth \cite{MR3454378} proved the restriction conjecture for $p>3.25$; moreover, for $n\ge 4$, he proved in \cite{guth2018} that the restriction conjecture holds if 
\begin{equation}\label{210316e1_7}
\begin{split}
    & p>2\cdot \frac{3n+1}{3n-3} \text{ for } n \text{ odd}.\\
    & p>2\cdot \frac{3n+2}{3n-2} \text{ for } n \text{ even}. 
\end{split}
\end{equation}
These results improved prior ones due to Tao \cite{MR2033842} and Bourgain and Guth \cite{MR2860188}. More recently, for certain dimensions $n$, in particular all ``large" $n$, Hickman and Rogers \cite{HR2019} and Hickman and Zahl \cite{hickman2020note} applied the polynomial Wolff axioms established by Katz and Rogers \cite{MR3881832} and further improved the results of Guth. The new ranges of $p$ in \cite{HR2019} and \cite{hickman2020note} are a bit technical to state, and we refer the interested readers to Figure 2 in \cite{hickman2020note}.\\

Let $\pknown$ denote the minimum of the exponent $p$ for which the restriction conjecture has been verified in \cite{MR3454378}, \cite{guth2018}, \cite{HR2019} and \cite{hickman2020note}.  
\begin{thm}\label{201204thm2_3}
For every $n\ge 3$, the Bochner-Riesz conjecture holds for every $p\ge \pknown$. 
\end{thm}

Before we comment on the proof of Theorem \ref{201204thm2_3}, let us briefly review some known results about the Bochner-Riesz conjecture in the literature. When $n=2$, the Bochner-Riesz conjecture was resolved by Carleson and Sj\"olin in \cite{MR361607}; see also H\"ormander \cite{MR340924} and Fefferman \cite{MR320624} for alternative proofs. When $n\ge 3$, this conjecture remains open; Fefferman \cite{MR257819}, Bourgain \cite{MR1097257, MR1132294}, Lee \cite{MR2046812} and Bourgain and Guth \cite{MR2860188} made significant partial progress towards this conjecture; see also Christ \cite{MR951506}, Seeger \cite{MR1405600}, Tao \cite{MR1665753} and Lee \cite{MR2264247, MR3803718} for more related results and some endpoints results.
The most recent progress was made by Guth, Hickman and Iliopoulou \cite{MR4047925}.  The approach in the current paper has closer relation with those in Bourgain and Guth \cite{MR2860188} and Guth, Hickman and Iliopoulou \cite{MR4047925}, and therefore we expand a discussion on these two papers.

In these two papers, the authors there viewed the Bochner-Riesz operator as an oscillatory integral operator of the H\"ormander type with positive-definite phase. To be more precise, they followed Carleson and Sj\"olin \cite{MR361607}\footnote{See also Theorem \ref{201204thm3_1} and Theorem \ref{201204thm3_2} below where the same reduction is used.} and reduced the $L^p$ bounds of the Bochner-Riesz operator to the $L^p$ bounds of oscillatory integral operators of the form 
\begin{equation}
    T^{\lambda}f(x):=\int_{\R^{n-1}} e^{2\pi i\phi^{\lambda}(x; \omega)}a^{\lambda}(x; \omega)f(\omega)d\omega, \ \ x\in \R^{n},
\end{equation}
where the phase function $\phi^{\lambda}$ satisfies the H\"ormander condition (see for instance (H1) and (H2) in \cite[page 252]{MR4047925}) and a positive-definite condition (see (H2$^+$) in \cite[page 254]{MR4047925}). For these oscillatory integral operators, it was proven in \cite{MR4047925} that 
\begin{equation}
    \|T^{\lambda}f\|_{L^p(\R^n)}\lesim_{\epsilon, p, \phi} \lambda^{\epsilon} \|f\|_{L^p(\R^{n-1})},
\end{equation}
for every $\lambda\ge 1$, $\epsilon>0$ and every $p$ satisfying \eqref{210316e1_7}. As a consequence, the authors there obtained the $L^p$ bounds \eqref{210316e1_3} of the Bochner-Riesz operator for the same range of $p$. It is worth mentioning that the authors of \cite{MR4047925} also proved that their result is sharp, that is, the range of $p$ in \eqref{210316e1_7} is sharp for $T^{\lambda}$ with a phase function satisfying the above-mentioned positive-definite H\"ormander condition. This also means that if one plans to prove \eqref{210316e1_3} by viewing the Bochner-Riesz operator as an oscillatory integral operator of the H\"ormander type, then the range \eqref{210316e1_7} is the best that one can hope for. 

\medskip

The way that Guth, Hickman and Iliopoulou \cite{MR4047925} proved the sharpness of their result is built on the work of Bourgain \cite{MR1132294}. The sharp examples in these two papers rely crucially on the fact that for general operators $T^{\lambda}$ satisfying the H\"ormander condition, wave packets may be curved. One key observation of the current paper is that, after applying the pseudo-conformal transformation (see \eqref{pseudo_conformal} below) to the Bochner-Riesz operator, all the new wave packets admit similar behavior as those in the Fourier restriction problem, say in \cite{MR3454378}, \cite{guth2018}, \cite{HR2019} and \cite{hickman2020note}. Roughly speaking, this is what allows us to apply the recent techniques developed in the Fourier restriction theory to the Bochner-Riesz problem. Here we would like to emphasize that the use of the pseudo-conformal transformation in the context of the Bochner-Riesz problem is not new. Indeed, Carbery \cite{MR1151328} already used it to prove that the Fourier restriction conjecture for paraboloids implies the Bochner-Riesz conjecture for paraboloids.  \\



In the end, we would like to briefly mention some other interesting features of the proof of our main theorem, and compare them with those in the literature aforementioned, in particular, in \cite{MR4047925} where the latest progress on the Bochner-Riesz conjecture were made. Firstly, in \cite{MR4047925}, the authors there always first reduce phase functions to normal forms (see \cite[page 328]{MR1132294} for the definition of normal forms, which are also referred as ``reduced forms" in \cite{MR4047925}) and then only work with normal forms. In our case, we work with the phase function of the Bochner-Riesz operator directly; indeed, our proof relies on a special parabolic rescaling structure of the phase function (see the proof of Lemma \ref{201229lem5_3}), which makes our induction argument perhaps different from that of \cite{MR4047925} (see also Remark \ref{rem:3_1} for more discussions). Whether this is just a technical point or not remains to be understood. Secondly, it is perhaps worth mentioning that the way we prove the transverse equidistribution estimate (the content of Section \ref{section_transverse}) is slightly different from that in \cite{guth2018} and \cite{MR4047925}; we observe that instead of considering output functions (for instance $H^{\lambda} g$ as in \eqref{201230e3_20}) as a medium, one can directly prove the transverse equidistribution estimate by essentially only working with the input function $g$, which gives us a slightly simpler proof. More details are included in Section 6.

\begin{remark}
\rm

We make a remark on the Bochner-Riesz problem and the Fourier restriction problem in the case $n=3$. In this case, so far the best result for the restriction conjecture is due to Wang \cite{Wang2018ARE}, where she proved that this conjecture holds for $p>3+3/13$. It is not implausible that if one combines the argument of \cite{Wang2018ARE} with that of the current paper, then one may be able to improve Theorem \ref{201204thm2_3} to the same range of $p$ when $n=3$. We do not pursue it here. Regarding the Bochner-Riesz problem in $n=3$, recently Wu \cite{ShukunWu} proved that the Bochner-Riesz conjecture holds for $p \geq \pknown=3.25$ when $n=3$. His proof partially relies on some ideas from Wang \cite{Wang2018ARE}. Our Theorem \ref{201204thm2_3} recovers the result in \cite{ShukunWu} via a quite different and a slightly simpler approach.

\end{remark}

\noindent {\bf Organization of the paper.} In Section \ref{f_section2}, we make several reductions to the Bochner-Riesz problem, including the well-known Carleson-Sj\"olin reduction and a reduction via the pseudo-conformal transformation. In Section \ref{210325section3}, we introduce the induction hypothesis and will further reduce the desired estimate to a broad norm estimate; the structure of our reduction argument is similar to that of Guth, Hickman and Iliopoulou \cite{MR4047925} (see Section 11 there). In Section \ref{f_section4}, we introduce wave packets and prove some of their properties that will be useful in future sections. In Section \ref{210325section5}, we compare wave packets at different scales; this is to prepare for the use of the multi-scale argument as in Guth \cite{guth2018} and Guth, Hickman and Iliopoulou \cite{MR4047925}. In Section \ref{section_transverse}, we prove a transverse equidistribution property of wave packets. After developing the relevant tools in the previous sections, one can already almost identify the Bochner-Riesz problem with the Fourier restriction problem. This allows us to follow Hickman and Rogers \cite{HR2019} and  Hickman and Zahl \cite{hickman2020note}, as is done in Section \ref{f_section7}, Section  \ref{f_section8} and Section \ref{f_section9}, to finish the proof of the desired broad norm estimate.

\medskip

\noindent {\bf Notations.}

\noindent$\bullet$ We write $A(R) \leq \mathrm{RapDec}(R)B$ to mean that for any power $\beta$, there is a constant $C_{\beta}$ such that
\begin{equation}
    A(R) \leq C_{\beta}R^{-\beta}B \;\; \text{for all $R \geq 1$}.
\end{equation}

\noindent$\bullet$ The quantities $p,n$ and $\epsilon$ will be called the \textit{admissible parameters} as the estimates in the paper may be allowed to depend on these parameters. 

\noindent$\bullet$ We introduce  a few other admissible parameters
\begin{equation}\label{constants_z}
    \epsilon^C \leq \delta \ll_{\epsilon} \delta_{n} \ll_{\epsilon} \delta_{n-1} \ll_{\epsilon} \cdots \ll_{\epsilon} \delta_1 \ll_{\epsilon} \epsilon_{\circ} \ll_{\epsilon} \epsilon.
\end{equation}
Here $C$ is some dimensional constant and the notation $A \ll_{\epsilon} B$ indicates that $A \leq C^{-1}_{n,\epsilon}B$ for some large admissible constant $C_{n,\epsilon} \geq 1$.

\noindent$\bullet$ For every number $R>0$ and set $S$, we denote by $N_{R}(S)$ the $R$-neighborhood of the set $S$.

\noindent$\bullet$ We use $B(\bx, r)$ to represent the open ball centered at $\bx$, of radius $r$, in $\ZR^n$. To avoid confusion, we also use $B^{n-1}(x,r)$ to denote the $n-1$ dimensional ball centered at $x$, of radius $r$.
\medskip

\noindent {\bf Acknowledgement.} The authors would like to thank Xiaochun Li, Zane Li, Andreas Seeger, Rajula Srivastava and Terence Tao for valuable discussions. S.G. was
supported in part by the NSF grant DMS-1800274. H.W. was supported by the National Science Foundation under Grant No. DMS-1926686.  R.Z. was supported by the NSF grant DMS-1856541, DMS-1926686 and by the Ky Fan and Yu-Fen Fan Endowment Fund at the Institute
for Advanced Study.

\section{Several reductions}\label{f_section2}
For $\lambda\ge 1$, we define the Carleson-Sj\"{o}lin operator
\begin{equation}\label{101021}
    S^{\lambda}f(x):=\int_{\R^n}e^{2\pi i \lambda|x-y|}a(x-y)f(y)\,dy,
\end{equation}
where $a \in C^{\infty}(\R^n)$ has compact support away from the origin. To prove Theorem \ref{201204thm2_3}, it is standard to reduce it to the following theorem. For the reduction, we refer to Stein \cite[Chapter IX]{MR1232192}. 
\begin{thm}\label{201204thm3_1}
For every $n\ge 3$, it holds that 
\begin{equation}\label{201114e2_2}
    \|S^{\lambda}f\|_{L^p(\R^n)} \lesssim_{ \epsilon} \lambda^{-n/p+\epsilon}\|f\|_{L^p(\R^n)}
\end{equation}
for every $p \geq \pknown$, $\lambda\ge 1$ and $\epsilon>0$.
\end{thm}

Next, we will reduce the $L^p$-boundedness of the Carleson-Sj\"{o}lin operator to that of some operator $\bars^{\lambda}: g \in L^{p}(\R^{n-1}) \mapsto L^{p}(\R^n)$ by freezing one variable. We define the operator $\bars^{\lambda}g$ by
\begin{equation}\label{1010210}
    \bars^{\lambda}g(u,t):=\int_{\R^{n-1}}e^{2\pi i \lambda t^{-1}\sqrt{\lambda^2+|u-t\omega|^2} }a^{\lambda}(u, t; \omega)g(\omega)\,d\omega,
\end{equation}
where  $u=(u_1, \dots, u_{n-1})\in \R^{n-1}$, $\omega=(\omega_1, \dots, \omega_{n-1}) \in \R^{n-1}$, and 
\begin{equation}
    a^{\lambda}(u, t; \omega):=a\big(\frac{u}{\lambda}-\frac{t}{\lambda}\omega,\frac{t}{\lambda}, \omega\big),
\end{equation}
with $a(\cdot, \cdot, \cdot)$ being a smooth function compactly supported in all variables and away from zero in its second variable. 
\begin{thm}\label{201204thm3_2}
Under the above notation, it holds that 
\begin{equation}
\|\bars^{\lambda}g\|_{L^p(\R^n)} \lesssim_{\epsilon}\lambda^{\epsilon}
\|g\|_{L^p(\R^{n-1})}
\end{equation}
for every $n\ge 3$, $p \geq \pknown$, $\lambda \geq 1$ and $\epsilon>0$.
\end{thm}

Let us prove Theorem \ref{201204thm3_1} by assuming Theorem \ref{201204thm3_2}. 
 We fix $a$ in the definition of the operator \eqref{101021}.
Since $a$ has a compact support, by a partition of unity and rotation, we may assume that 
\begin{equation}
    \mathrm{supp}(a) \subset \{(x',x_n,y',y_n): |x_n-y_n| \simeq 1  \},
\end{equation}
where $x'=(x_1, \dots, x_{n-1})$ and $y'=(y_1, \dots, y_{n-1})$.  We write the Carleson-Sj\"{o}lin operator as \begin{equation}
    S^{\lambda}f(x)=\int_{\R}\bars_{y_n}^{\lambda}f_{y_n}(x)\,dy_n,
\end{equation}
where
\begin{equation}\label{210111e2_8}
    \bars_{y_n}^{\lambda}f_{y_n}(x):=
    \int_{\R^{n-1}}e^{2\pi i \lambda|x-y|}a(x-y) f_{y_n}(y')\,dy'
\end{equation}
and
\begin{equation}
    f_{y_n}(y'):=f(y',y_n).
\end{equation}
By H\"{o}lder's inequality and by the fact that $a$ has a compact support, it suffices to prove
\begin{equation}\label{1010216}
    \|\bars^{\lambda}_{y_n}f_{y_n}\|_{L^p(\R^n)} \lesssim \lambda^{-n/p+\epsilon}\|f_{y_n}\|_{L^p_{y'}(\R^{n-1})}
\end{equation}
uniformly for every $|y_n| \lesssim 1$.
By a translation $x_n \mapsto x_n-y_n$, we obtain an operator independent of the variable $y_n$. Hence, we without loss of generality assume that $y_n=0$. To proceed, we write the phase function in the oscillatory integral \eqref{210111e2_8} as follows:
\begin{equation}
    \lambda\big|(x',x_n)-(y', 0)\big|= \lambda |x_n|\sqrt{1+|x'-y'|^2/x_n^2}.
\end{equation}
 We apply a change of variables 
\begin{equation}\label{pseudo_conformal}
 (u_1, \dots, u_{n-1}, t):=(x'/x_n, 1/x_n),
\end{equation}
which is the pseudo-conformal change of variables (see for instance Carbery \cite{MR1151328}, Tao \cite[Section 2.3]{MR2233925} or Rogers \cite{MR2456277}). By the support assumption on $a$, the Jacobian is comparable to one. After this change of variables, our operator $\bars^{\lambda}_0f_0$ becomes
\begin{equation}\label{201114e3_10}
    \tilde{S}^{\lambda}f_0(u,t):=
    \int_{\R^{n-1}}e^{2\pi i \lambda t^{-1}\sqrt{1+|u-ty'|^2 } }
    \tilde{a}(u-ty', t) f_0(y')\,dy',
\end{equation}
where $y'=(y_1, \dots, y_{n-1})$ and $\tilde{a}(\cdot,\cdot)$ is a smooth function that has compact supports in all its variables and is supported away from the origin in its second variable. At this point, due to spatial orthogonality, we see that in \eqref{201114e3_10} we can insert for free a smooth cut-off function $\tilde{a}(y')$, supported near the origin, and it is equivalent to bound 
\begin{equation}
    \int_{\R^{n-1}}e^{2\pi i \lambda t^{-1}\sqrt{1+|u-ty'|^2 } }
    \tilde{a}(u-ty', t)\tilde{a}(y') f_0(y')\,dy'.
\end{equation}
In the end, we pass from scale one to scale $\lambda$ via the change of variables $(u,t) \mapsto \lambda^{-1}(u,t)$. This gives an operator of the form in \eqref{1010210}, and therefore finishes the proof of Theorem \ref{201204thm3_1}. \\

In the end, we make a reduction to the amplitude function $a^{\lambda}(u, t; \omega)$ so that it will have a product form $a_1^{\lambda}(u, t) a_2^{\lambda}(\omega)$. Before the reduction, let us fix the notation. In the definition of $\bar{S}^{\lambda}$, we see that the variable $t$  plays a distinguished role, compared with other variables. Therefore, from now on, we will write $x=(x_1, \dots, x_{n-1})$ and $\bx=(x, t)\in \R^n$.  
Moreover, define 
\begin{equation}\label{210131e2_15}
    \phi^{\lambda}(x, t;\omega):=\lambda \phi\big(\frac{x}{\lambda}, \frac{t}{\lambda};\omega\big):=\frac{\lambda^2}{t}\sqrt{1+|\frac{x}{\lambda}-\frac{t}{\lambda}\omega|^2}.
\end{equation}
Under this notation, we can write 
\begin{equation}
    \bar{S}^{\lambda} g(\bx)=\int_{\R^{n-1}} e^{2\pi i \phi^{\lambda}(\bx; \omega)}a^{\lambda}(\bx; \omega) g(\omega) d\omega.
\end{equation}
Via a standard Fourier expansion, it suffices to consider 
\begin{equation}
\int_{\R^{n-1}} e^{2\pi i  \phi^{\lambda}(x, t; \omega)} a_1^{\lambda}(x, t)a_2(\omega) g(\omega)d\omega,
\end{equation}
where
\begin{equation}
    a_1^{\lambda}(x, t):=a_1(x/\lambda, t/\lambda),
\end{equation}
with $a_1(\cdot, \cdot)$ being a compactly supported function in both variables and supported away from the origin in its second variable, and $a_2(\cdot)$ is compactly supported near the origin. We need to prove 
\begin{equation}\label{201204e3_19}
\Big\|\int_{\R^{n-1}} e^{2\pi i  \phi^{\lambda}(x, t; \omega)} a_1^{\lambda}(x, t)a_2(\omega) g(\omega)d\omega
\Big\|_{L^p(\R^n)}\lesim \lambda^{\epsilon} \|g\|_{L^p(\R^{n-1})}. 
\end{equation}
Such an estimate will be proven via an inductive argument on $\lambda$. In order to set up the induction, we need to be more quantitative about the choice of amplitude functions. Define 
\begin{equation}\label{201230e3_20}
    H^{\lambda}g(\bx) :=  \int_{\R^{n-1}} e^{2\pi i \phi^{\lambda}(x, t;  \omega)}  g(\omega)d\omega.
\end{equation}
Let $C_n$ be a large enough constant such that \begin{equation}
    B_{C_n \lambda}(0)\setminus \{\bx: |t|\le \lambda/C_n\}, \text{ with } B_{\lambda}(0):=[-\lambda, \lambda]^n,
\end{equation}
contains the support of $a^{\lambda}_1$. 
To prove \eqref{201204e3_19}, it is elementary to see that it suffices to prove 
\begin{thm}\label{201204thm3_3}
Under the above notation, it holds that 
\begin{equation}\label{201204e3_21}
    \Big\|H^{\lambda} g\Big\|_{L^p(B_{C_n \lambda}(0)\setminus \{\bx: |t|\le \lambda/C_n\})} \lesim_{ \epsilon} \lambda^{\epsilon} \|g\|_{L^p([0, 1]^{n-1})},
\end{equation}
for every $\epsilon>0$, every $g: [0, 1]^{n-1}\mapsto \C$, dimension $n\ge 3$ and exponent $p \geq \pknown$.
\end{thm}

\section{Reduction to broad norm estimates}\label{210325section3}

\subsection{The induction hypothesis}\label{sec3.1}

We will prove \eqref{201204e3_21} via an inductive argument. In this subsection we set up the induction and state the induction hypothesis. Let $1\le R\le \lambda$. Let $K_{R}$ be a large number that is to be determined: It is much larger than one, but much smaller compared with $R^{\epsilon}$.\footnote{It will be chosen to be $R^{\delta'}$ for some extremely small $\delta'\ll \epsilon$.} Out of certain technical reason, we introduce a scale $\lambda_{R}$ that is comparable to $\lambda$. For $K\ge 2$, we define $\lambda_{K, R}$: The parameter $K$ will eventually be set to be $K_{R}$ and $\lambda_R=\lambda_{K_{R}, R}$. If $\lambda/R< K$, we define $\lambda_{K, R}=C_n \lambda. $ If $\lambda/R\ge K$, then we define
\begin{equation}
    \lambda_{K, R}:=C_n \lambda\Big(2+\frac{1}{K}+\dots+\frac{1}{K^{[\log_{K} (\lambda/R)]-1}}\Big),
\end{equation}
where $[\log_{K}(\lambda/R)]$ refers to the largest integer that does not exceed $\log_{K}(\lambda/R)$. Set $\lambda_R:=\lambda_{K_{R}, R}$. Note that $C_n \lambda\le \lambda_{K, R}\le 3C_n \lambda$. To prove \eqref{201204e3_21}, it suffices to prove \begin{equation}\label{201201e5_7}
    \Big\|H^{\lambda} g\Big\|_{L^p(B_R)} \lesim_{n, p, \epsilon} R^{\epsilon} \|g\|_{L^p([0, 1]^{n-1})},
\end{equation}
for every $1\le R\le \lambda^{1-\epsilon}$, and every cube 
\begin{equation}\label{210201e3_3}
    B_R\subset [-\lambda_{R}, \lambda_{R}]\times [R/C_n, C_n \lambda]. 
\end{equation}
This will be proven via an inductive argument on $\lambda$ and $R$. 
\begin{remark}\label{rem:3_1}
\rm
  We will see that it is crucial to run an induction on both parameters $\lambda$ and $R$. It is worth mentioning that the smaller $R$ is, the more ``singular" our phase function $\phi^{\lambda}$ behaves, which, roughly speaking, can be explained by that $t$ appears in the denominator in \eqref{210131e2_15}. To deal with the case that $R$ is much smaller compared with $\lambda$, we will need to use very particular properties (see for instance \eqref{210131e3_20z} and \eqref{210131e3_28}) of the phase function $\phi^{\lambda}$. That $R$ can be much smaller compared with $\lambda$ seems to indicate that our induction is perhaps different from that of Guth, Hickman and Iliopoulou \cite{MR4047925}. In \cite{MR4047925}, the authors there inserted (and they could as well) a compactly supported amplitude function in the definition of $H^{\lambda}$ so that they always have $|t|\gtrsim \lambda$. 
\end{remark}
\begin{remark}
\rm

The requirement that $R\le \lambda^{1-\epsilon}$ will be used in the proof of Lemma \ref{201108lemma5_1}. 
\end{remark}
The base case of the induction $\lambda=R=1$ is trivial. We assume that we have proven that 
\begin{equation}\label{201201e5_9}
    \Big\|H^{\lambda'} g\Big\|_{L^p(B_{R'})} \le C_{n, p, \epsilon} (R')^{\epsilon} \|g\|_{L^p([0, 1]^{n-1})},
\end{equation}
for every $\lambda'\le \lambda/2$, $1\le R'\le (\lambda')^{1-\epsilon}$, and every cube 
\begin{equation}
    B_{R'}\subset [-\lambda'_{R'}, \lambda'_{R'}]\times [R'/C_n, C_n \lambda']. 
\end{equation}
Our goal is to prove the same hold with $\lambda, R$.

\subsection{Reduction to broad norm estimate}

\medskip

For $\bx\in [-\lambda_{R}, \lambda_{R}]\times [R/C_n, C_n \lambda]$ and $\omega\in  [0, 1]^{n-1}$, we define Gauss maps and rescaled Gauss maps. Define 
\begin{equation}
    G_0(\bx; \omega):=\partial_{\omega_1}\nabla_{\bx}\phi\wedge \dots \wedge \partial_{\omega_{n-1}}\nabla_{\bx}\phi.
\end{equation}
Moreover, define 
\begin{equation}
    G(\bx; \omega):=\frac{G_0(\bx; \omega)}{|G_0(\bx; \omega)|}.
\end{equation}
Define rescaled Gauss map 
\begin{equation}
    G^{\lambda}(\bx; \omega):=G(\bx/\lambda; \omega).
\end{equation}
Via some elementary computation, we obtain 
\begin{lem}\label{210111lem3_1}
Under the above notation, we have that the rescaled Gauss maps equals 
\begin{equation}
\label{Gauss-map}
    G^{\lambda}(\bx; \omega)=\frac{(\omega, 1)}{\sqrt{1+|\omega|^2}}. 
\end{equation}
In particular, the rescaled Gauss map does not depend on $\bx$. 
\end{lem}
\begin{proof}[Proof of Lemma \ref{210111lem3_1}] For future use, we collect a few useful computations regarding the phase function $\phi^{\lambda}$. First of all, 
\begin{equation}\label{210131e3_10}
    \nabla_{x}\phi^{\lambda}(x, t; \omega)=\frac{\lambda}{t}\frac{ (x-t\omega)}{\sqrt{\lambda^2+|x-t\omega|^2}}; \ \  \partial_{t}\phi^{\lambda}(x, t; \omega)=-\frac{\lambda}{t^2}\frac{ \lambda^2+x\cdot (x-t\omega)}{\sqrt{\lambda^2+|x-t\omega|^2}}.
\end{equation}
Next, we compute the gradient in $\omega$:
\begin{equation}
    \partial_{\omega}\phi^{\lambda}(x, t; \omega)=\frac{\lambda(t\omega-x)}{\sqrt{\lambda^2+|x-t\omega|^2}}.
\end{equation}
Regarding the related second order derivatives, we have 
\begin{equation}\label{210131e3_12}
    \partial_{x_i}\partial_{\omega_j} \phi^{\lambda}(x, t; \omega)=\frac{\lambda(x_i-t \omega_i)(x_j-t\omega_j)}{\big(\lambda^2+|x-t\omega|^2\big)^{3/2}} \text{ if } i\neq j, 
\end{equation}
\begin{equation}\label{210131e3_13}
    \partial_{x_i}\partial_{\omega_i} \phi^{\lambda}(x, t; \omega)=-\frac{\lambda \big(\lambda^2+|x-t\omega|^2-(x_i-t\omega_i)^2\big)}{\big(\lambda^2+|x-t\omega|^2\big)^{3/2}},
\end{equation}
and 
\begin{equation}\label{210131e3_14}
    \partial_t\partial_{\omega_j}\phi^{\lambda}(x, t; \omega)=\frac{-x_j\omega\cdot (x-t\omega)+\omega_j x\cdot(x-t\omega)}{\big(\lambda^2+|x-t\omega|^2\big)^{3/2}}.
\end{equation}
Combining all these, we obtain 
\begin{equation}
\begin{split}
    & \partial_{\omega_1}\nabla_{\bx}\phi^{\lambda}\wedge \dots \wedge \partial_{\omega_{n-1}}\nabla_{\bx}\phi^{\lambda}
\end{split}
\end{equation}
is parallel to $(\omega, 1)$. By normalization, we obtain the desired result. 
\end{proof}
Let $K\ge 1$. We divide $[0,1]^{n-1}$ into  caps $\tau$ of side length $K^{-2}$. Let $g_{\tau}$ denote the restriction of $g$ to $\tau$. Moreover, 
\begin{equation}
    G^{\lambda}(\tau):=\{G^{\lambda}(\bx; \omega): \omega\in \tau\}. 
\end{equation}
Let $V\subset \R^n$ be a linear subspace. Let $\ang(G^{\lambda}(\tau), V)$ denote the smallest angle between any non-zero vector $v\in V$ and $v'\in G^{\lambda}(\tau)$. Moreover, we say that $\tau\notin_{K}V$ if $\ang(G^{\lambda}( \tau), V)\ge K^{-1}$; otherwise, we say $\tau\in_K V$. If the value of $K$ is clear from the context, we often abbreviate $\tau\notin_K V$ to $\tau\notin V$. 
Next, let us introduce the notion of broad norms.  Fix $B_{K^2}\subset [-\lambda_R, \lambda_R]\times [R/C_n, C_n \lambda]$ centered at $\bx_0$. Define 
\begin{equation}
    \mu_{\operat g}(B_{K^2}):=\min_{V_1, \dots, V_A\in \text{Gr}(k-1, n)} \Big(\max_{\substack{\tau\notin V_a\\ \text{ for any } 1\le a\le A}} \|\operat g_{\tau}\|_{L^p(B_{K^2})}^p\Big).
\end{equation}
Here $\mathrm{Gr}(k-1,n)$ is the Gressmannian manifold of all $(k-1)$-dimensional subspaces in $\R^n$, and $k$ is to be determined, and $A$ is a parameter that is less important and its choice will become clear later. For $U\subset \R^n$, define 
\begin{equation}
    \|\operat g\|_{\BLka^p(U)}:=\Big(\sum_{B_{K^2}} \frac{|B_{K^2}\cap U|}{|B_{K^2}|} \mu_{\operat g}(B_{K^2})\Big)^{1/p}. 
\end{equation}
Next we should study and prove broad norm estimates.
\begin{thm}[Broad norm estimate]\label{201204thm5_1}
Let $2\le k\le n-1$, and 
\begin{equation}
    p\ge p_n(k):=2+\frac{6}{2(n-1)+(k-1)\prod_{i=k}^{n-1} \frac{2i}{2i+1}}.
\end{equation}
Then for every $\epsilon>0$, there exists $A$ such that 
\begin{equation}
\label{main-esti}
    \|H^{\lambda} g\|_{\BLka^p(B_{R})}\lesim_{K, \epsilon} R^{\epsilon} \|g\|_{L^2}^{2/p}\|g\|_{L^{\infty}}^{1-2/p},
\end{equation}
for every $K\ge 1$, $1\le R\le \lambda$, where $B_{R}$ is a ball of radius $R$ satisfying $B_{R}\subset [-3C_n \lambda, 3 C_n \lambda]^{n-1}\times [R/C_n,  C_n \lambda]$. Moreover, the implicit constant depends polynomially on $K$. 
\end{thm}

In the rest of this section, we will finish the proof of Theorem \ref{201204thm3_3} by assuming Theorem \ref{201204thm5_1}. We will show that Theorem \ref{201204thm5_1} implies \eqref{201204e3_21} whenever $n\ge 3$ and 
\begin{equation}\label{210112e3_18}
    2+\frac{4}{2n-k}\le p\le 2+\frac{2}{k-2}.
\end{equation}
The same optimization process as in Hickman and Zahl \cite[page 4]{hickman2020note} will give Theorem \ref{201204thm3_3}.\\

Let us begin the proof. The main steps are the proofs of the following Lemma \ref{201108lemma5_1} and Lemma \ref{201229lem5_3}. After proving these two lemmas, it is standard to deduce Theorem \ref{201204thm3_3} whenever $p$ satisfies \eqref{210112e3_18}. For the sake of completeness, we provide some details for this step. By a restricted type interpolation, we may assume that $g=\chi_E$ for some set $E$. Let us take $K=R^{\epsilon'}$ for some $0<\epsilon' \ll \epsilon$. By Lemma \ref{201108lemma5_1} and Lemma \ref{201229lem5_3}, and a standard application of the broad-narrow analysis (see for instance \cite[page 358]{MR4047925}) of Bourgain and Guth \cite{MR2860188} and Guth \cite{guth2018}, one obtains that
\begin{equation}
    \|H^{\lambda}g\|_{L^p(B_R)} \lesssim_{n,p,\epsilon} R^{\epsilon/2}\|g\|_2^{p/2}\|g\|_{\infty}^{1-2/p}+R^{\epsilon}R^{-C\epsilon'}\|g\|_{p}.
\end{equation}
Since $g$ is a characteristic function, the above term is bounded by a constant multiple of
\begin{equation}
    \big(R^{-\epsilon/2}+R^{-C\epsilon'} \big)R^{\epsilon}\|g\|_{L^p}.
\end{equation}
Since $R$ is a large number, this closes the induction and completes the proof.\\

From now on, we will focus on the proofs of those two lemma. We start with Lemma \ref{201108lemma5_1}, a decoupling inequality. 

\begin{lem}\label{201108lemma5_1}
Let $2\le m\le n$ and $V$ be an $m$-dimensional linear subspace in $\R^n$. Let $K=K_{R}$ and  $B_{K^2}\subset [-\lambda_R, \lambda_R]\times [R/C_n, C_n \lambda]$. For every $2\le p\le 2m/(m-1)$ and every $\varepsilon>0$, it holds that 
\begin{equation}
\begin{split}
    & \Big\|\sum_{\tau\in V: l(\tau)=K^{-1}}\operat g_{\tau}\Big\|_{L^p(B_{K^2})}\lesssim_{N,\varepsilon, p} \rapid(\lambda) \|g\|_2+\\
    &  \sum_{\beta\in \N_0^{n-1}} \sum_{\beta'\in \N_0^{n}} \sum_{\bfn \in \N^n} \frac{2^{-N |\beta'|} K^{(m-1)(1/2-1/p)+\varepsilon}}{(1+|\beta|)^{10n} (1+|\bfn|)^{100n}}  \Big(\sum_{\tau} \|\operat (g_{\tau}(\omega) b_{\bfn, \beta'}(\bx_1; \omega)e^{2\pi i\beta\cdot \omega})\|_{L^p(B_{K^2})}^p\Big)^{1/p}.
\end{split}
\end{equation}
Here $\bx_1$ is the center of $B_{R}$ (not $B_{K^2}$), $\tau\in V$ means $\tau\in_{K} V$, and $N$ can be as large as we would like. The function $b_{\bfn, \beta'}(\bx_1; \omega)$ satisfies the uniform bound $|b_{\bfn, \beta'}|\lesim 1$ in all parameters. 
\end{lem}

The proof is based on the decoupling inequality of Bourgain and Demeter \cite{MR3374964} and Taylor's expansion. It is quite important that on the right hand side we have $B_{K^2}$ without any tails. 
\begin{remark}
\rm
We would like to mention that it is crucial for $b_{\bfn, \beta'}$ not to depend on the location of $B_{K^2}$. This will allow us sum over all balls $B_{K^2}\subset B_R$. If $b_{\bfn, \beta'}$ were allowed to depend on the location of $B_{K^2}$, the proof of Lemma \ref{201108lemma5_1} would be much simpler and we would not need the requirement that $R\le \lambda^{1-\epsilon}$. 
\end{remark}
\begin{proof}[Proof of Lemma \ref{201108lemma5_1}]
Let us use $\bx_0=(x_0, t_0)$ to denote the center of $B_{K^2}$. We will approximate the phase function $\phi^{\lambda}$ on $B_{K^2}$. Write 
\begin{equation}
    \phi^{\lambda}(\bx; \omega)-\phi^{\lambda}(\bx_0; \omega)+\nabla_{\bx}\phi^{\lambda}(\bx_0; \omega)\bx_0=\nabla_{\bx}\phi^{\lambda}(\bx_0; \omega)\bx+e^{\lambda}_{2, \bx_0}(\bx; \omega).
\end{equation}
Here $e^{\lambda}_{2, \bx_0}$ is the second order remainder term in Taylor's expansion. To see what estimates it satisfies, let us first collect some useful estimates for various derivatives. First of all, from \eqref{210131e3_10}, we see that 
\begin{equation}
    \begin{split}
        & \big|\nabla_x \phi^{\lambda}(\bx; \omega)\big|\lesim \lambda/R,\ \  \big|\partial_t \phi^{\lambda}(\bx; \omega)\big|\lesim (\lambda/R)^2.
    \end{split}
\end{equation}
By taking a further derivative in $\bx$, we obtain 
\begin{equation}
    \big|\nabla^2_x \phi^{\lambda}(\bx; \omega)\big|\lesim 1/R,\ \  \big|\partial_{tt} \phi^{\lambda}(\bx; \omega)\big|\lesim \lambda^2/R^3.
\end{equation}
We see the second order derivative in $x$ is not bad, but that in $t$ is always very bad when $R\ll \lambda$. In other words, we do not necessarily know that the ``error" term $e^{\lambda}_{2, \bx_0}(\bx; \omega)$ has amplitude smaller than one. One way of fixing this problem is to observe that 
\begin{equation}\label{210131e3_26z}
    \big|e^{\lambda}_{2, \bx_0}(\bx; \omega)-e^{\lambda}_{2, \bx_0}(\bx; 0)\big| \lesim K^4/\lambda,
\end{equation}
which follows from 
\begin{equation}\label{210131e3_20z}
    \big|\nabla_{x, t}^2\nabla_{\omega}\phi^{\lambda}(\bx; \omega)\big|\lesim \lambda^{-1}
\end{equation}
and mean value theorems; the pointwise bound \eqref{210131e3_20z} follows from taking a further derivative in $\bx$ for the terms in \eqref{210131e3_12}, \eqref{210131e3_13} and \eqref{210131e3_14}. For later use, we record more estimates on higher order derivatives of the phase function. Via the chain rule, we obtain 
\begin{equation}\label{210131e3_28}
    |\nabla_{\bx}^{\beta} \nabla_{\omega}^{\beta'}\phi^{\lambda}(\bx; \omega)| \lesim_{\beta} \lambda^{-|\beta|+1}
\end{equation}
for every multi-indices $\beta$ and $\beta'$ with $|\beta|, |\beta'|\ge 1$. In particular, these imply 
\begin{equation}\label{210131e3_29z}
    |\nabla^{\beta}_{\omega} e^{\lambda}_{2, \bx_0}(\bx; \omega)|\lesim_{\beta} K^4/\lambda, \text{  for every } |\beta|\ge 1.
\end{equation}
Let us write 
\begin{equation}
    \begin{split}
        & \sum_{\tau\in V} H^{\lambda} g_{\tau}  =\sum_{\tau} \int e^{2\pi i \phi^{\lambda}(\bx; \omega)}g_{\tau}(\omega)d\omega\\
        & =\sum_{\tau} e^{2\pi ie^{\lambda}_{2, \bx_0}(\bx; 0)}\int e^{2\pi i \nabla_{\bx} \phi^{\lambda}(\bx_0; \omega)\bx} e^{2\pi i e^{\lambda}_{2, \bx_0}(\bx; \omega)-2\pi ie^{\lambda}_{2, \bx_0}(\bx; 0)} \big(g_{\tau}(\omega) e^{2\pi i (\phi^{\lambda}(\bx_0; \omega)-\nabla_{\bx}\phi^{\lambda}(\bx_0; \omega)\bx_0)}\big)d\omega. 
    \end{split}
\end{equation}
Denote 
\begin{equation}\label{210201e3_33}
    a^{\lambda}_{ \bx_0}(\bx; \omega):=e^{2\pi i e^{\lambda}_{2, \bx_0}(\bx; \omega)-2\pi ie^{\lambda}_{2, \bx_0}(\bx; 0)}
\end{equation}
and 
\begin{equation}\label{210201e3_34}
    g^{\lambda}_{\tau, \bx_0}(\omega):=\big(g_{\tau}(\omega) e^{2\pi i (\phi^{\lambda}(\bx_0; \omega)-\nabla_{\bx}\phi^{\lambda}(\bx_0; \omega)\bx_0)}\big).
\end{equation}
We apply the Fourier expansion to $a^{\lambda}_{ \bx_0}(\bx; \omega)$ and write it as
\begin{equation}\label{210201e3_35}
    \sum_{\beta\in \N_0^{n-1}} a_{\beta, \bx_0}^{\lambda}(\bx) e^{2\pi i \beta\cdot \omega};
\end{equation}
the bound in \eqref{210131e3_26z} guarantees that such an expansion is meaningful. So far we have 
\begin{equation}
    \sum_{\tau} H^{\lambda} g_{\tau}=\sum_{\beta\in \N_0^{n-1}}\sum_{\tau} e^{2\pi ie^{\lambda}_{2, \bx_0}(\bx; 0)} \int e^{2\pi i (\nabla_{\bx} \phi^{\lambda}(\bx_0; \omega))\cdot\bx} a^{\lambda}_{\beta, \bx_0}(\bx)g^{\lambda}_{\tau, \bx_0}(\omega)e^{2\pi i\beta\cdot \omega}d\omega. 
\end{equation}
By the triangle inequality, 
\begin{equation}\label{210112e3_27}
\begin{split}
    \|\sum_{\tau} H^{\lambda} g_{\tau}\|_{L^p(B_{K^2})}& \le \sum_{\beta} \Big\|\sum_{\tau}\int e^{2\pi i (\nabla_{\bx} \phi^{\lambda}(\bx_0; \omega))\cdot\bx} a^{\lambda}_{\beta, \bx_0}(\bx)g^{\lambda}_{\tau, \bx_0}(\omega)e^{2\pi i\beta\cdot \omega}d\omega \Big\|_{L^p(B_{K^2})}\\
    & \lesssim \sum_{\beta} (1+|\beta|)^{-10n} \Big\|\sum_{\tau}\int e^{2\pi i(\nabla_{\bx} \phi^{\lambda}(\bx_0; \omega))\cdot\bx} g^{\lambda}_{\tau, \bx_0}(\omega)e^{2\pi i\beta\cdot \omega}d\omega \Big\|_{L^p(B_{K^2})},
\end{split}
\end{equation}
where the decay in $\beta$ follows from \eqref{210131e3_28} and \eqref{210131e3_29z}. Via a direct computation, we obtain $(\nabla_{\bx} \phi^{\lambda}(\bx_0; \omega))\cdot\bx$ equals 
\begin{equation}
    \begin{split}
        & \frac{\lambda}{t_0} \frac{x_0-t_0 \omega}{\sqrt{\lambda^2+|x_0-t_0 \omega|^2}} \cdot x- \frac{\lambda(\lambda^2+|x_0|^2-t_0 (x_0\cdot \omega))}{t_0^2 \sqrt{\lambda^2+|x_0-t_0 \omega|^2}}\cdot t.
    \end{split}
\end{equation}
Regarding the parametrized surface 
\begin{equation}\label{210112e3_29}
    \Big(\frac{\lambda}{t_0} \frac{x_0-t_0 \omega}{\sqrt{\lambda^2+|x_0-t_0 \omega|^2}},- \frac{\lambda(\lambda^2+|x_0|^2-t_0 (x_0\cdot \omega))}{t_0^2 \sqrt{\lambda^2+|x_0-t_0 \omega|^2}}\Big),
\end{equation}
we have that the Jacobian of the first $(n-1)$ components in $\omega$ is comparable to $1$; more precisely, without loss of generality, we may assume that the first component of $x_0-t_0w$ is nonzero. Then the Jacobian matrix $\nabla_{\omega}\nabla_{x} \phi^{\lambda}(\bx_0; \omega)$, which are explicitly calculated in \eqref{210131e3_12} and \eqref{210131e3_13}, has eigenvalue $-\lambda^3(\lambda^2+|x_0-t_0\omega|^2)^{-3/2}$ with the eigenvector $x_0-t_0w$, and eigenvalue  and $-\lambda(\lambda^2+|x_0-t_0\omega|^2)^{-1/2}$ with the $n-1$ linearly independent eigenvectors
\begin{equation}
    (-(x_0-t_0w)_2,(x_0-t_0w)_1,0,\ldots,0),\ldots,(-(x_0-t_0w)_n,0,\ldots,0,(x_0-t_0w)_1),
\end{equation}
where $(x_0-t_0w)_i$ is the $i$th component of the vector $x_0-t_0w$. This can be verified by a direct computation. Therefore, we obtain
\begin{equation}\label{210201e3_40}
    \big|\det \nabla_{\omega}\nabla_{x} \phi^{\lambda}(\bx_0; \omega)\big|=\lambda^{n+2} \big(\lambda^2+|x_0-t_0\omega|^2\big)^{-(n+2)/2}\simeq 1. 
\end{equation}
As a consequence, we obtain that the above parametrized surface is regular. Next, we will show that it has non-vanishing Gaussian curvatures. Recall the computation of normal directions in Lemma \ref{210111lem3_1}. Via a direct computation, we see that the Gaussian curvature at a fixed point $\omega_0$ is comparable to 
\begin{equation}
    \begin{split}
        & \det \Big(\nabla^2_{\omega}\langle \nabla_{\bx} \phi^{\lambda}(\bx_0; \omega), (\omega_0, 1)\rangle\Big)\Big|_{\omega=\omega_0}\simeq 1.
    \end{split}
\end{equation}
This, combined with the fact that
\begin{equation}
    \nabla^2_{\omega}\langle \nabla_{\bx} \phi^{\lambda}(\bx_0; \omega), (\omega_0, 1)\rangle\big|_{x_0=0; \omega_0=0}
\end{equation}
is the identity matrix and a simple continuity argument, further implies that  the parametrized surface in \eqref{210112e3_29} is elliptic. This allows us to apply decoupling inequalities for surfaces of non-vanishing Gaussian curvatures, due to Bourgain and Demeter \cite{MR3374964}, and bound \eqref{210112e3_27} by  
\begin{equation}\label{210112e3_33}
    \sum_{\beta} (1+|\beta|)^{-10n} K^{(m-1)(1/2-1/p)+\epsilon} \Big(\sum_{\tau}\Big\|\int e^{2\pi i(\nabla_{\bx} \phi^{\lambda}(\bx_0; \omega))\cdot\bx} g^{\lambda}_{\tau, \bx_0}(\omega)e^{2\pi i \beta\cdot \omega}d\omega \Big\|_{L^p(w_{B_{K^2}})}^p\Big)^{1/p},
\end{equation}
where $w_{B_{K^2}}$ is given by 
\begin{equation}
    \big(1+K^{-2}|\bx-\bx_0|\big)^{-N(\epsilon)},
\end{equation}
where $N(\epsilon)$ is a large constant depending on $\epsilon$ and its choice will become later. 
\begin{remark}
\rm
We will need $N(\epsilon)\to \infty$ much faster than $1/\epsilon$ as $\epsilon\to 0$. For instance, taking $N(\epsilon)=\epsilon^{-100}$ will be more than enough. This choice of weight will be used in the following way: If $|\bx-\bx_0|\ge K^2 \lambda^{\epsilon^{10}}$, then the weight function can be controlled by $\lambda^{-\epsilon^{-90}}$, and therefore is negligible as $\rapid(\lambda)$. 
\end{remark}
The weight function $w_{B_{K^2}}$ has tails that are not allowed in the study of the Bochner-Riesz problem, and we need to get rid of it. For $\bfn\in \Z^n$, let $B_{K^2, \bfn}$ be the translation of $B_{K^2}$ by $K^2\bfn$. 
\begin{equation}\label{210201e3_45}
    \begin{split}
        & \Big\|\int e^{2\pi i(\nabla_{\bx} \phi^{\lambda}(\bx_0; \omega))\cdot\bx} g^{\lambda}_{\tau, \bx_0}(\omega)e^{2\pi i \beta\cdot \omega}d\omega \Big\|_{L^p(w_{B_{K^2}})}\\
        & \le \rapid(\lambda)\|g\|_2+ \sum_{\substack{\bfn\in \Z^n\\ |\bfn|\le \lambda^{\epsilon^{10}}}} (1+|\bfn|)^{-N(\epsilon)} \Big\|\int e^{2\pi i(\nabla_{\bx} \phi^{\lambda}(\bx_0; \omega))\cdot\bx} g^{\lambda}_{\tau, \bx_0}(\omega)e^{2\pi i \beta\cdot \omega}d\omega \Big\|_{L^p(B_{K^2, \bfn})}.
    \end{split}
\end{equation}
As we have a constant coefficient operator in the last display, we can modulate the function $g^{\lambda}_{\tau, \bx_0}$ such that the ball $B_{K^2, \bfn}$ is shifted back to $B_{K^2}$. Let us use $e_{\bfn, K}(\bx_0; \omega)$ to denote the modulation function. Therefore, 
\begin{equation}\label{210131e3_46}
    e_{\bfn, K}(\bx_0; \omega)=e^{2\pi i K^2 (\nabla_{\bx} \phi^{\lambda}(\bx_0; \omega))\cdot\bfn},
\end{equation}
and we have the bound 
\begin{equation}\label{210112e3_36}
    \begin{split}
        & \Big\|\int e^{2\pi i (\nabla_{\bx} \phi^{\lambda}(\bx_0; \omega))\cdot\bx} g^{\lambda}_{\tau, \bx_0}(\omega)e^{2\pi i \beta\cdot \omega}d\omega \Big\|_{L^p(w_{B_{K^2}})}\\
        & \le \rapid(\lambda)\|g\|_2+ \sum_{\substack{\bfn\in \Z^n\\|\bfn|\le \lambda^{\epsilon^{10}}}} (1+|\bfn|)^{-N(\epsilon)} \Big\|\int e^{2\pi i (\nabla_{\bx} \phi^{\lambda}(\bx_0; \omega))\cdot\bx} g^{\lambda}_{\tau, \bx_0}(\omega) e_{\bfn, K}(\bx_0; \omega)e^{2\pi i \beta\cdot \omega}d\omega \Big\|_{L^p(B_{K^2})}.
    \end{split}
\end{equation}
Unfortunately the modulation function in \eqref{210131e3_46} depends on $\bx_0$. Let us try to get rid of this dependence via Taylor's expansion; we learnt this idea from Beltran, Hickman and Sogge \cite{MR4078231}. First of all, let us write the $L^p$ norm on the right hand side of \eqref{210112e3_36} as 
\begin{equation}\label{210201e3_48}
    \Big\|\int e^{2\pi i (\nabla_{\bx} \phi^{\lambda}(\bx_0; \omega))\cdot\bx} g^{\lambda}_{\tau, \bx_0}(\omega) b_{\bfn, K}(\bx_0; \omega)e^{2\pi i \beta\cdot \omega} d\omega \Big\|_{L^p(B_{K^2})},
\end{equation}
with 
\begin{equation}
    b_{\bfn, K}(\bx_0; \omega):=e^{2\pi iK^2 \bfn\cdot(\nabla_{\bx}\phi^{\lambda}(\bx_0; \omega)-\nabla_{\bx}\phi^{\lambda}(\bx_0; 0))}.
\end{equation}
To get rid of the dependence of $b_{\bfn, K}$ on $\bx_0$, we carry out a Taylor expansion of it about $\bx_1$, an arbitrary point inside $B_{R}$; for convenience, we take $\bx_1$ to be the center of $B_R$. Recall the bound in \eqref{210131e3_28}. As a consequence, we obtain that 
\begin{equation}
    \Big|\nabla_{\bx}^{\beta}\Big[K^2 \bfn\cdot(\nabla_{\bx}\phi^{\lambda}(\bx; \omega)-\nabla_{\bx}\phi^{\lambda}(\bx; 0))\Big]\Big|\lesim K^2|\bfn| \lambda^{-|\beta|}, \text{ for every } \beta,
\end{equation}
which further implies that
\begin{equation}
    \big|\nabla_{\bx}^{\beta} b_{\bfn, K}(\bx; \omega)\big| \lesim (K^2|\bfn|)^{|\beta|} \lambda^{-|\beta|} \text{ for every } |\beta|\ge 1.
\end{equation}
This allows us to write \begin{equation}
    b_{\bfn, K}(\bx_0; \omega)=\sum_{\beta'} 2^{-N|\beta'|} (K^2|\bfn|)^{|\beta'|}\lambda^{-|\beta'|} b_{\bfn, \beta'}(\bx_1; \omega) (\bx_0-\bx_1)^{\beta'}+\mathrm{RapDec}(\lambda),
\end{equation}
where the dependence of $b_{\bfn, \beta'}$ on $K$ has been compressed, and these Taylor coefficients satisfy the uniform bound $|b_{\bfn, \beta'}(\bx_1; \omega)|\lesim 1$. 
We go back to \eqref{210201e3_48} and bound it by 
\begin{equation}
    \sum_{\beta'} 2^{-N|\beta'|}  (K^2\lambda^{\epsilon^{10}})^{|\beta'|} R^{|\beta'|}\lambda^{-|\beta'|}\Big\|\int e^{2\pi i (\nabla_{\bx} \phi^{\lambda}(\bx_0; \omega))\cdot\bx} g^{\lambda}_{\tau, \bx_0}(\omega) b_{\bfn, \beta'}(\bx_1; \omega)e^{2\pi i \beta\cdot \omega}d\omega \Big\|_{L^p(B_{K^2})}.
\end{equation}
Recall the assumption on $R$ that $R\le \lambda^{1-\epsilon}$. It can be used to control the constant factors in the last display. So far we have controlled \eqref{210201e3_45} by 
\begin{equation}\label{210201e3_54}
    \rapid(\lambda)\|g\|_2+\sum_{\beta'} \sum_{\bfn}  2^{-N|\beta'|}(1+|\bfn|)^{-N(\epsilon)} \Big\|\int e^{2\pi i (\nabla_{\bx} \phi^{\lambda}(\bx_0; \omega))\cdot\bx} g^{\lambda}_{\tau, \bx_0}(\omega) b_{\bfn, \beta'}(\bx_1; \omega)e^{2\pi i \beta\cdot \omega}d\omega \Big\|_{L^p(B_{K^2})}.
\end{equation}
Now each term we have is of the form 
\begin{equation}
\begin{split}
    &  \int e^{2\pi i (\nabla_{\bx} \phi^{\lambda}(\bx_0; \omega))\cdot\bx} g^{\lambda}_{\tau, \bx_0}(\omega)b_{\bfn, \beta'}(\bx_1; \omega) e^{2\pi i \beta\cdot \omega} d\omega\\
    & =\int e^{2\pi i(\nabla_{\bx} \phi^{\lambda}(\bx_0; \omega))\cdot\bx} g_{\tau}(\omega) e^{2\pi i (\phi^{\lambda}(\bx_0; \omega)-\nabla_{\bx}\phi^{\lambda}(\bx_0; \omega)\bx_0)}b_{\bfn, \beta'}(\bx_1; \omega)e^{2\pi i \beta\cdot \omega}d\omega\\
    &=\int e^{2\pi i \phi^{\lambda}(\bx; \omega)-2\pi i e^{\lambda}_{2, \bx_0}(\bx; \omega)}g_{\tau}(\omega) b_{\bfn, \beta'}(\bx_1; \omega)e^{2\pi i \beta\cdot \omega}d\omega. 
\end{split}
\end{equation}
We apply a Fourier expansion ``back" (compared with the one in \eqref{210201e3_33}--\eqref{210201e3_35}) to get rid of $e^{\lambda}_{2, \bx_0}(\bx; \omega)$ and obtain 
\begin{equation}
    \begin{split}
        & \Big\|\int e^{2\pi i (\nabla_{\bx} \phi^{\lambda}(\bx_0; \omega))\cdot\bx} g^{\lambda}_{\tau, \bx_0}(\omega)b_{\bfn, \beta'}(\bx_1; \omega) e^{2\pi i \beta\cdot \omega} d\omega \Big\|_{L^p(B_{K^2})}\\
        & \lesssim  \sum_{\beta''\in \N_0^{n-1}}
        (1+|\beta''|)^{-10n}
        \Big\|\int e^{2\pi i \phi^{\lambda}(\bx; \omega)} g_{\tau}(\omega) b_{\bfn, \beta'}(\bx_1; \omega) e^{2\pi i (\beta+\beta'')\cdot \omega}d\omega \Big\|_{L^p(B_{K^2})},
    \end{split}
\end{equation}
which, when substituted into \eqref{210201e3_54} and \eqref{210112e3_33}, implies the desired the decoupling inequality. 
\end{proof}




\begin{lem}[Parabolic rescaling]\label{201229lem5_3}
Let $w=(w_1, \dots, w_{n-1}) \in [0,1]^{n-1}$ and $K=K_{R}$. Let $\tau=[w_1,w_1+K^{-1}] \times\dots\times [w_{n-1},w_{n-1}+K^{-1}]$ with $\tau\subset [0, 1]^{n-1}$.
Under the induction hypothesis, for every  ball $B_R$ with $1\le R\le \lambda^{1-\epsilon}$ satisfying 
\begin{equation}
    B_R\subset [-\lambda_{K, R}, \lambda_{K, R}]\times [R/C_n, C_n \lambda]
\end{equation}
we obtain
\begin{equation}
    \|H^{\lambda}g_{\tau}\|_{L^p(B_R)} \lesssim_{\epsilon, \delta}R^{\epsilon} R^{\delta}( K^{-1})^{(n-1)-\frac{2n}{p}}
\|g_{\tau}\|_{L^p([0, 1]^{n-1})},
\end{equation}
for every $\delta>0$.
\end{lem}


\begin{proof}[Proof of Lemma \ref{201229lem5_3}]
We use a change of variables
\begin{equation}
    \eta:=-w+\omega.
\end{equation} 
Recall that our phase function is given by 
\begin{equation}
    \phi^{\lambda}(x, t; \omega)=\frac{\lambda}{t}\sqrt{\lambda^2+|x-t\omega|^2}.
\end{equation}
Under the change of variables, it becomes 
\begin{equation}
    \frac{\lambda}{t}\sqrt{\lambda^2+|x-t(\eta+w)|^2}=\frac{\lambda}{t}\sqrt{\lambda^2+|(x-tw)-t\eta|^2}.
\end{equation}
Next we apply the change of variables 
\begin{equation}
    x-tw\mapsto x; \ t\mapsto t.
\end{equation}
Under this change of variables, let us assume that $B_R$ becomes $\widetilde{B}_R$. Note that $\widetilde{B}_R$ is a parallelogram: The smallest rectangle that contains it is at most twice as large as $B_R$. 
Now our phase becomes 
\begin{equation}
    \frac{\lambda}{t}\sqrt{1+|x-t\eta|^2},
\end{equation}
and therefore we need to bound 
\begin{equation}
    \|H^{\lambda} g_{\widetilde{\tau}}\|_{L^p(\widetilde{B}_R)}, \text{ with } \widetilde{\tau}=[0, 1/K]^{n-1}.
\end{equation}
Next we apply a scaling argument and set $\eta\mapsto \eta/K.$
Our phase becomes 
\begin{equation}
    \frac{\lambda}{t}\sqrt{1+|x-t\frac{\eta}{K}|^2}.
\end{equation}
In the end, we apply the change of variables $x\mapsto Kx;  t\mapsto K^2 t.$ 
It is important to keep track of the domain of evaluating the $L^p$ norm. Let us use $\widetilde{D}_R$ to denote the smallest rectangle that contains the image of $\widetilde{B}_R$ after the above change of variables. After some elementary computation, we see that 
\begin{equation}\label{201202e6_33}
    \widetilde{D}_R \subset [-\frac{\lambda_{K, R}}{K}-\frac{C_n R}{K}, \frac{\lambda_{K, R}}{K}+\frac{C_n R}{K}]^{n-1}\times [\frac{R}{K^2}\frac{1}{C_n}, \frac{\lambda}{K^2}C_n],
\end{equation}
and $\widetilde{D}_R$ is of dimension \begin{equation}
    (2\frac{R}{K})\times (2\frac{R}{K})\times \dots \times (2\frac{R}{K})\times \frac{R}{K^2}. 
\end{equation}
Set $\lambda'=\lambda/K$, $R'=R/K$ and $R''=R/K^2$. What is important in \eqref{201202e6_33} is that 
\begin{equation}
    \frac{\lambda_{K, R}}{K}+\frac{C_n R}{K}\le  \lambda'_{K, R''}\le \lambda'_{R''},
\end{equation}
which is exactly the reason of having a complicated expression of $\lambda_{K, R}$ in our induction hypothesis. As a consequence, we know that 
\begin{equation}\label{210201e3_69}
    \widetilde{D}_R \subset [-\lambda'_{K, R''}, \lambda'_{K, R''}]\times [R''/C_n, \lambda' C_n]. 
\end{equation}
\begin{remark}
\rm
Here we make a remark on the choice of the $t$-interval $[R/C_n, C_n \lambda]$ in \eqref{210201e3_3} from our induction hypothesis. In our current setup, it should not be replaced by anything like $[R/C_n, C_n R]$ or $[R/C_n, C_n R^{\frac{1}{1-\epsilon}}]$. For instance, if we replace it by the latter, then the analogue of \eqref{210201e3_69}, which is
\begin{equation}
    \widetilde{D}_R \subset [-\lambda'_{K, R''}, \lambda'_{K, R''}]\times [R''/C_n, (R'')^{\frac{1}{1-\epsilon}} C_n],
\end{equation}
does not hold, and therefore we can not apply the induction hypothesis. If one would really like to use a time interval of the form $[R/C_n, C_n R^{\frac{1}{1-\epsilon}}]$ in the induction hypothesis, which may bring some convenience like $t$ is always essentially comparable to $R$, then one can cut $\widetilde{D}_R$ into smaller balls and then make use of the extra gain in $K$ from \eqref{210201e3_71}, which is not explored in the current setup. 
\end{remark}
From now on, we will try to bound $\|H^{\lambda'} \widetilde{g}\|_{L^p(\widetilde{D}_R)},$
where $\widetilde{g}$ is now a function supported on $[0, 1]^{n-1}$. It remains to prove 
\begin{equation}\label{210201e3_71}
    \|H^{\lambda'} \widetilde{g}\|_{L^p(\widetilde{D}_R)} \lesim_{\epsilon, \delta} (R'/K)^{\epsilon} (R'/K)^{\delta} \|\widetilde{g}\|_p.
\end{equation}
Actually we will not need the gain in $K$, and only need to prove 
\begin{equation}
    \|H^{\lambda'} g\|_{L^p(\widetilde{D}_R)} \lesim_{\epsilon, \delta} R^{\epsilon} R^{\delta} \|g\|_p,
\end{equation}
for every $\delta>0$ and every $g$. \\

Let us denote 
\begin{equation}\label{201231e5_44}
    D_{R'}\subset [-\lambda'_{K, R''}, \lambda'_{K, R''}]\times [R''/C_n, \lambda' C_n]
\end{equation} 
a rectangular box of dimension 
\begin{equation}
    R'\times \dots R'\times R''.
\end{equation}
We need to prove 
\begin{equation}\label{201204e6_43}
     \|H^{\lambda'} g\|_{L^p(D_{R'})} \lesim_{\epsilon, \delta} R^{\epsilon} R^{\delta} \|g\|_p.
\end{equation}
To prove such a bound, we need the following discrete  version of the operator $H^{\lambda'}$. This step is where we apply our induction hypothesis. 
\begin{lem}[Lemma 11.8 \cite{MR4047925}]\label{201229lem5_4}
Let $\mathcal{D}$ be a maximal $(R'')^{-1}$-separated discrete subset of $[0, 1]^{n-1}$. Then 
\begin{equation}\label{201229e5_47}
    \Big\|\sum_{\omega_{\theta}\in \mathcal{D}} e^{i \phi^{\lambda}(\cdot\ ; \omega_{\theta})} F(\omega_{\theta}) \Big\|_{L^p(B_{R''})} \lesim_{\epsilon} (R'')^{\epsilon} (R'')^{(n-1)/p'} \|F\|_{\ell^p(\mathcal{D})},
\end{equation}
for every $F: \mathcal{D} \rightarrow \C$ and every ball $B_{R''}\subset D_{R'}$ of radius $R''$.  
\end{lem}
\begin{proof}[Proof of Lemma \ref{201229lem5_4}] The proof relies an approximation argument via Taylor's expansion, and is essentially the same as that of Lemma 11.8 \cite{MR4047925}, see also Section 5 of \cite{MR2860188}. Let $\bx_0$ be the center of $B_{R''}$. Let $\psi: \R^{n-1}\to \R$ be a function supported on the ball of radius 2 centered at the origin, satisfying $0\le \psi\le 1$ and $\psi(\omega)=1$ for every $|\omega|\le 1$. For each $\omega_{\theta}\in \mathcal{D}$, define $\psi_{\theta}(\omega):=\psi(10R''(\omega-\omega_{\theta}))$. For every $x\in B_{R''}$, the sum on the left hand side of \eqref{201229e5_47} is a constant multiple of 
\begin{equation}\label{201231e5_48}
    \begin{split}
        (R'')^{n-1}\int_{\R^{n-1}} e^{2\pi i \phi^{\lambda}(\bx; \omega)} e^{-2\pi i\phi^{\lambda}(\bx_0; \omega)+2\pi i\phi^{\lambda}(\bx_0; \omega_{\theta})} \Big[\sum_{\omega_{\theta}\in \mathcal{D}} e^{-2\pi i \Omega^{\lambda}_{\theta}(\bx; \omega)} F(\omega_{\theta}) \psi_{\theta}(\omega)\Big]d\omega, 
    \end{split}
\end{equation}
where 
\begin{equation}
    \Omega^{\lambda}_{\theta}(\bx; \omega):=\phi^{\lambda}(\bx; \omega)-\phi^{\lambda}(\bx; \omega_{\theta})-\phi^{\lambda}(\bx_0; \omega)+\phi^{\lambda}(\bx_0; \omega_{\theta}). 
\end{equation}
By mean value theorems and the bound \eqref{210131e3_28}, we see that 
\begin{equation}\label{210201e3_78}
    |\Omega_{\theta}^{\lambda}(\bx; \omega)|\lesim |\bx-\bx_0||\omega-\omega_{\theta}|\lesim 1,
\end{equation}
whenever $x\in B_{R''}$ and $\omega$ is in the support of $\psi_{\theta}$. 
Roughly speaking, this allows us to treat $e^{-2\pi i \Omega^{\lambda}_{\theta}(\bx; \omega)}$ as the constant function $1$, which can be made rigorous via Fourier expansion. To apply Fourier expansion in $\bx$, we need to examine higher order derivatives of the function. Similar to how we obtained \eqref{210201e3_78}, we obtain from \eqref{210131e3_28} and mean value theorems that \begin{equation}
    \begin{split}
     \big|\nabla_{\bx}^{\beta} \Omega_{\theta}^{\lambda}(\bx; \omega)\big| \lesim 1, \ \  \big|\nabla_{\bx}^{\beta} e^{2\pi i\Omega_{\theta}^{\lambda}(\bx; \omega)}\big| \lesim 1
    \end{split}
\end{equation}
for every multi-index $\beta$. As a consequence, we are able to bound \eqref{201231e5_48} by 
\begin{equation}
    (R'')^{n-1} \sum_{k\in \Z^n} (1+|k|)^{-(n+1)}|H^{\lambda} f_k(x)|, 
\end{equation}
where 
\begin{equation}
    f_k(\omega):=\sum_{\omega_{\theta}\in \mathcal{D}} F(\omega_{\theta}) c_{k, \theta}(\omega)\psi_{\theta}(\omega),
\end{equation}
with $c_{k, \theta}$ satisfying the uniform bound $\|c_{k, \theta}\|_{\infty}\le 1$. We apply our induction hypothesis to $H^{\lambda} f_k$ on the ball $B_{R''}$; in particular, \eqref{201231e5_44} guarantees that our induction hypothesis is applicable. This is gives us the desired estimate.
\end{proof}

Let us begin the proof of \eqref{201204e6_43}. Out of certain technical reason of handling tails, we introduce new notation. For given $R, \lambda$  satisfying $R\le \lambda^{1-\epsilon}$, let $a_{\lambda, R}(x, t)$ be a non-negative smooth bump function supported on $[-2\lambda_R, 2\lambda_R]\times [R/(2C_n), 2C_n \lambda]$, and equal to one on $[-\lambda_R, \lambda_R]\times [R/C_n, C_n \lambda]$. Define 
\begin{equation}\label{210112e3_63}
    H^{\lambda, R} g(\bx):=a_{\lambda, R}(x, t)\int e^{2\pi i\phi^{\lambda}(\bx; \omega)} g(\omega) d\omega.
\end{equation}
This notation will only be used in this section. Note that for all the $\bx$ that we care about, that is, $\bx\in [-\lambda_R, \lambda_R]\times [R/C_n, C_n \lambda]$, it always holds that 
\begin{equation}
    H^{\lambda, R} g(\bx)=H^{\lambda} g(\bx).
\end{equation}
Let us return to the proof of \eqref{201204e6_43}. 
Cover $[0, 1]^{n-1}$ by caps $\theta$ of side length $(R'')^{-1}$. Decompose $g$ as $g=\sum_{\theta} g_{\theta}$. Let $\omega_{\theta}$ be the center of $\theta$. Define 
\begin{equation}
    H^{\lambda'}_{\theta} g(\bx):=e^{-2\pi i \phi^{\lambda'}(\bx; \omega_{\theta})} H^{\lambda'} g(\bx),
\end{equation}
and 
\begin{equation}
    H^{\lambda', R''}_{\theta} g(\bx):=a_{\lambda', R''}(x, t) e^{-2\pi i \phi^{\lambda'}(\bx; \omega_{\theta})} H^{\lambda'} g(\bx),
\end{equation}
so that for every $\bx\in B_{R''}\subset D_{R'}$, it holds that
\begin{equation}
    \begin{split}
        H^{\lambda'} g(\bx)=H^{\lambda', R''} g(\bx)& =\sum_{\theta}  e^{2\pi i \phi^{\lambda'}(\bx; \omega_{\theta})} H^{\lambda'}_{\theta} g_{\theta}(\bx)\\
        &=\sum_{\theta}  e^{2\pi i \phi^{\lambda'}(\bx; \omega_{\theta})} H^{\lambda', R''}_{\theta} g_{\theta}(\bx).
    \end{split}
\end{equation}
To proceed, we need the following lemma. 
\begin{lem}\label{201204lem6_4}
For every $\delta>0$, the following statement holds. For every $\bx$ we have  
\begin{equation}
    H^{\lambda', R''}_{\theta} g_{\theta}(\bx)=H^{\lambda', R''}_{\theta} g_{\theta}*\eta_{(R'')^{1-\delta}}(\bx)+  \rapid(\lambda') \|g\|_2, 
\end{equation}
for some rapidly decreasing function $\eta$ such that $|\eta|$ admits a smooth, rapidly decreasing majorant $\zeta: \R^n\mapsto [0, \infty)$ which is locally constant at scale 1.  
\end{lem} 
The proof of this lemma is pretty standard, see for instance the bottom of page 363 in \cite{MR4047925} and Lemma 5.8 there. 
\begin{remark}
\rm
The parameter $\delta>0$ will be picked to be extremely small, and much smaller compared with $\epsilon$. The use of the parameter $\delta$ and the scale $(R'')^{\delta}$ is the reason for $K$ to depend on $R$, which further explains the necessity of the polynomial dependence of the implicit constant on $K$ in Theorem \ref{201204thm5_1}. 
\end{remark}



One way we will be using the locally constant property of $\zeta$  is 
\begin{equation}\label{201231e5_60}
    \zeta_{(R'')^{1-\delta}}(\bx)\lesim (R'')^{\delta}\zeta_{(R'')^{1-\delta}}(\by) \text{ if } |\bx-\by|\lesim R''.
\end{equation}
To control the $L^p$ norm of $H^{\lambda'} g=H^{\lambda', R''} g$ on $D_{R'}$, we cover $D_{R'}$ by finitely-overlapping $R''$-balls, and let $B_{R''}$ be some member of this cover. Let $\bx_0$ denote the center of $B_{R''}$, then 
\begin{equation}
    |H^{\lambda', R''} g(\bx_0+\bz)|\lesim R^{\delta} \int_{\R^n} \Big|\sum_{\theta} e^{2\pi i \phi^{\lambda'}(\bx_0+\bz; \omega_{\theta})} H^{\lambda', R''}_{\theta} g_{\theta}(\by) \Big| \zeta_{(R'')^{1-\delta}}(\bx_0-\by) d\by,
\end{equation}
for every $\bz$ with $|\bz|\le R''$. We follow the same lines of proof as in \cite[page 364]{MR4047925}. To be more precise, we take the $L^p$ norm in $\bz$ and see that $\|H^{\lambda'}g\|_{L^p(B_{R''})}$ can be bounded by 
\begin{equation}
    R^{\delta} \int_{\R^n} \Big\|\sum_{\theta} e^{2\pi i \phi^{\lambda'}(\bx_0+\bz; \omega_{\theta})} H^{\lambda', R''}_{\theta} g_{\theta}(\by) \Big\|_{L^p(\{|\bz|\le R''\})} \zeta_{(R'')^{1-\delta}(\bx_0-\by)}d\by.
\end{equation}
For each fixed $\by$, we apply Lemma \ref{201229lem5_4} and see that the $L^p$ norm in the above display can be bounded by 
\begin{equation}
    (R'')^{\epsilon} (R'')^{(n-1)/p} \Big(\sum_{\theta} |H_{\theta}^{\lambda', R''} g_{\theta}(\by)|^p \Big)^{1/p}.
\end{equation}
We then apply the locally constant property \eqref{201231e5_60} again, sum over all balls $B_{R''}$ inside $D_{R'}$ and obtain 
\begin{equation}
    \|H^{\lambda'} g\|_{L^p(D_{R'})}\lesim_{\epsilon} (R'')^{\epsilon+O(\delta)} (R'')^{(n-1)/p'-n/p}\Big( \int_{\R^n} \sum_{\theta} \|H^{\lambda', R''} g_{\theta}\|^p_{L^p(D_{R'}-\by)} \zeta_{(R'')^{1-\delta}}(\by) d\by\Big)^{1/p},
\end{equation}
where $D_{R'}-\by$ means a translation in $\by$. It remains to prove that 
\begin{equation}\label{201231e5_65}
    \Big( \int_{\R^n} \sum_{\theta} \|H^{\lambda', R''} g_{\theta}\|^p_{L^p(D_{R'}-\by)} \zeta_{(R'')^{1-\delta}}(\by) d\by\Big)^{1/p} \lesim (R'')^{-(n-1)/p'+n/p} \|g\|_p.
\end{equation}
This is the step where it is convenient to have a cut-off function in space and time in the definition of $H^{\lambda', R''}f$. We will show that  
\begin{equation}\label{201231e5_66}
    \|H^{\lambda', R''} g_{\theta}\|_{L^p(D_{R'}-\by)} \lesim (R'')^{-(n-1)/p'+n/p} \|g_{\theta}\|_p,
\end{equation}
which immediately implies \eqref{201231e5_65}. To prove this bound, we use interpolation at $p=2$ and $p=\infty$. At $p=2$, we need the following lemma. 
\begin{lem}\label{201231lem4_7}
For every $1\le R\le \lambda$, it holds that 
\begin{equation}
\label{Hormander-type}
    \|H^{\lambda, R}g(x, t)\|_{L^2_x} \lesim \|g\|_2,
\end{equation}
uniformly in $t\in \R$. 
\end{lem}
\begin{proof}[Proof of Lemma \ref{201231lem4_7}]
This estimate is well-known. As it is short, we still include it here. By the main theorem in \cite{MR340924}, which is proven via a $TT^*$ argument and Schur's text, the desired estimate follows once we verify that 
\begin{equation}
    |\det \nabla_{x}\nabla_{\omega}\phi(x, t; \omega)|\simeq 1,
\end{equation}
uniformly in $|t|\lesim 1, |x|\lesim 1$. However this has been verified and used in \eqref{210201e3_40}.  
\end{proof}
As a consequence of Lemma \ref{201231lem4_7}, we obtain
\begin{equation}
    \|H^{\lambda', R''} g_{\theta}\|_{L^{2}(D_{R'}-\by)}\lesim  R^{-(n-1)(1/2-1/p)+1/2}\|g_{\theta}\|_p.
\end{equation}
At $p=\infty$, we have the trivial bound that 
\begin{equation}
    \|H^{\lambda', R''} g_{\theta}\|_{L^{\infty}(D_{R'}-\by)}\le R^{-(n-1)/p'}\|g_{\theta}\|_p.
\end{equation}
Interpolation implies the desired \eqref{201231e5_66}. This finishes the proof of Lemma \ref{201229lem5_3}. 
\end{proof}

\section{Wave packets and their essential supports}\label{f_section4}

The rest of the paper is devoted to the proof of Theorem \ref{201204thm5_1}, the broad norm estimate. When we were handling the narrow part in the previous section, we used an inductive argument, and in order to close induction we had to be very careful with the choice of various amplitude functions $\bar{a}_1, \bar{a}_2$, and regions of taking integrations; we also had to introduce $H^{\lambda, R}g(\bx)$ and distinguish it from $H^{\lambda} g(\bx)$ as it was very often inconvenient to carry amplitude functions in our operator. From now on, these points are not as crucial as before, and we will always work with $H^{\lambda, R}g(\bx)$, a function that is compactly supported in $[-4C_n \lambda, 4C_n \lambda]\times [R/(2C_n), 2C_n \lambda]$. Moreover, to simplify notation, we always abbreviate $H^{\lambda, R}$ to $H^{\lambda}$: The dependence on $R$ will be emphasized through the region we evaluate the $L^p$ norm of $H^{\lambda}$.

\subsection{Wave packet decomposition}
Let $r \geq 1$ and take a collection $\Theta_r$ of dyadic cubes of side length  $\frac{9}{11}r^{-1/2}$ covering the ball $B^{n-1}(0,2)$. We take a smooth partition of unity $(\psi_{\theta})_{\theta \in \Theta_r}$ with $\supp \psi_{\theta}\subset \frac{11}{10}\theta$ for the ball $B^{n-1}(0,2)$ such that
\begin{equation}
    \| \partial_{w}^{\alpha}\psi_{\theta}\|_{L^{\infty}} \lesssim_{\alpha} r^{|\alpha|/2}
\end{equation}
for any $\alpha \in \N_0^{n-1}$. We denote by $\omega_{\theta}$ the center of $\theta$. Given a function $g$, we perform a Fourier series decomposition to the function $g \psi_{\theta}$ on the region $\frac{11}{9}\theta$ and obtain
\begin{equation}
    g(w)\psi_{\theta}(w)\cdot \mathbbm{1}_{\frac{11}{10}\theta}(\omega)=\Big(\frac{  r^{1/2}}{2\pi} \Big)^{n-1}
    \sum_{v \in  r^{1/2}\Z^{n-1}}(g \psi_{\theta})^{\wedge}(v)e^{2\pi iv \cdot w }\mathbbm{1}_{\frac{11}{10}\theta}(\omega).
\end{equation}
Let $\widetilde{\psi}_{\theta}$ be a non-negative smooth cutoff function supported on $\frac{11}{9}\theta$ and equal to 1 on $\frac{11}{10}\theta$. We can therefore write 
\begin{equation}
    g(w)\psi_{\theta}(w)\cdot \widetilde{\psi}_{\theta}(\omega)=\Big(\frac{  r^{1/2}}{2\pi} \Big)^{n-1}
    \sum_{v \in  r^{1/2}\Z^{n-1}}(g \psi_{\theta})^{\wedge}(v)e^{2\pi iv \cdot w }\widetilde{\psi}_{\theta}(\omega).
\end{equation}
If we also define
\begin{equation}
    g_{\theta,v}(w):=
    \Big(\frac{r^{1/2}}{2\pi} \Big)^{n-1}
    (g \psi_{\theta})^{\wedge}(v)e^{2\pi iv \cdot \omega }\widetilde{\psi}_{\theta}(\om).
\end{equation}
then we have 
\begin{equation}
    g=\sum_{(\theta,v) \in \Theta_r \times r^{1/2}\Z^{n-1} }g_{\theta,v}. 
\end{equation}
This finishes our wave packet decomposition. 
\medskip

Let
$1\le r\le R$. We fix a ball $B(\bx_0,r)$ of radius $r$ satisfying 
\begin{equation}\label{210103e5_6}
    B(\bx_0,r) \subset [-3C_n \lambda, 3 C_n \lambda]^{n-1} \times [R/C_n, \lambda C_n].
\end{equation} 
Let us fix a function $f$. We define a collection of tubes associated to the ball $B(\bx_0,r)$ by
\begin{equation}\label{4747}
    \T[B(\bx_0,r)]:=\{ T_{\theta,v}(\bx_0): (\theta,v) \in \Theta_r \times r^{1/2}\Z^{n-1}   \},
\end{equation}
where $T_{\theta,v}(\bx_0)$ is some set that will be determined later. Given $T_{\theta,v}(\bx_0) \in \T[B(\bx_0,r)]$, we define a function
\begin{equation}\label{modulations}
    g_{T_{\theta,v}(\bx_0) }(\om):= e^{-2\pi i\phi^{\lambda}(\bx_0; \om) }(  g(\cdot)e^{2\pi i\phi^{\lambda}(\bx_0; \ \cdot) } )_{\theta,v}(\om).
\end{equation}
Notice that
\begin{equation}\label{4949}
    H^{\lambda}g(x,t)=\sum_{T \in \T[B(\bx_0,r)] }H^{\lambda}g_T(x,t).
\end{equation}
In the rest of the section, we will define tubes $T$ so that $H^{\lambda}g_T$ is essentially supported on $T$.

\begin{remark}
{\rm
It is worth mentioning the motivation and advantages of the notation $g_{T_{\theta,v}}$. In \cite{guth2018} and \cite{MR4047925}, the following notations are introduced:
 \begin{equation}
     \widetilde{\bx}=\bx-\bx_0 \;\;\; \text{and} \;\;\; \tilde{g}(w)=g(w)e^{2\pi i \phi^{\lambda}(\bx_0;w)}.
 \end{equation}
 In these notations, we have
 \begin{equation}
     H^{\lambda}g(x)=\tilde{H}^{\lambda}\tilde{g}(\widetilde{\bx}),
 \end{equation}
 where
 \begin{equation}
  \tilde{H}^{\lambda}\tilde{g}(\tilde{\bx}) :=  \int_{\R^{n-1}} e^{2\pi i \phi^{\lambda}(\tilde{\bx}+\bx_0;  \omega)} 
  e^{-2\pi i \phi^{\lambda}(\bx_0;w)}
  \tilde{g}(\omega)d\omega.   
 \end{equation}
  The motivation of their notations is that the function $\tilde{H}^{\lambda}\tilde{g}$ is defined on the ball centered at the origin so that we can apply the wave packet decomposition to the function $\tilde{g}$. However, in our proof of Theorem \ref{201204thm5_1} (in particular, Section 9 and 10), these notations cause some confusion because the function $\tilde{g}$ depends on the point $\bx_0$, but the notation $\tilde{g}$ does not show this dependence. 
 
 To make things clear, what we have done above is to apply the wave packet decomposition to $\tilde{g}$ and change back the variable $\widetilde{\bx}$ to the original variable $\bx$ (See \eqref{modulations} and \eqref{4949}). Note that our tube $T_{\theta,v}(\bx_0)$ indicates the dependence on the point $\bx_0$. }
\end{remark}

\begin{lem}[$L^2$-orthogonality]\label{210103lem5_1}
For any set $\T \subset \T[B(\bx_0,r)]$, it holds that
\begin{equation}
    \|\sum_{T \in \T} g_{T}\|_2^2 \lesim
    \sum_{T \in \T}
    \| g_{T}\|_2^2\lesim \|g\|_2^2.
\end{equation}
Moreover, if $\ZT$ is any collection of tubes with the same $\theta$, then 
    \begin{equation}\label{210112e4_11}
    \|\sum_{T \in \T} g_{T}\|_2^2 \sim
    \sum_{T \in \T}
    \| g_{T}\|_2^2.
\end{equation}
\end{lem}
Lemma \ref{210103lem5_1} follows immediately from Plancherel's theorem, and the proof is omitted. 
\begin{remark}
\rm

In the above lemma, without the extra assumption, \eqref{210112e4_11} may not be correct, due to the frequency overlapping. 
\end{remark}

The definition of $T_{\theta, v}(\bx_0)$ relies on which of the following two cases we are in:
\begin{equation}\label{2011227_8z}
    |v-\nabla_\omega\phi^{\lambda}(\bx_0;\omega_{\theta})|\ge \frac{10 C_n\lambda}{\sqrt{1+(10C_n)^2}} \text{ or } |v-\nabla_\omega\phi^{\lambda}(\bx_0;\omega_{\theta})|\le \frac{10 C_n \lambda}{\sqrt{1+(10 C_n)^2}}.
\end{equation}
In the former case, we define $T_{\theta,v}$ to be an empty set, and show that $H^{\lambda}f_T$ always  decays rapidly, that is, it is bounded by  $\lambda^{-N}\|f\|_2$ pointwise. In the latter case, we define $T_{\theta,v}$ to be as follows: Define a line $l_{\theta,v}$ as
\begin{equation}
\label{coreline}
    l_{\theta,v}:=\big\{(x,t):\omega_\theta t-x=\frac{\la (\nabla_\omega\phi^{\lambda}(\bx_0;\omega_{\theta})-v)}{\sqrt{\la^2-|\nabla_\omega\phi^{\lambda}(\bx_0;\omega_{\theta})-v|^2}}\big\}
\end{equation}
and a $r^{1/2+\de}$ tube $T_{\theta,v}$ as $T_{\theta,v}=\mc{N}_{r^{1/2+\de}}(l_{\theta,v})$. Here $0<\delta\ll \epsilon$ is a small parameter whose choice will become clear later, and $\mc{N}_{r^{1/2}+\delta}$ means the $r^{1/2+\delta}$-neighbourhood. Next, we show that the essential support of $H^{\lambda} f_T$ is $T$. The precise meaning of essential supports will be made clear in the statements of the following two lemmas. 
\begin{lem}\label{210103lem5_2}
If $T=T_{\theta, v}(\bx_0)$ is such that the former case in \eqref{2011227_8z} holds, then we have 
\begin{equation}
    \|H^{\lambda} g_T\|_{L^{\infty}(B(\bx_0, r))} \lesim_N \lambda^{-N} \|g\|_2,
\end{equation}
for  every $N\in \N$, whenever $B(\bx_0, r)$ satisfies \eqref{210103e5_6}. 
\end{lem}
\begin{proof}[Proof of Lemma \ref{210103lem5_2}]
For given $\bx_0$, we denote $g_{\bx_0}(\omega)=g(\omega) e^{2\pi i\phi^{\lambda}(\bx_0; \omega)}$. Recall the definition of $g_T$ that 
\begin{equation}
\label{single-wp-tube}
    g_T(\omega)=e^{-2\pi i\phi^{\lambda}(\bx_0; \omega)} (g_{\bx_0}\psi_{\theta})^{\wedge}(v) \Big(\frac{r^{1/2}}{2\pi}\Big)^{n-1} e^{2\pi iv\cdot \omega}\widetilde{\psi}_{\theta}(\omega). 
\end{equation}
Let $H^{\lambda}$ act on $f_T$ and we obtain 
\begin{equation}
\label{H-lambda-g}
    H^{\lambda} g_T(x, t)=\Big(\frac{r^{1/2}}{2\pi}\Big)^{n-1} (g_{\bx_0}\psi_{\theta})^{\wedge}(v)a_{\lambda, R}(x, t) \int e^{2\pi i\phi_{\bx_0}^{\lambda}(\bx; \omega)+2\pi iv\cdot \omega} \widetilde{\psi}_{\theta}(\omega)d\omega,
\end{equation}
where  
\begin{equation}\label{210103e5_19}
    \phi_{\bx_0}^{\lambda}(\bx; \omega):=\phi^{\lambda}(\bx; \omega)-\phi^{\lambda}(\bx_0; \omega).
\end{equation}
We make the change of variable $\omega\to\omega+\om_\theta$ and 
consider the oscillatory integral 
\begin{equation}
\begin{split}
    & \int e^{2\pi i \big(\phi^{\lambda}_{\bx_0}(x,t;\om+\om_\theta)+v\cdot\omega\big)} \widetilde{\psi}(r^{1/2} \omega) d\omega\\
    & = r^{-(n-1)/2}\int e^{2\pi i \big(\phi^{\lambda}_{\bx_0}(x,t;r^{-1/2}\om+\om_\theta)+r^{-1/2}v\cdot\omega\big)} \widetilde{\psi}(\omega) d\omega.
\end{split}
\end{equation}
Denote $\omega':=r^{-1/2}\omega$. First of all, $\nabla_{\omega}$ of the phase function equals 
\begin{equation}\label{2011227_13}
    r^{-1/2}\Big(\nabla_{\omega} \phi^{\lambda}_{\bx_0}(x,t;\om'+\om_\theta)+v\Big).
\end{equation}
We write 
\begin{equation}\label{2011227_14}
\begin{split}
    & \eqref{2011227_13} = r^{-1/2}\Big(\nabla_{\omega} \phi^{\lambda}_{\bx_0}(x,t;\om'+\om_\theta)-\nabla_{\omega} \phi^{\lambda}_{\bx_0}(x,t;\om_\theta)\Big)\\
    & + r^{-1/2} \Big(\nabla_{\omega} \phi^{\lambda}_{\bx_0}(x,t;\om_\theta)+v\Big).
\end{split}
\end{equation}
We first look at the second term. By the definition of $\phi^{\lambda}_{\bx_0}$, it equals to 
\begin{equation}\label{201122e7_15}
    \begin{split}
        & r^{-1/2}\Big(\nabla_{\omega} \phi^{\lambda}(x, t; \omega_{\theta})-\nabla_{\omega} \phi^{\lambda}(x_0, t_0; \omega_{\theta})+v\Big).
    \end{split}
\end{equation}
Recall we are in the case that 
\begin{equation}
    |\nabla_{\omega} \phi^{\lambda}(x_0, t_0; \omega_{\theta})-v|\ge \frac{10 C_n \lambda}{\sqrt{1+(10C_n)^2}}.
\end{equation}
Moreover, recall that $|t|\le 2C_n \lambda$ and $|x|\le 4C_n \lambda, $ and therefore we obtain 
\begin{equation}
\begin{split}
    & |\nabla_{\omega} \phi^{\lambda}(x, t; \omega_{\theta})|=\lambda\Big|\frac{t\om_{\theta}-x}{\sqrt{\la^2+|x-t\om_{\theta}|^2}}\Big| \le \frac{6C_n\lambda}{\sqrt{1+(6 C_n)^2}}
\end{split}
\end{equation}
By the triangle inequality, we obtain that 
\begin{equation}
    \eqref{201122e7_15}\gtrsim \Big(\frac{10 C_n \lambda}{\sqrt{1+(10C_n)^2}}-\frac{6 C_n \lambda}{\sqrt{1+(6C_n)^2}}\Big)r^{-1/2} \gtrsim \lambda/ r^{1/2}. 
\end{equation}
Let us continue to compute the first term in \eqref{2011227_14}. 
First,
\begin{align}\label{210201e4_25}
    & \nabla_\om\phi^{\lambda}_{\bx_0}(x,t;\om)=\la\Big(\frac{t\om-x}{\sqrt{\la^2+|x-t\om|^2}}-\frac{t_0\om-x_0}{\sqrt{\la^2+|x_0-t_0\om|^2}}\Big).
\end{align}
Therefore 
\begin{equation}\label{210201e4_26}
    \begin{split}
        & \nabla_\om\phi^{\lambda}_{\bx_0}(x,t;\om'+\om_\theta)-\nabla_\om\phi^{\lambda}_{\bx_0}(x,t;\om_\theta)\\
        & = \la\Big(\frac{t(\omega'+\om_\theta)-x}{\sqrt{\la^2+|x-t(\om'+\omega_{\theta})|^2}}-\frac{t_0(\om'+\omega_{\theta})-x_0}{\sqrt{\la^2+|x_0-t_0(\om'+\omega_{\theta})|^2}}\Big)\\
        &-\la\Big(\frac{t\om_\theta-x}{\sqrt{\la^2+|x-t\om_{\theta}|^2}}-\frac{t_0\om_{\theta}-x_0}{\sqrt{\la^2+|x_0-t_0\om_{\theta}|^2}}\Big)
    \end{split}
\end{equation}
We use the bound \eqref{210131e3_28} and mean value theorems, and see that the above expression can be bounded from above by $r^{1/2}$, which, together with the leading coefficient $r^{-1/2}$, produces the bound $1$ for the first term. Putting everything together, we see that 
\begin{equation}
    \eqref{2011227_14}\gtrsim \lambda/r^{1/2}-1\gtrsim \lambda/r^{1/2}. \end{equation}
The desired rapid decay follows immediately from integration by parts.
\end{proof}

\begin{lem}\label{210103lem5_3}
For $T=T_{\theta, v}$ and $(x,t)\in B(\bx_0, r)\setminus T_{\theta,v}$ with $B(\bx_0, r)$ satisfying \eqref{210103e5_6}, it holds that 
\begin{equation}
\label{wave-packet-decay}
    | H^\la  g_T(x,t)|\leq (1+r^{-1/2}|\nabla_\om\phi^{\lambda}_{\bx_0}(x,t;\om_\theta)+v|)^{-N}{\rm RapDec}(r)\|g\|_2.
\end{equation}
Here $\phi^{\lambda}_{\bx_0}$ is defined as in \eqref{210103e5_19}. 
\end{lem}
This lemma describes the essential support of a wave packet, and it is the right hand side of \eqref{wave-packet-decay} that motivates the definition of the core line in \eqref{coreline}; see \eqref{coreline-2} below that connects the expression of the core line in \eqref{coreline} with the right hand side of \eqref{wave-packet-decay}. 
\begin{proof}[Proof of Lemma \ref{210103lem5_3}]
Note that each piece $H^\la  g_{T}$ has the expression
\begin{equation}
    H^{\lambda} g_T(x, t)=\Big(\frac{r^{1/2}}{2\pi}\Big)^{n-1} (g_{\bx_0}\psi_{\theta})^{\wedge}(v) a_{\lambda, R}(x, t) \int e^{i2\pi \big(\phi_{\bx_0}^{\lambda}(\bx; \omega)+v\cdot \omega\big)} \widetilde{\psi}_{\theta}(\omega)d\omega,
\end{equation}
We will see that the proof is quite similar to that of Lemma \ref{210103lem5_2}. We make the change of variable $\omega\to\omega+\om_\theta$ and 
consider the oscillatory integral 
\begin{equation}
\begin{split}
    & \int e^{i2\pi \big( \phi^{\lambda}_{\bx_0}(x,t;\om+\om_\theta)+v\cdot\omega\big)} \widetilde{\psi}(r^{1/2} \omega) d\omega\\
    & = r^{-(n-1)/2}\int e^{i2\pi \big(\phi^{\lambda}_{\bx_0}(x,t;r^{-1/2}\om+\om_\theta)+r^{-1/2}v\cdot\omega\big)} \widetilde{\psi}(\omega) d\omega.
\end{split}
\end{equation}
Let us compute the derivative of the phase function. We write the derivative as in \eqref{2011227_14}. There are two terms there. The first term contributes $1$, which has been shown in \eqref{210201e4_25} and \eqref{210201e4_26}. We therefore have 
\begin{equation}
    |\eqref{2011227_14}|\gtrsim r^{-1/2} \big|\nabla_{\omega} \phi^{\lambda}_{\bx_0}(x, t; \omega_{\theta})+v \big|-1. 
\end{equation}
This  gives us the desired rapid decay in the complement of the set
\begin{equation}\label{201122e7_29}
    E_{\theta,v}=\{(x,t):|\nabla_\om\phi^{\lambda}_{\bx_0}(x,t;\om_\theta)+v|\lesim r^{(1+\de)/2}\}.
\end{equation}
It remains to show $E_{\theta,v}\subset T_{\theta,v}$. We solve the equation  $\nabla_\om\phi^{\lambda}_{\bx_0}(x,t;\om_\theta)+v=0$ for the variables $(x,t)$ to have
\begin{equation}
\label{coreline-2}
    \frac{\la\big(t\omega_\theta-x\big)}{\sqrt{\la^2+|x-t\om_\theta|^2}}=-v+\frac{\la(t_0\omega_\theta-x_0)}{\sqrt{\la^2+|x_0-t_0\om_\theta|^2}}=-v+\nabla_{\omega}\phi^{\lambda}(\bx_0; \omega_{\theta}),
\end{equation}
which indeed is the straight line $l_{\theta,v}$. Moreover, we can solve \eqref{201122e7_29} directly, and see that it indeed lies inside our tube. Another way of seeing it is to notice that the Jacobian of $\nabla_{\omega}\phi^{\lambda}_{\bx_0}(x, t; \omega_{\theta})$ in $x$ is comparable to 1, as has been verified in \eqref{210201e3_40}.
\end{proof}

\subsection{$L^2$ properties for wave packets}

In this subsection, we prove two lemmas about $L^2$ properties of wave packets. One of these two lemmas, Lemma \ref{210103lem6_1}, will not be directly used in our proof; however we would still like to include it here and show that wave packets in our setting have similar properties to those in the Fourier restriction problem. 

\begin{lem}\label{210103lem6_1}
Let $B(\bx_0, r)$ satisfy \eqref{210103e5_6}. Given a tube $T_1=T_{\theta_1,v_1}\in \T[B(\bx_0, r)]$, all but $r^{O(\de)}$ many tubes $T_2=T_{\theta_2,v_2}\in \T[B(\bx_0, r)]$ satisfy
\begin{equation}
\label{l2-orthogonality-1}
    \int_{|t-t_0|\leq r}\int H^\la g_{T_1}(x,t)\overline{H^\la g_{T_2}}(x,t)dxdt={\rm{RapDec}}(r)\|g\|_2^2.
\end{equation}
\end{lem}
\begin{proof}[Proof of Lemma \ref{210103lem6_1}]
Note that each piece $T^\la  g_{T}$ has the expression
\begin{equation}
    H^{\lambda} g_T(x, t)=\Big(\frac{r^{1/2}}{2\pi}\Big)^{n-1} (g_{\bx_0}\psi_{\theta})^{\wedge}(v) a_{\lambda, R}(x,t) \int e^{i2\pi \big(\phi_{\bx_0}^{\lambda}(\bx; \omega)+v\cdot \omega\big)} \widetilde{\psi}_{\theta}(\omega)d\omega,
\end{equation}
To estimate \eqref{l2-orthogonality-1}, it suffices to consider the oscillatory integral
\begin{equation}
    \int_{|t-t_0|\le r} \int e^{2\pi i
    \big(\phi^{\lambda}_{\bx_0}(x,t;\om_1)-\phi^{\lambda}_{\bx_0}(x,t;\om_2)\big)}a_{\lambda, R}(x,t)dx dt,
\end{equation}
with $|\omega_1-\omega_{\theta_1}|\le r^{-1/2}$ and $|\omega_2-\omega_{\theta_2}|\le r^{-1/2}$, and show that it decays rapidly. As we will apply the argument of integration by parts in $x$, we do the change of variable $x\mapsto \lambda x$ and obtain 
\begin{equation}
    \lambda^{n-1} \int_{|t-t_0|\le r}\int e^{2\pi i\lambda\big(\phi(x, t'; \omega_1)-\phi(x'_0, t'_0; \omega_1)-\phi(x, t'; \omega_2)+\phi(x'_0, t'_0; \omega_2)\big)}a_{\lambda, R}(\lambda x, t)dx dt,
\end{equation}
where we have also set $t':=t/\lambda$, $x'_0:=x_0/\lambda$ and $t'_0:=t_0/\lambda$.

Let us first assume that we are in the case $|\omega_1-\omega_2|\ge r^{-1/2+\delta}$. We compute the partial derivative of the phase function in the $x$ variable
and obtain 
\begin{equation}\label{210202e4_38}
    \begin{split}
        & \frac{\lambda}{t'}\frac{x-t'\omega_1}{\sqrt{1+|x-t'\omega_1|^2}} -\frac{\lambda}{t'}\frac{x-t'\omega_2}{\sqrt{1+|x-t'\omega_2|^2}}.
    \end{split}
\end{equation}
Recall that the mixed Hessian of $\phi(x, t; \omega)$ in $x$ and $\omega$ is non-vanishing, that is, 
\begin{equation}
    |\det \nabla_{\omega}\nabla_{x}\phi(x, t; \omega)|\simeq 1,
\end{equation}
as has been verified in \eqref{210201e3_40}. This, combined with mean value theorems, implies that 
\begin{equation}
    |\eqref{210202e4_38}|\gtrsim \lambda |\omega_1-\omega_2|.
\end{equation}
Now we can apply integration by parts and finish the proof of the case $|\omega_1-\omega_2|\ge r^{-1/2+\delta}$. \\

In the end, we consider the case $|\omega_1-\omega_2|\le r^{-1/2+\delta}$ and $|v_1-v_2|\ge r^{1/2+\delta}$. However, this case follows immediately from the pointwise estimate in Lemma \ref{210103lem5_3}.
\end{proof}


\section{Comparing wave packets at different scales}\label{210325section5}
In the previous section, we built up a wave packet decomposition for a function $g$ on the ball $B(\bx_0,r)$.   Since our proof of Theorem \ref{201204thm5_1} relies on the multiscale argument as in \cite{guth2018} and \cite{MR4047925}, it is necessary to compare wave packets at two different scales. In this section, we will prove lemmas in Section 9 of \cite{MR4047925} in our setting. It is worth mentioning that in their paper they consider more general operators. However, their lemmas are valid only when the operators are of ``reduced form''. Since we do not reduce our operator to reduced form, we are unable to simply cite their lemmas. Instead, we follow their arguments and prove the lemmas in our setting. We include the details for the completeness of the paper.
\medskip

Consider another ball $B(\widetilde{\bx}_0,\rho)\subset B(\bx_0,r)$ for some $\widetilde{\bx}_0=(\widetilde{x}_0,\widetilde{t}_0)\in\ZR^n$ and $r^{1/2} <\rho<r$. 
Since we are considering wave packets at different scales, it would be convenient to introduce some notations to distinguish wave packets from different scales. We will use the notation $\widetilde{\T}[B(\widetilde{\bx}_0,\rho)]$ for the collection ${\T}[B(\widetilde{\bx}_0,\rho)]$ in \eqref{4747}. We will denote the elements of $\widetilde{\T}[B(\widetilde{\bx}_0,\rho)]$ by $\widetilde{T}_{\widetilde{\theta},\widetilde{v}}$. Here, the symbol $\sim$ indicates that the objects  are generated at a  smaller scale.
\medskip




We find it appropriate to introduce three more definitions here. 
\begin{defi}
We say a function $h$ is concentrated on wave packets from a tube set $\ZT_\al$, if
\begin{equation}
    h=\sum_{T\in\ZT_\al} h_T+\rapid(r)\|h\|_2.
\end{equation}
\end{defi}
\begin{defi}
For a ball $B$ and $\W \subset \T[B]$, we define
\begin{equation}
\label{h_W}
    h|_{\W}:=\sum_{T \in \W}h_T.
\end{equation}
\end{defi}
\begin{defi}
Let $({\theta},v) \in \Theta_r \times r^{1/2}\Z^{n-1}$ and let $(\tilde{\theta},\wt v) \in \Theta_\rho \times \rho^{1/2}\Z^{n-1}$. we define a collection of smaller tubes $\wt T_{\ti\theta,\wt v}$ that are ``close to" the bigger tube $T_{\theta,v}$ as
\begin{equation}\label{smalltubesinlargetube}
\widetilde{\ZT}_{\theta,v}[B(\widetilde{\bx}_0,\rho)]:=
    \big\{\widetilde{T}_{\widetilde{\theta},\widetilde{v}} \in 
    \widetilde{\T}[B(\widetilde{\bx}_0,\rho)]
    :{\rm{dist}}(\theta,\widetilde{\theta})\lesssim \rho^{-1/2}, \; |\tilde{v}-(\nabla_\om\phi_{\bx_0}^{\lambda}(\widetilde{\bx}_0;\omega_{\theta})+v)|\lesim{r^{(1+\de)/2}}\big\}.
\end{equation}
We sometimes abbreviate the collection $\widetilde{\ZT}_{\theta,v}[B(\widetilde{\bx}_0,\rho)]$ to $\widetilde{\T}_{\theta,v}$ for simplicity.
\end{defi}

Notice that we can write
\begin{equation}\label{210314e5_3}
    \tilde{v}-(\nabla_{\omega}\phi^{\lambda}_{\bx_0}(\widetilde{\bx}_0; \omega_{\theta})+v)=\big(\nabla_{\omega}\phi^{\lambda}(\widetilde{\bx}_0; \omega_{\theta})+\tilde{v}\big)-\big(\nabla_{\omega}\phi^{\lambda}(\bx_0; \omega_{\theta})+v\big).
\end{equation}
Heuristically, on the ball $B(\widetilde{\bx}_0,\rho)$, only those new wave packets concentrated in $\widetilde{\ZT}_{\theta,v}[B(\widetilde{\bx}_0,\rho)]$ would make significant contribution to our old wave packet $g_{T_{\theta,v}}$. This will be stated rigorously in the following lemma.

\begin{lem}\label{1861}
Let $T_{\theta,v} \in \T[B(\bx_0,r)]$. Then it holds that 
\begin{equation}
    g_{T_{\theta,v}}=(g_{T_{\theta,v}})|_{\widetilde{\ZT}_{\theta,v}[B(\widetilde{\bx}_0,\rho)]}+{\rm{RapDec}}(r)\|g\|_2.
\end{equation}
\end{lem}

\begin{proof}

By the definition of wave packets, 
\begin{equation}
\begin{split}
    (g_{T_{\theta,v}(\bx_0)})_{\widetilde{T}_{\tilde{\theta},\tilde{v}}(\widetilde{\bx}_0)}(\omega)&=e^{-2\pi i\phi^{\lambda}(\widetilde{\bx}_0; \omega) }(  g_{T_{\theta,v}}(\cdot)e^{2\pi i\phi^{\lambda}(\widetilde{\bx}_0; \, \cdot\,) } )_{\tilde{\theta},\tilde{v}}(\omega)
    \\&=e^{-2\pi i\phi^{\lambda}(\widetilde{\bx}_0; \omega) }\Big( 
    e^{2\pi i(\phi^{\lambda}(\widetilde{\bx}_0; \, \cdot\,)-\phi^{\lambda}({\bx}_0; \, \cdot\,) ) }
    (g(\cdot)e^{2\pi i\phi^{\lambda}(\bx_0;\,\cdot \,) } )_{{\theta,v}}(\cdot) \Big)_{\tilde{\theta},\tilde{v}}(\omega). 
\end{split}
\end{equation}
Since a function $g_{\theta,v}$ is supported near $\theta$ for every function $g$, by the above expression, we see that 
\begin{equation}
    (g_{T_{\theta,v}(\bx_0)})_{\widetilde{T}_{\tilde{\theta},\tilde{v}}(\widetilde{\bx}_0)} \equiv 0 \;\; \text{unless} \;\; \text{dist}(\theta,\widetilde{\theta})\lesssim \rho^{-1/2}.
\end{equation}
By renaming the function $g(\omega) e^{2\pi i \phi^{\lambda}(\bx_0,\omega)}$ by $g(\omega)$,
it remains to show that
\begin{equation}
    \big( 
    e^{2\pi i\phi^{\lambda}_{\bx_0}(\widetilde{\bx}_0; \, \cdot\,) }
    g_{{\theta,v}}(\cdot) \big)_{\tilde{\theta},\tilde{v}}=\mathrm{RapDec}(r)\|g\|_2,
\end{equation}
whenever 
$
    |-v+\tilde{v}-\nabla_\om\phi_{\bx_0}^{\lambda}(\widetilde{\bx}_0;\omega_{\theta})|\gtrsim {r^{(1+\de)/2}}.$
By the definition, it amounts to proving
\begin{equation}
\big(e^{2\pi i\phi_{\bx_0}^{\lambda}(\widetilde{\bx}_0; \, \cdot\,) }g_{\theta,v}(\cdot) \psi_{\tilde{\theta}}(\cdot)\big)^{\wedge}(\tilde{v})
=\mathrm{RapDec}(r)\|g\|_2.
\end{equation}
The left hand side can be written as
\begin{equation}
    \big(\widehat{\psi}_{\tilde{\theta}}*\big(
    e^{2\pi i\phi_{\bx_0}^{\lambda}(\widetilde{\bx}_0; \, \cdot\,) }g_{\theta,v}(\cdot) \big)^{\wedge}\big)(\tilde{v}).
\end{equation}
Since 
 the function $\widehat{\psi}_{\tilde{\theta}}$ is essentially supported in $B(0,\rho^{1/2})$, it suffices to show that
\begin{equation}\label{21.01.06.42}
\begin{split}
    \big(
    e^{2\pi i\phi_{\bx_0}^{\lambda}(\widetilde{\bx}_0; \, \cdot\,) }g_{\theta,v}(\cdot) \big)^{\wedge}(z)
    =(1+r^{-1/2}|z-v-\nabla_\omega\phi_{\bx_0}^{\lambda}(\wt\bx_0;\omega_{\theta})|)^{-(n+1)} \mathrm{RapDec}(r)\|g\|_2
\end{split}
\end{equation}
for every $|z-v-\nabla_\omega\phi_{\bx_0}^{\lambda}(\wt\bx_0;\omega_{\theta})| \gtrsim r^{(1+\delta)/2}.$
We take a function $\widetilde{\widetilde{\psi}}_{\theta}(\omega)=\widetilde{\widetilde{\psi}}(r^{1/2}(\omega-\omega_{\theta}))$ for some compactly supported function $\widetilde{\widetilde{\psi}}$ so that this function is adapted to $\theta$ but is equal to one on the support of $\widetilde{\psi}_{\theta}$. Since $g_{\theta,v}$ is supported on the support of $\widetilde{\psi}_{\theta}$, the left hand side of \eqref{21.01.06.42} can be written as
\begin{equation}
    \big(
    e^{2\pi i\phi_{\bx_0}^{\lambda}(\widetilde{\bx}_0; \, \cdot\,) }\widetilde{\widetilde{\psi}}_{\theta}(\cdot)g_{\theta,v}(\cdot) \big)^{\wedge}(z)=
    \big(
    e^{2\pi i\phi_{\bx_0}^{\lambda}(\widetilde{\bx}_0; \, \cdot\,) }\widetilde{\widetilde{\psi}}_{\theta}(\cdot)\big)^{\wedge}
    *\widehat{g_{\theta,v}}(z).
\end{equation}
Since the function $\widehat{g_{\theta,v}}$ is concentrated on $B(v,r^{(1+\delta)/2})$, by the above expression, the claim \eqref{21.01.06.42} is reduced to
\begin{equation}\label{21.01.06.610}
\begin{split}
    \int_{\R^n} 
    e^{2\pi i(-z \cdot \omega+
    \phi_{\bx_0}^{\lambda}(\widetilde{\bx}_0; \omega) ) }\widetilde{\widetilde{\psi}}_{\theta}(\omega)\,d\omega
    =\big(1+r^{-1/2}|z-\nabla_\om\phi_{\bx_0}^{\lambda}(\widetilde{\bx}_0;\omega_{\theta})|\big)^{-(n+1)}\mathrm{RapDec}(r)
\end{split}
\end{equation}
whenever $|z-\nabla_\om\phi_{\bx_0}^{\lambda}(\widetilde{\bx}_0;\omega_{\theta})| \gtrsim r^{(1+\delta)/2}$.
We apply the change of variables: $\omega \mapsto r^{-1/2}\omega+\omega_{\theta}$ and the above integral becomes
\begin{equation}
    r^{-(n-1)/2}e^{-2\pi i z \cdot \omega_{\theta}}
    \int_{\R^n} 
    e^{2\pi i(-r^{-1/2}z \cdot \omega+
    \phi_{\bx_0}^{\lambda}(\widetilde{\bx}_0; \omega_{\theta}+r^{-1/2}\omega)) }\widetilde{\widetilde{\psi}}(\omega)\,d\omega.
\end{equation}
By the stationary phase method, the estimate \eqref{21.01.06.610} follows from
\begin{equation}
    |z -\nabla_{\omega}
    \phi_{\bx_0}^{\lambda}(\widetilde{\bx}_0; \omega_{\theta}+r^{-1/2}\omega) | \gtrsim r^{(1+\delta)/2}
\end{equation}
for every $|z-\nabla_\om\phi_{\bx_0}^{\lambda}(\widetilde{\bx}_0;\omega_{\theta})| \gtrsim r^{(1+\delta)/2}$. By using a mean value theorem with $|\bx_0-\widetilde{\bx}_0| \lesssim r$, this follows from
\begin{equation}
    |\nabla^2_{\omega} \nabla_{\bx}\phi^{\lambda}(x,t;\omega)| \lesssim 1
\end{equation}
for every $|x|+|t| \lesssim \lambda$, which has been verified in \eqref{210131e3_28}.
\end{proof}

\begin{lem}\label{0128lemma}
Suppose that $T_{\theta,v}\in \T[B(\bx_0,r)]$.
If $\widetilde{T}_{\tilde{\theta},\tilde{v}} \in \widetilde{\T}_{\theta,v}[B(\widetilde{\bx}_0,\rho)]$, then it holds that
\begin{equation}\label{210312e5_16}
    \mathrm{HausDist}(\widetilde{T}_{\tilde{\theta},\tilde{v}},T_{\theta,v} \cap B(\widetilde{\bx}_0,\rho ))\lesssim r^{1/2+\delta}
\end{equation}
and
\begin{equation}\label{210312e5_17}
    \ang(G(w_{\theta}),G(w_{\tilde{\theta}})) \lesssim \rho^{-1/2}.
\end{equation}
\end{lem}

\begin{proof}
The bound \eqref{210312e5_17} is trivial, and we only need to prove \eqref{210312e5_16}. Let us assume that neither $\widetilde{T}_{\tilde{\theta}, \tilde{v}}$ nor $T_{\theta,v} \cap B(\widetilde{\bx}_0,\rho )$ is empty. Let $l_{\theta, v}$ and $l_{\tilde{\theta}, \tilde{v}}$ be the core lines of $T_{\theta,v}$ and $\widetilde{T}_{\tilde{\theta}, \tilde{v}}$ respectively. The two tubes involved in \eqref{210312e5_16} have width $\lesim r^{1/2+\delta}$, and therefore to prove \eqref{210312e5_16}, it suffices to consider the core lines of the tubes only. From the definition of core lines (see also \eqref{coreline-2}), we can write $l_{\theta, v}$ as
\begin{equation}\label{210314e5_18}
    \nabla_{\omega}\phi^{\lambda}(x, t; \omega_{\theta})-\nabla_{\omega} \phi^{\lambda}(x_0, t_0; \omega_{\theta})=-v;
\end{equation}
moreover, we can write the core line of $l_{\tilde{\theta}, \tilde{v}}$ as \begin{equation}\label{210314e5_19}
    \nabla_{\omega}\phi^{\lambda}(x, t; \omega_{\tilde{\theta}})-\nabla_{\omega} \phi^{\lambda}(\widetilde{x}_0, \widetilde{t}_0; \omega_{\tilde{\theta}})=-\tilde{v}.
\end{equation}
Suppose that $l_{\theta, v}$ passes through a point $(x_1, \widetilde{t}_0)$ and that $l_{\tilde{\theta}, \tilde{v}}$ passes through $(x_2, \widetilde{t}_0)$. Note that the angle between $l_{\theta, v}$ and $l_{\tilde{\theta}, \tilde{v}}$ is at most $\rho^{-1/2}$ and that we are computing a Hausdorff distance within a ball of radius $\rho$. To prove \eqref{210312e5_16}, it therefore suffices to show that $|x_1-x_2|\lesim r^{1/2+\delta}$, which, by \eqref{210201e3_40}, is the same as saying that 
\begin{equation}\label{210313e5_20}
    \Big|\nabla_{\omega} \phi^{\lambda}(x_1, \widetilde{t}_0; \omega_{\theta})-\nabla_{\omega}\phi^{\lambda}(x_2, \widetilde{t}_0; \omega_{\theta})\Big|\lesim r^{1/2+\delta}. 
\end{equation}
We consider two cases $\rho\le r^{1/2+\delta}$ and $\rho\ge r^{1/2+\delta}$ separately. \\

Assume we are in the former case. By the assumption that $T_{\theta, v}\cap B(\widetilde{\bx}_0, \rho)$ is not empty, we obtain $|x_1-\widetilde{x}_0|\lesim r^{1/2+\delta}$. Moreover, by \eqref{smalltubesinlargetube}, \eqref{210314e5_3}, \eqref{210314e5_18}, \eqref{210314e5_19} and the triangle inequality, we obtain that 
\begin{equation}\label{210314e5_21}
    \begin{split}
        & \big|\nabla_{\omega}\phi^{\lambda}(x_2, \widetilde{t}_0; \omega_{\tilde{\theta}})-\nabla_{\omega}\phi^{\lambda}(\widetilde{x}_0, \widetilde{t}_0; \omega_{\tilde{\theta}})\big|\\
        & \lesim r^{1/2+\delta} + \big|\nabla_{\omega}\phi^{\lambda}(x_1, \widetilde{t}_0; \omega_{\theta})-\nabla_{\omega}\phi^{\lambda}(\widetilde{x}_0, \widetilde{t}_0; \omega_{\theta})\big|.
    \end{split}
\end{equation}
This, together with \eqref{210201e3_40} and the mean value theorem, implies that 
\begin{equation}
    |x_2-\widetilde{x}_0| \lesim r^{1/2+\delta},
\end{equation}
which further leads to the desired bound for $|x_1-x_2|$.\\

Assume we are in the latter case. The starting point of the proof of this is similar as above. By the assumption that $T_{\theta, v}\cap B(\widetilde{\bx}_0, \rho)$ is not empty, we obtain $|x_1-\widetilde{x}_0|\lesim \rho$. This, combined with \eqref{210314e5_21}, implies that $|x_2-\widetilde{x}_0|\lesim \rho$. Next, the proof starts to be different. 
\begin{equation}
    \begin{split}
        & \nabla_{\omega} \phi^{\lambda}(x_1, \widetilde{t}_0; \omega_{\theta})-\nabla_{\omega}\phi^{\lambda}(x_2, \widetilde{t}_0; \omega_{\theta})\\
        &=\nabla_{\omega} \phi^{\lambda}(x_1, \widetilde{t}_0; \omega_{\theta})-\nabla_{\omega} \phi^{\lambda}(x_2, \widetilde{t}_0; \omega_{\tilde{\theta}})+\nabla_{\omega} \phi^{\lambda}(x_2, \widetilde{t}_0; \omega_{\tilde{\theta}})-\nabla_{\omega}\phi^{\lambda}(x_2, \widetilde{t}_0; \omega_{\theta})
    \end{split}
\end{equation}
By \eqref{210314e5_18} and \eqref{210314e5_19}, the last expression can be written as 
\begin{equation}
    \nabla_{\omega} \phi^{\lambda}(\widetilde{x}_0, \widetilde{t}_0; \omega_{\theta})-\nabla_{\omega} \phi^{\lambda}(\widetilde{x}_0, \widetilde{t}_0; \omega_{\tilde{\theta}})+\nabla_{\omega} \phi^{\lambda}(x_2, \widetilde{t}_0; \omega_{\tilde{\theta}})-\nabla_{\omega}\phi^{\lambda}(x_2, \widetilde{t}_0; \omega_{\theta}).
\end{equation}
By the mean value theorem and the bound \eqref{210131e3_28} with $\beta=1$ and $\beta'=2$, the absolute value of the last display can be bounded by 
\begin{equation}
    |x_2-\widetilde{x}_0| |\om_{\tilde{\theta}}-\om_{\theta}|\lesim \rho \rho^{-1/2}\lesim \rho^{1/2}.
\end{equation}
Recall that $\rho\le r$. This finishes the proof of latter case, thus the proof of the whole lemma.
\end{proof}



\section{The transverse equidistribution property}\label{section_transverse}

The proof of the transverse equidistribution estimate requires us to study wave packets from different scales. To make ourselves clear, we sometimes use ``large wave packet" to mean the scale $r$ wave packet, and use ``small wave packet" to mean the scale $\rho$ one.

Recall the admissible parameters in \eqref{constants_z}.
Let us introduce more notations for the next several sections. The first notation is about the transverse complete intersection. The second definition is about the tangency between tubes and a transverse complete intersection. The third definition is about collections of tangent tubes at two different scales.

\begin{defi}

Let $P_1,\ldots,P_{n-m}:\R^n \rightarrow \R$ be polynomials. We consider the common zero set
\begin{equation}\label{variety}
    Z(P_1,\ldots,P_{n-m}):= \{x \in \R^n: P_1(x)=\cdots=P_{n-m}(x)=0 \}.
\end{equation}
Suppose that for all $z \in Z(P_1,\ldots,P_{n-m})$, one has
\begin{equation}
\bigwedge_{j=1}^{n-m}
    \nabla P_j(z) \neq 0.
\end{equation}
Then a connected branch of this set, or a union of connected branches of this set, is called an $m$-dimensional transverse complete intersection. Given a set $Z$ of the form \eqref{variety}, the degree of $Z$ is defined by
\begin{equation}
    \min \Big({\prod_{i=1}^{n-m}\mathrm{deg}(P_i)}\Big),
\end{equation}
where the minimum is taken over all possible representations of $Z=Z(P_1,\ldots,P_{n-m})$.

\end{defi}

\begin{defi} Let $r \geq 1$ and $Z$ be an $m$-dimensional transverse complete intersection. A tube $T_{\theta,v}(\bx_0) \in \T[B(\bx_0,r)]$ is said to be $r^{-1/2+\delta_m}$-tangent to $Z$ in $B(\bx_0,r)$ if it satisfies
\begin{itemize}
    \item $T_{\theta,v}(\bx_0) \subset N_{r^{1/2+\delta_m}}(Z) \cap B(\bx_0,r)$; 
    \item For every $z \in Z \cap B(\bx_0,r)$, if there is $y \in T_{\theta,v}(\bx_0)$ with $|z-y| \lesssim r^{1/2+\delta_m}$, then one has
    \begin{equation}
    \label{angle-condition}
        \ang(G(\theta),T_zZ) \lesssim r^{-1/2+\delta_m}.
    \end{equation}
    Here, $T_zZ$ is the tangent space of $Z$ at $z$.
\end{itemize}

\end{defi}

\begin{defi}
Let $r\geq \rho \geq 1$ and $Z$ be an $m$-dimensional transverse complete intersection and let $B(\tilde{\bx}_0,\rho) \subset B(\bx_0,r)$. Define a collection of bigger tangent tubes inside a smaller ball as
\begin{equation}
\begin{split}
\label{T-Z}
    \T_{Z}[B(\bx_0,r)]
    := \{ T \in \T[B(\bx_0,r)]: T \text{ is $r^{-1/2+\delta_m}$-tangent to } Z \text{ in } B(\bx_0,r) \}.
\end{split}
\end{equation}
Given an arbitrary translation $b\in\ZR^n$, define 
\begin{equation}
\begin{split}
    \widetilde{\T}_b[B(\tilde{\bx}_0,\rho)]:=\{ \widetilde{T} \in \widetilde{\T}[B(\tilde{\bx}_0,\rho)]: \widetilde{T} \text{ is $\rho^{-1/2+\delta_m}$-tangent to $Z+b$ in $B(\tilde{\bx}_0,\rho)$}  \}.
\end{split}
\end{equation}
For simplicity, we sometimes use the notation $\T_Z$ and $\widetilde{\T}_b$ in short for $\T_{Z}[B(\bx_0,r)]$ and $\widetilde{\T}_b[B(\tilde{\bx}_0,\rho)]$, respectively.
\end{defi}

Let $1 \leq r^{1/2} \leq \rho \leq r$ and $\wt\bx_0=(\wt x_0,\wt t_0)\in\ZR^n$.
For every pair $(\tilde{\theta},w) \in \Theta_\rho \times r^{1/2}\Z^{n-1}$, we define a collection of bigger tubes 
\begin{equation}
\T_{\tilde{\theta},w} :=\{T_{\theta,v} \in \T[B(\bx_0,r)]: \mathrm{dist}(\theta,\tilde{\theta}) \lesssim \rho^{-1/2} \;\; \mathrm{and} \;\; |v+ \nabla_\om\phi_{\bx_0}^{\lambda}(\wt\bx_0;\om_\theta) -w| \lesssim r^{1/2} \}.
\end{equation}
Then for every $T\in\ZT_{\ti\theta,w}\cap\ZT_Z$, if $T\cap B(\wt\bx_0,2\rho)\not=\varnothing$, the intersection of $T$ and the horizontal plane $\{t=\wt t_0\}$ is contained in the ball $B:=B^{n-1}(w_{H},Cr^{1/2+\de})\subset B^{n-1}(\wt x_0,3\rho)$, where the center $w_H$ is defined as
\begin{equation}
\label{w-H}
    w_H:=\wt t_0\om_{\ti\theta}+\frac{\la (\nabla_\omega\phi^{\lambda}(\wt\bx_0;\omega_{\ti\theta})-w)}{(\la^2-|\nabla_\omega\phi^{\lambda}(\wt\bx_0;\omega_{\ti\theta})-w|^2)^{1/2}}.
\end{equation}
This is due to the coreline equation \eqref{coreline} for the tube $T$, and Lemma \ref{0128lemma}. We remark that if we perturb $\wt t_0$ by an extremely small factor $\be$, the intersection between $T$ and the horizontal plane $\{t=\bar t_0\}$ with $\bar t_0=\wt t_0+\be$ is still contained in the ball $B$.

\medskip

The main goal of this section is to prove the following lemma.
\begin{lem}\label{transverse-equidistribution}
Let $|b| \lesssim r^{1/2+\delta_m}$ and recall \eqref{h_W}.
Suppose that $h$ is concentrated on large wave packets from $\T_{Z} \cap \T_{\tilde{\theta},w}$ for some $(\tilde{\theta},w) \in \Theta_\rho \times r^{1/2}\Z^{n-1}$.  Then for every $\widetilde{\W} \subset \widetilde{\T}_b$, 
\begin{equation}
\label{local-TEE}
    \big\| h|_{\widetilde\W} \big\|_2^2 \lesssim r^{O(\delta_m)}(r/\rho)^{-(n-m)/2}\|h\|_2^2.
\end{equation}
As a consequence, for every function $h$ concentrated on $\T_Z$, 
\begin{equation}
\label{global-TEE}
    \big\| h|_{\widetilde{\W}}\big\|_2^2 \lesssim r^{O(\delta_m)}(r/\rho)^{-(n-m)/2}\|h\|_2^2.
\end{equation}
\end{lem}
We first show that \eqref{global-TEE} follows from the local one \eqref{local-TEE}. Decompose the function $h$ as
\begin{equation}
    h=\sum_{(\ti\theta,w)}h_{\ti\theta,w}
\end{equation}
such that $h_{\ti\theta,w}$ is concentrated on wave packets from $\ZT_Z\cap\ZT_{\ti\theta,w}$. Since when $|w-x|\geq Cr^{1/2+\de}$, 
$\wh{h}_{\ti\theta,w}(x)=\rapid(r)\|h\|_2$, we have the almost $L^2$ orthogonality
\begin{equation}
    \big\| h|_{\widetilde{\W}}\big\|_2^2\lesim r^{O(\de)}\sum_{(\ti\theta,w)} \big\| h_{\ti\theta,w}|_{\widetilde{\W}}\big\|_2^2.
\end{equation}
Finally, we use the local estimate \eqref{local-TEE} and Lemma \ref{210103lem5_1} to conclude the global estimate \eqref{global-TEE}.

\medskip
To prove \eqref{local-TEE}, note that $\wt\ZT_b$ only consists smaller tubes $\wt T$ that  $\wt T\cap B(\wt\bx_0,3\rho/2)\not=\varnothing$. By Lemma \ref{210103lem5_3}, Lemma \ref{1861} and Lemma \ref{0128lemma}, we can assume without loss of generality that for every bigger tube $T\in\ZT_{\ti\theta,w}\cap\ZT_Z$, one has $T\cap B(\wt\bx_0,2\rho)\not=\varnothing$.

We break the proof of the local transverse equidistribution estimate \eqref{local-TEE} into several smaller lemmas. 
The first thing we would like to find out is the location of $(h|_{\widetilde{\W}})^\wedge$. Consider $Z_0$, the intersection between of variety $Z+b$ and the horizontal hyperplane $\{(x,t):t=\bar t_0\}$. By the transversality theorem (For example, Theorem \ref{transversality-thm}), we can choose $\bar t_0=\wt t_0+\be$ for an extremely small number $\be$, so that $Z_0$ is a transverse complete intersection. Since the small perturbation $\be$ is harmless in our proof, to save us from abundant notations, let us assume $\bar t_0=\wt t_0$. Hence $Z_0=Z+b\cap\{t=\wt t_0\}$ can be considered as a transverse complete intersection in $\ZR^{n-1}$, and $\deg(Z_0)=O(\deg(Z))$.

If $H^\la$ was the Fourier extension operator, $(h|_{\widetilde{\W}})^\wedge$ is just contained in a thin neighborhood of $Z_0$. While in our case, $(h|_{\widetilde{\W}})^\wedge$ is roughly contained in a thin neighborhood of $\Phi(Z_0)$, where $\Phi:\ZR^{n-1}\to\ZR^{n-1}$ is a smooth map defined as 
\begin{equation}
\label{phi}
    \Phi(x):=-\nabla_\om\phi^{\lambda}(x,\wt t_0;\om_{\ti\theta})=- \frac{\la\big(\wt t_0\omega_{\ti\theta}-x\big)}{(\la^2+|x-\bar{t}_0\om_{\ti\theta}|^2)^{1/2}}.
\end{equation}

Before proving something about the support of $(h|_{\widetilde{\W}})^\wedge$ rigorously, let us take a look at the smooth function $\Phi$. The next lemma shows that $\Phi$ indeed looks like the identity map of $\ZR^{n-1}$.
\begin{lem}
\label{essential-identity-lem}
For any $x\in\ZR^{n-1}$ and $r\geq \l>0$ that $B^{n-1}(x,\l)\subset B^{n-1}(0,3C_n\la)$, one has
\begin{equation}
\label{esseitial-identity-eq}
   B(\Phi(x),\l/C)\subset\Phi(B(x,\l))\subset B(\Phi(x),C\l). 
\end{equation}
The constant $C$ only depends on the choice of $C_n$ in Theorem \ref{201204thm5_1}.
\end{lem}
\begin{proof}
By Taylor's theorem, we have that for any $y\in B(x,r)$,
\begin{equation}
    \Phi(y)-\Phi(x)=(y-x)J_\Phi(x)+O(\la^{-1}|y-x|^2).
\end{equation}
The Jacobian $J_\Phi$ is a $(n-1)\times (n-1)$ symmetric matrix $A=(a_{kl})$, with
\begin{equation}
\label{jacobian}
    a_{kl}:=\begin{cases}
    \dfrac{\la(\wt t\om_{\ti\theta}-x)_k(\wt t\om_{\ti\theta}-x)_l}{(\la^2+|x-\wt t\om_{\ti\theta}|^2)^{3/2}} & \text{ when } k\not=l,\\[3ex]
    \dfrac{\la[\la^2+\sum_{j\not=k}(\wt t\om_{\ti\theta}-x)_j^2]}{(\la^2+|x-\wt t\om_{\ti\theta}|^2)^{3/2}} & \text{ when } k=l.
\end{cases}
\end{equation}
Similar to the argument in \eqref{210112e3_29}, we can simplify $A$ as 
\begin{equation}
    A=\frac{\la}{(\la^2+|x-\wt t\om_{\ti\theta}|^2)^{1/2}}\Big(\la^2 I_n-\frac{(\wt t\om_{\ti\theta}-x)^{T}(\wt t\om_{\ti\theta}-x)}{\la^2+|x-\wt t\om_{\ti\theta}|^2}\Big).
\end{equation}
Inside the bracket, the first matrix has eigenvalues $1$, while the second matrix has only one eigenvalue $|x-\wt t\om_{\ti\theta}|^2/(\la^2+|x-\wt t\om_{\ti\theta}|^2)\leq 4C_n/(1+4C_n)<1$ since $x\in B(0,3C_n\la)$ and since $t\leq C_n\la$. This proves that the eigenvalues of $A$ are all positive and have lower bound $C^{-1}$ and upper bound $C$ for an absolute constant that depends only on $C_n$, uniformly in $x$. Thus, 
\begin{equation}
    |\Phi(y)-\Phi(x)|\sim|y-x|+O(\la^{-1}|y-x|^2),
\end{equation}
which yields $\Phi(B(x,\l))\subset B(\Phi(x),C\l)$ since $|y-x|\leq r\leq R\leq \la^{1-\e}$ as assumed in \eqref{210201e3_3}. For the other side, by the implicit function theorem, we know that $\Phi^{-1}$ is well define in the domain $\Phi(B^{n-1}(0,3C_n\la))$ and $J_{\Phi^{-1}}(x)=(J_{\Phi}(x))^{-1}$. So the eigenvalues of $J_{\Phi^{-1}}(x)$ have an upper bound $C^{-1}$ and a lower bound $C$. Hence we can prove $B(\Phi(x),\l/C)\subset\Phi(B(x,\l))$ similarly.
\end{proof}
A direct corollary of this lemma is the following:
\begin{cor}
\label{corollary-essen-iden}
Let $Z_0$ and $B$ be as above. Then
\begin{equation}
    N_{\rho^{1/2+\de_m}/C}(\Phi(Z_0))\cap \Phi((1/C)B)\subset\Phi(N_{\rho^{1/2+\de_m}}(Z_0)\cap B)\subset N_{C\rho^{1/2+\de_m}}(\Phi(Z_0))\cap \Phi(CB).
\end{equation}
\end{cor}

\medskip

Now via Lemma \ref{essential-identity-lem} and Corollary \ref{corollary-essen-iden}, we can say something rigorously related to the support of $(h|_{\widetilde{\W}})^\wedge$. This is shown in the next lemma.

\begin{lem}
Let $h$ be concentrated on bigger wave packets from $\ZT_{\ti\theta,w}\cap\ZT_Z$. Recall that $Z_0$ was defined in above \eqref{phi}, and $h|_{\widetilde{\W}}$ was introduced in Lemma \ref{transverse-equidistribution}. Then
\begin{equation}
\label{l2-comparison}
    \big\|(h|_{\widetilde{\W}})^\wedge\big\|_2\lesim\big\|\wh{h}\cdot \Id_{\{N_{C\rho^{1/2+\de_m}}(\Phi(Z_0))\cap \Phi(CB)\}}\big\|_2+\rapid(\rho)\|h\|_2.
\end{equation}
\end{lem}
\begin{proof}
By Lemma \ref{210103lem5_1}, one has 
\begin{equation} 
    \|(h|_{\widetilde{\W}})^\wedge\|_2^2\lesim\sum_{\wt T\in\wt\W}\|{h}_{\wt T}\|_2^2.
\end{equation}
The scale $\rho$ wave packet $\wh{h}_{\ti T}$ was defined in \eqref{single-wp-tube}. We use its definition to have
\begin{equation}
    \sum_{\wt T\in\wt\W}\|\wh{h}_{\wt T}\|_2^2=\rho^{n-1}\sum_{\wt T\in\wt\W}|(h_{\wt\bx_0}\psi_{\ti\theta})^{\wedge}(v)|^2\|\wt\psi_{\ti\theta}\|_2^2=\rho^{(n-1)/2}\sum_{\wt T\in\wt\W}|(h_{\wt\bx_0}\psi_{\ti\theta})^{\wedge}(v)|^2.
\end{equation}
Since we can write $(h_{\wt\bx_0}\psi_{\ti\theta})^{\wedge}(v)=\wh{h}\ast (e^{-2\pi i\phi^{\lambda}(\wt\bx_0; \ \cdot) }\wt\psi_{\ti\theta}(\cdot))^\wedge(v)$, and since the $L^1$ norm of the second function $(e^{-2\pi i\phi^{\lambda}(\wt\bx_0; \ \cdot) }\wt\psi_{\ti\theta}(\cdot))^\wedge$ is bounded above by $O(1)$, one can use H\"older's inequality for $|(h_{\wt\bx_0}\psi_{\ti\theta})^{\wedge}(v)|^2$ and obtain
\begin{align}
\label{wpt-trans-equi}
    \sum_{\wt T\in\wt\W}\|\wh{h}_{\wt T}\|_2^2\lesim  \int|\wh{h}(y)|^2\Big(\sum_{\wt T_{\ti\theta,\ti v}\in\wt\W}|\rho^{(n-1)/2}(e^{-2\pi i\phi^{\lambda}(\wt\bx_0; \ \cdot) }\wt\psi_{\ti\theta}(\cdot))^\wedge(y-\wt v)|\Big)dy.
\end{align}
We claim that the sum inside the bracket is $O(1)$, and it decays rapidly outside of the set $N_{C\rho^{1/2+\de_m}}(\Phi(Z_0))\cap \Phi(CB)$. This proves the lemma.

To prove the $O(1)$ upper bound, one just needs to notice that there are $O(1)$ caps $\ti\theta$ making contribution in $\wt\W$. It remains to prove the rapidly-decaying property.

To prove the rapidly-decaying property, we first notice that $h_{\wt T_{\ti\theta,\wt v}}=\rapid(\rho)\|h\|_2$ unless $\wt T_{\ti\theta,\wt v}\cap\{t=\wt t_0\}\subset N_{2{\rho^{1/2+\de_m}}}(Z_0)\cap 2B$ for every smaller tube $\wt T_{\ti\theta,\wt v}\in\wt\W$. This is because on one hand $\wt T_{\ti\theta,\wt v}\cap \{t=\wt t_0\}\subset N_{2{\rho^{1/2+\de_m}}}(Z_0)$ since $\wt T_{\ti\theta,\wt v}$ is tangent to $Z+b$; on the other hand, since $h$ is concentrated on $\ZT_{\ti\theta,w}$, by Lemma \ref{1861} we know that $h_{\wt T_{\ti\theta,\wt v}}=\rapid(\rho)\|h\|_2$ unless $\wt T_{\ti\theta,\wt v}\in \widetilde{\ZT}_{\theta,v}[B(\widetilde{\bx}_0,\rho)]$ for some $(\theta,v)$ with  $T_{\theta,v}\in\ZT_{\wt\theta,w}$. But for such $T_{\theta,v}$, one has $T_{\theta,v}\cap \{t=\wt t_0\}\subset B$ as shown in above \eqref{w-H}. Hence by Lemma \ref{0128lemma} we have that $h_{\wt T_{\ti\theta,\wt v}}=\rapid(\rho)\|h\|_2$ unless $\wt T_{\ti\theta,\wt v}\cap \{t=\wt t_0\}\subset 2B$.

Now we only need to consider those smaller tubes $\wt T_{\ti\theta,\wt v}$ with $\wt T_{\ti\theta,\wt v}\cap\{t=\wt t_0\}\subset N_{2{\rho^{1/2+\de_m}}}(Z_0)\cap 2B$. Using this information, we would like to find out the location of $\wt v$ in the bracket of \eqref{wpt-trans-equi}. Recall that the coreline of the smaller tube $\wt T_{\ti\theta,\wt v}$ satisfies \eqref{coreline}, which indeed is \eqref{coreline-2}. On the hyperplane $\{t=\wt t_0\}$, we can rewrite (4.36) as $\wt v=-\nabla_\om\phi^\la(x,\wt t;\om_{\ti\theta})+\nabla_{\omega}\phi^{\lambda}(\wt\bx_0; \omega_{\ti\theta})=:\Psi(x)$. Hence by Corollary \ref{corollary-essen-iden} one  has
\begin{equation}
\label{location-widetilde-v}
    \wt v\in \Psi(T_{\ti\theta,\wt v}\cap\{t=\wt t_0\})\subset \Psi(N_{2{\rho^{1/2+\de_m}}}(Z_0)\cap 2B)\subset N_{C\rho^{1/2+\de_m}}\Psi(Z_0)\cap\Psi(2B).
\end{equation}

Finally, we know from Lemma \ref{210103lem5_3} that the function $(e^{-2\pi i\phi^{\lambda}(\wt\bx_0; \ \cdot) }\wt\psi_{\ti\theta}(\cdot))^\wedge(y-\wt v)$ in \eqref{wpt-trans-equi} decays rapidly unless $|y-\wt v+\nabla_\om\phi^\la(\wt\bx_0,\om_{\ti\theta})|\lesim \rho^{1/2+\de_m}$, that is, $y-\wt v\in B^{n-1}(-\nabla_\om\phi^\la(\wt\bx_0,\om_{\ti\theta}), C\rho^{1/2+\de_m})$. Plugging this back into \eqref{location-widetilde-v} and noticing $\Phi+\nabla_{\omega}\phi^{\lambda}(\wt\bx_0; \omega_{\ti\theta})=\Psi$, we therefore can conclude that $(e^{-2\pi i\phi^{\lambda}(\wt\bx_0; \ \cdot) }\wt\psi_{\ti\theta}(\cdot))^\wedge(y-\wt v)$ decays rapidly unless
\begin{equation}
     y\in N_{C\rho^{1/2+\de_m}}\Phi(Z_0)\cap\Phi(CB),
\end{equation}



\noindent where $\Phi$ was defined in \eqref{phi}. This proves the claim and hence the lemma.
\end{proof}

To conclude \eqref{local-TEE}, we need to prove the transverse equidistribution estimate stated below.
\begin{prop}
\label{trans-dist-main}
Recall $Z_0=(Z+b)\cap\{t=\wt t_0\}$ and $B=B^{n-1}(w,r^{1/2+\de_m})$. Suppose that $h$ is concentrated on scale $r$ wave packets in $\ZT_{\ti\theta,w}\cap\ZT_Z$. Then 
\begin{equation}
    \int|\wh{h}|^2\cdot \Id_{\{N_{C\rho^{1/2+\de_m}}(\Phi(Z_0))\cap \Phi(CB)\}}\lesim r^{O(\de_m)}\Big(\frac{\rho}{r}\Big)^{(n-m)/2}\|h\|_2^2.
\end{equation}
\end{prop}
\noindent Note that the desired estimate \eqref{local-TEE} is a direct corollary of \eqref{l2-comparison} and Proposition \ref{trans-dist-main}.

\medskip

The proof of Proposition \ref{trans-dist-main} relies on an auxiliary lemma. Let us introduce some more definitions. First, define
\begin{equation}
    \ZT_{Z,B,\ti\theta}:=\{T_{\theta,v}\in\ZT_Z:\mathrm{dist}(\theta,\tilde{\theta}) \lesssim \rho^{-1/2}~\text{and}~T_{\theta,v}\cap\{t=\wt t_0\}\cap B\not=\varnothing\},
\end{equation}
so as proved in above \eqref{w-H}, one has $\ZT_{\ti\theta,w}\cap\ZT_Z\subset\ZT_{Z,B,\ti\theta}$. We will prove Proposition \ref{trans-dist-main} with $\ZT_{\ti\theta,w}\cap\ZT_Z$ replaced by $\ZT_{Z,B,\ti\theta}$. Next, for every (linear) subspace $V$ in $\ZR^n$, we define a collection of bigger wave packets $\ZT_{V,B,\ti\theta}$ as
\begin{equation}
    \ZT_{V,B,\ti\theta}:=\{(\theta,v):T_{\theta,v}\cap B\not=\varnothing,~\ang(G(\om_\theta),V)\lesim r^{-1/2+\de_m}~\text{and}~\mathrm{dist}(\theta,\tilde{\theta}) \lesssim \rho^{-1/2}\}.
\end{equation}


\begin{lem}
\label{trans-equi-sub-lem}
Suppose that $V$ is a subspace of $\ZR^{n}$. Let $V_0:=V\cap \{t=\wt t_0\}$ and define $V'$ as $V_0^\perp$ in $\ZR^{n-1}$. If $g$ is concentrated on bigger wave packets from $\ZT_{V,B,\ti\theta}$, if $\Pi\subset\{t=\wt t_0\}$ is any affine subspace parallel to $V'$ and if $y_1\in\Pi\cap \Phi(CB)$, then
\begin{equation}
\label{transverse-equi-subspace}
    \int_{\Pi\cap B(y_1,\rho^{1/2+\de_m})}|\wh g|^2\lesim r^{O(\de_m)}\Big(\frac{\rho^{1/2}}{r^{1/2}}\Big)^{\dim(V')}\int_{\Pi}|\wh g|^2. 
\end{equation}
\end{lem}
\begin{proof}
Since $g$ is concentrated on large wave packets from $\ZT_{V,B,\ti\theta}$ and since $G(\om)$ is linear (up to a scalar depending on $\om$) as shown in \eqref{Gauss-map}, there is a shift $\om_V\in B^{n-1}(0,1)$ such that
\begin{equation}
    \{\om:\ang(G(\om),V)\lesim r^{-1/2+\de_m}\}\subset\{\om:\dist(\om,V_0+\om_V)\lesim r^{-1/2+\de_m}\}.
\end{equation}
It implies that $g$ is supported in the $r^{-1/2+\de_m}$ neighborhood of $V_0+\om_V$ inside the unit ball $B^{n-1}(0,1)$. As a result, the Fourier transform of $(\wh {g}|_\Pi)^{\vee}$ is supported in an $n-m$ dimensional $r^{-1/2+\de_m}$ ball centered at $\text{proj}_{V'}(\om_V)$, which implies (See also Lemma \ref{201204lem6_4})
\begin{equation}
    |(\wh g|_\Pi)|\lesim |(\wh g|_\Pi)|\ast \eta_{r^{1/2-\de_m}}. 
\end{equation}
Finally, we integrate $|(\wh g|_\Pi)|^2$ inside the ball $B(y,\rho^{1/2+\de_m})$ and invoke H\"older's inequality to conclude \eqref{transverse-equi-subspace}.
\end{proof}

\begin{proof}[Proof of Proposition \ref{trans-dist-main}]
We prove Proposition \ref{trans-dist-main} via Lemma \ref{trans-equi-sub-lem}. Our proof is similar to the proof of Lemma 6.2 in \cite{guth2018} and the proof of Lemma 8.4 in \cite{MR4047925}.

Since wave packets in $\ZT_{Z,B,\ti\theta}$ are tangent to the variety $Z$ inside the ball $B$, by the angular condition \eqref{angle-condition} we have
\begin{equation}
    \ang(G(\theta),T_zZ)\lesim r^{-1/2+\de_m}
\end{equation}
for every $z\in Z\cap 2B$ and $T_{\theta,v}\in\ZT_{Z,B,\ti\theta}$. One thus can find a subspace $V\subset\ZR^n$ of minimal dimension and $\dim V\leq\dim Z$ such that for all $\theta$ making contribution in $\ZT_{Z,B,\ti\theta}$, 
\begin{equation}
\label{angle-V}
    \ang(G(\theta),V)\lesim r^{-1/2+\de_m}.
\end{equation}
It implies that the function $h$ is indeed concentrated on wave packets from $\ZT_{V,B,\ti\theta}$. So we can apply Lemma \ref{trans-equi-sub-lem} to obtain a subspace $V'$ that 
\begin{equation}
\label{affine-subspace}
    \int_{\Pi\cap B(y_1,\rho^{1/2+\de_m})}|\wh h|^2\lesim r^{O(\de_m)}\Big(\frac{\rho^{1/2}}{r^{1/2}}\Big)^{\dim(V')}\int_{\Pi}|\wh h|^2
\end{equation}
for any $\Pi\subset\{t=\wt t_0\}$ being parallel to $V'$ and $y_1\in\Pi\cap \Phi(CB)$. To finish the proof, we need three additional claims.

{\bf{Claim 1.}} The pushfoward $\Phi(Z_0)$ is quantitatively transverse to $V'$ at every point $z\in\Phi(Z_0)\cap\Phi(B(0,3C_n\la))$.

{\bf{Claim 2.}} $\Phi^{-1}(\Pi)$ is an $n-1-\dim(V')$ dimensional transverse complete intersection in $\ZR^{n-1}$.

{\bf{Claim 3.}} $\Pi\cap N_{C\rho^{1/2+\de_m}}(\Phi(Z_0))\cap \Phi(CB)$ can be covered by $r^{O(\de_m)}(r^{1/2}/\rho^{1/2})^{\dim Z_0-\dim V_0}$ many balls in $\Pi$ of radius $\rho^{1/2+\de_m}$. 

Assume at first that the three claims were verified. We plug Claim 3 back to \eqref{affine-subspace} so that
\begin{equation}
    \int_{\Pi}|\wh{h}|^2\cdot \Id_{\{N_{C\rho^{1/2+\de_m}}(\Phi(Z_0))\cap \Phi(CB)\}}\lesim r^{O(\de_m)}\Big(\frac{\rho}{r}\Big)^{(n-m)/2}\int_{\Pi}|\wh h|^2.
\end{equation}
Integrate over all (generic) affine subspaces $\Pi$ that are parallel to $V'$ to conclude the proof of Proposition \ref{trans-dist-main}.

\medskip

\noindent\emph{Proof of Claim 1.} We follow essentially the proof strategy in \cite{guth2018}. Suppose that Claim 1 fails. It means that there is a point $\Phi(z)\in \Phi(Z_0)\cap \Phi(B(0,C\la))$ and a subspace $W_\Phi\subset T_{\Phi(z)}(\Phi(Z_0))$ with $\dim Z_0-\dim W_\Phi+\dim V'<n-1$, such that for any non-zero vector $w\in W_\Phi$, there is a big constant $C'$ only depending on the $C_n$ in Theorem \ref{201204thm5_1} that
\begin{equation}
    \ang(w, V')\leq (C')^{-2}.
\end{equation}
Let $(\Phi^{-1})^\ast_z:T_{\Phi(z)}\Phi(Z_0)\to T_zZ_0$ be the map between tangent spaces. Then $\ang(v,(\Phi^{-1})_z^\ast(v)) <\pi/2-(C')^{-1}$ for any $v\in T_{\Phi(z)}\Phi(Z_0)$ because of the positivity of $J_{\Phi^{-1}}$ (See Lemma \ref{essential-identity-lem}). Hence if we define $W$ as the pullback $W:=(\Phi^{-1})^\ast_z(W_\Phi)\subset T_z Z_0$, then $\ang(w,V_0)\gtrsim1$ since $V_0$ is perpendicular to $V'$. As we will show below, this indeed implies for arbitrary non-zero vector $w\in W$,
\begin{equation}
\label{angle-condition-1}
    \ang(w,V)\gtrsim1.
\end{equation}
To prove \eqref{angle-condition-1}, let  $w=w_1+w_2$ with $w_1\perp V_0$ and $w_2\in V_0$, so $\text{orth}_V(w)=\text{orth}_V(w_1)$ (here we use $\text{orth}_V(w):=w-\text{proj}_V(w)$ in convention). Note that the lower bound $\ang(w,V_0)\gtrsim1$ gives $|w_1|\gtrsim|w|$. Now for any unit vector $v\in V$, we write $v=v_1+v_2$ where $v_1\perp V_0$ and $v_2\in V_0$, so $|v_1|\leq1$ and $\ang(v_1,\ZR^{n-1})\gtrsim1$ due to \eqref{angle-V}. Hence $|w_1\cdot v|=|w_1\cdot v_1|\leq c|w_1|$ for some absolute constant $c<1$, which implies $|\text{orth}_V(w)|=|\text{orth}_V(w_1)|\sim|w_1|\gtrsim|w|$ and hence \eqref{angle-condition-1}. Finally, since $\dim(W)=\dim(W_\Phi)$, one has $\dim(Z_0)-\dim(W)<\dim(V)$. Hence the angle estimate \eqref{angle-condition-1} contradicts the minimality of $\dim(V)$, which can be seen via repeating the proof in page 115 of \cite{guth2018}. This proves Claim 1.

\medskip

\noindent\emph{Proof of Claim 2.}
Since $\Pi$ is an affine subspace, it is the intersection of $\ka:=n-1-\dim(V')$ mutually orthogonal hyperplanes $L_1,\ldots,L_\ka$. Suppose that each $L_j$ is parameterized as
\begin{equation}
    L_j:m_j\cdot x+l_j=0,
\end{equation}
where $m_j$ is the normal vector of $L_j$ and $l_j\in\ZR$. Then $\Phi^{-1}(L_j)$ is indeed a branch of the quadratic variety
\begin{equation}
    \Phi^{-1}(L_j):|\la m_j\cdot(\wt t_0\om_{\ti\theta}-x)|^2-(\la^2+|x-\wt t_0\om_{\ti\theta}|^2)l_j=0.
\end{equation}
Similar to the proof of Lemma \ref{essential-identity-lem}, we can show $J_{\Phi^{-1}}(x)$ is not degenerate for all $x\in\text{Range}(\Phi)$. Hence $\Phi^{-1}(L_1)\cap\ldots\cap\Phi^{-1}(L_\ka)$ is a transverse complete intersection.

\medskip

\noindent\emph{Proof of Claim 3.} Similar to Claim 1, we know that $\Phi^{-1}(\Pi)\cap Z_0$ is a transverse complete intersection inside $B(0,C_n\la)$ (In fact, by Lemma \ref{transverse-lemma}, the pullback $\Phi^{-1}(\Pi)$ is transverse to $Z$ for generic affine spaces $\Pi$). Note that the dimension of the variety $\Phi^{-1}(\Pi)\cap Z_0$ is $\dim Z_0-\dim V_0$. By Wongkew's theorem \cite{Wongkew}, the set $N_{C\rho^{1/2+\de_m}}(\Phi^{-1}(\Pi)\cap Z_0)\cap CB$ can be covered by 
\begin{equation}
    r^{O(\de_m)}\Big(\frac{r^{1/2}}{\rho^{1/2}}\Big)^{\dim Z_0-\dim V_0}
\end{equation}
many $(n-1)$-dimensional balls. Hence via Lemma \ref{essential-identity-lem} we know that $N_{C\rho^{1/2+\de_m}}(\Pi\cap \Phi(Z_0))\cap \Phi(CB)$ can be covered the the same amount of $(n-1)$-dimensional balls, up to a constant. 

Since $T_z\Phi(Z_0)$ is quantitatively transverse to $\Pi$ at every point $z\in\Phi(Z_0)\cap \Phi(2CB)$ by Claim 1, one can argue similarly as in Lemma 8.13 in \cite{MR4047925} to conclude 
\begin{equation}
    \Pi\cap N_{C\rho^{1/2+\de_m}}(\Phi(Z_0))\cap \Phi(CB)\subset {\rm Proj}_{\Pi}[N_{C\rho^{1/2+\de_m}}(\Pi\cap \Phi(Z_0))\cap \Phi(CB)].
\end{equation}
Therefore $\Pi\cap N_{C\rho^{1/2+\de_m}}(\Phi(Z_0))\cap \Phi(CB)$ can be covered by $r^{O(\de_m)}(r^{1/2}/\rho^{1/2})^{\dim Z_0-\dim V_0}$ many lower dimensional balls in $\Pi$ of radius $\rho^{1/2+\de_m}$. This proves Claim 3.

\end{proof}

\section{Multigrains and functions concentrated near a variety}\label{f_section7}

In this section, following \cite{HR2019} and \cite{hickman2020note}, we introduce some definitions and propositions for later use in Section \ref{f_section8} and \ref{f_section9}. In contrast to \cite{hickman2020note}, the definitions and propositions are stated in terms of the ``dimension'' of the transverse complete intersection instead of the codimension so that the notations are consistent with those in the previous sections.


\subsection{Multigrain and nested structure}
Here we introduce some definitions and lemmas about grains and multigrains, following Section 3 of \cite{hickman2020note}. 


\begin{defi}
A grain is defined to be a pair $(S,B_r)$ where $S \subset \R^n$ is a transverse complete intersection and $B_r \subset \R^n$ is a ball of some radius $r>0$. The dimension of a grain $(S,B_r)$ is the dimension of the transverse complete intersection $S$, and its degree is the degree of $S$.
\end{defi}

\begin{defi} Let $(S,B(\bx_0,r))$ be a grain of dimension $m$. A function $f$ is said to be tangent to $(S,B(\bx_0,r))$ if it is concentrated on wave packets belonging to the collection
\begin{equation}
    \{T_{\theta,v}(\bx_0) \in \T[B(\bx_0,r)]:T_{\theta,v}(\bx_0) \;\mathrm{ is }\; r^{-1/2+\delta_m}\mathrm{-tangent \; to}\; S\; \mathrm{in}\; B(\bx_0,r) \}.
\end{equation}
\end{defi}

\begin{defi}
A multigrain $\vec{S}_{m}$ is an $(n-m+1)$-tuple of grains
\begin{equation*}
    \vec{S}_{m} = (\mathcal{G}_n, \dots, \mathcal{G}_{m}), \qquad \mathcal{G}_i = (S_i,B_{r_i}) \qquad \textrm{     for $m \leq i \leq n$}
\end{equation*}
satisfying
\begin{itemize}
    \item $\mathrm{dim} (S_i) = i$ for $m \leq i \leq n$,
    \item $S_n\supset S_{n-1}\supset\cdots\supset S_{m}$,
    \item $B_{r_{n}} \supset B_{r_{n-1}} \supset \cdots \supset B_{r_{m}}$.
    \end{itemize}
\end{defi}
The parameter $n-m$ is referred to as the \textit{level} of the multigrain $\vec{S}_{m}$. The \textit{complexity} of the multigrain is defined to be the maximum of the degrees $\deg S_i$ over all $m \leq i \leq n$. Finally, the \textit{multiscale} of $\vec{S}_{m}$ is the tuple $\vec{r} = (r_n, r_{n-1}, \dots, r_{m})$.
For two multigrains $\vec{S}_l$ and $\vec{S}_m$ with $m \leq l$, we write 
\begin{equation}
\nonumber
    \vec{S}_m \preceq \vec{S}_l
\end{equation}
if the first $n-l+1$ components of $\vec{S}_m$ agree those of $\vec{S}_l$.

\begin{defi}\label{nestedtube}
Let $\vec{S}_{m} = (\mathcal{G}_n, \dots, \mathcal{G}_{m})$ be a multigrain and 
\begin{equation*}
    \mathcal{G}_i = \big(S_i, B(\bx_i,r_i)\big) \quad \textrm{for $m \leq i \leq n$.}
\end{equation*}
Define $\T[\vec{S}_{m}]$ to be the set of scale $R := r_n$ tubes $T_{\theta_n,v_n}(\bx_n) \in \T[B(\bx_n,r_n)] $ satisfying the following hypothesis:\medskip

\noindent \textbf{Nested tube hypothesis.} There exists $T_{\theta_i,v_i}(\bx_i) \in \T[B(\bx_i,r_i )]$ for $m \leq i < n$ such that 
\begin{enumerate}
    \item $\mathrm{dist}(\theta_i,\theta_j) \lesssim r_j^{-1/2}$,
    \item $\mathrm{dist}\big(T_{\theta_j,v_j}(\bx_j),\, T_{\theta_i,v_i}(\bx_i) \cap B(\bx_j,r_j) \big) \lesssim r_i^{(1+\delta)/2}$,
    \item $T_{\theta_j,v_j}(\bx_j) \subset N_{r_j^{1/2+\delta_j}}S_j$
\end{enumerate}
 hold true for all $i,j$ with $m \leq j \leq i \leq n$.
\end{defi}

The direction set of $\T[\vec{S}_{m}]$ is defined by
\begin{equation}
    \Theta[\vec{S}_{m}]:=\{\theta \in \Theta_R: T_{\theta,v} \in \T[\vec{S}_{m}] \text{ for some $v \in R^{1/2}\Z^{n-1}$} \}.
\end{equation}
A main ingredient of the proof of Theorem \ref{201204thm5_1} is the following lemma.
\begin{lem}[Lemma 3.7 of \cite{hickman2020note}]\label{nestedpoly}
Let $\vec{S}_m$ be a level $n-m$ multigrain with multiscale $\vec{r}_m=(r_n,\ldots,r_m)$ and complexity at most $d$. If $R=r_n$ and the constants in \eqref{constants_z} are chosen appropriately, then 
\begin{equation}
    \#\Theta[\vec{S}_{m}] \lesssim_{\epsilon_{\circ},d} \Big(\prod_{i=m}^{n-1}r_i^{-1/2} \Big)
    R^{(n-1)/2+\epsilon_{\circ}}.
\end{equation}
\end{lem}

Since our tubes $T_{\theta,v}$ are straight, and the tubes $T_{\theta,v}$ corresponding to the same $\theta$ indicate the same direction (see the definition of the tube $T_{\theta,v}$ \eqref{coreline}), the above lemma can be proved by using the nested polynomial Wolff axioms of Zahl \cite{MR4205111} and Hickman, Rogers and Zhang \cite{HRZ}; for instance, one can follow the proof of Lemma 3.7 of \cite{hickman2020note}. We leave out the details.



\subsection{Some lemmas}

There are more definitions and lemmas that we will need in Section 8 and 9. Let us state them here, following Section 8 of \cite{HR2019} and Section 5 of  \cite{hickman2020note}. In this subsection, we fix a scale $r \geq 1$ and a smaller scale $r^{1/2} \leq \rho \leq r$, and we consider balls
\begin{equation}
    B(\widetilde{\bx}_0,\rho) \subset B(\bx_0,r) \subset [-3C_n\lambda,3C_n\lambda]^{n-1} \times [R/C_n,C_nR].
\end{equation} 
Recall some definitions in Section 6:
\begin{equation}
\begin{split}
    \T_{Z}[B(\bx_0,r)]&= \{ T \in \T[B(\bx_0,r)]: T \text{ is $r^{-1/2+\delta_m}$-tangent to } Z \text{ in } B(\bx_0,r) \},
    \\
    \widetilde{\T}_{b}[B(\widetilde{\bx}_0,\rho)]&= \{ \widetilde{T} \in \widetilde{\T}[B(\widetilde{\bx}_0,\rho)]: \widetilde{T} \text{ is $\rho^{-1/2+\delta_m}$-tangent to } Z+b \text{ in } B(\widetilde{\bx}_0,\rho) \}.
\end{split}
\end{equation}
For some technical issue and rigorousness, we introduce a ``thickening" of an arbitrary subset of $\T[B(\bx_0,r)]$: Given any $\W \subset \T[B(\bx_0,r)]$, we define
\begin{equation}\label{020184}
    \W^*:=\{T_{\theta,v} \in \T[B(\bx_0,r)]: \mathrm{dist}(\theta,{\theta}') \lesssim r^{-1/2}\, \mathrm{and}\, |v-{v}'| \lesssim r^{(1+\delta)/2}\, \mathrm{for\; some}\, ({\theta}',{v}') \subset \W \}.
\end{equation}
Note that $\W^*$ is a set slightly larger than $\W$. Intuitively, one can identify $\W^\ast$ with $\W$.

\begin{defi}
\label{def-upperarrow}
For $\widetilde{\mathbb{W}} \subset \widetilde{\T}[B(\widetilde{\bx}_0,\rho)]$ we define $\uparrow\!\!\!
    \widetilde{\mathbb{W}}$ to be a collection of all pairs $T_{\theta,v}(\bx_0) \in \T[B(\bx_0,r)]$ such that there exists a non-empty set $\widetilde{T}_{\tilde{\theta},\tilde{v}}(\widetilde{\bx}_0) \in \widetilde{\mathbb{W}}$ satisfying
\begin{enumerate}
    \item $\mathrm{dist}(\tilde{\theta},{\theta}) \lesssim \rho^{-1/2}$
    \item $\mathrm{dist}(\widetilde{T}_{\tilde{\theta},\tilde{v}}(\widetilde{\bx}_0),T_{\theta,v}(\bx_0) \cap B(\widetilde{\bx}_0,\rho)) \lesssim r^{1/2+\delta}$.
\end{enumerate}

\end{defi}

\noindent This definition naturally appears when comparing wave packets at different scales by the following property: For every $\widetilde{\mathbb{W}} \subset \widetilde{\T}[B(\widetilde{\bx}_0,\rho)]$ and $g \in L^1$
\begin{equation}\label{reverselemma}
    g|_{\widetilde{\mathbb{W}}} = (g|_{\uparrow
    \widetilde{\mathbb{W}}})|_{\widetilde{\mathbb{W}}}+\mathrm{RapDec}(r)\|g\|_2.
\end{equation}
It can be thought of as a reverse version of Lemma \ref{1861}. Since the proof is straightforward, we leave out the details.

\begin{lem}\label{intersectionofsets}
For every set $\W_1,\W_2 \subset \T[B(\bx_0,r)]$ and function $g$, it holds that
\begin{equation}
\|(g|_{\W_1})|_{\W_2}\|_{L^2} \lesssim  \|g|_{{\W_1} \cap {\W_2^*}}\|_{L^2}+\mathrm{RapDec}(r)\|g\|_2.
\end{equation}
\end{lem}

\begin{proof}
We split our function into two parts:
\begin{equation}
    (g|_{\W_1})|_{\W_2}=(g|_{\W_1 \cap \W_2^* })|_{\W_2}+
    (g|_{\W_1 \setminus \W_2^* })|_{\W_2}.
\end{equation}
We apply Lemma \ref{210103lem5_1} to the function $g|_{\W_1 \cap \W_2^*}$ and obtain
\begin{equation}
    \|(g|_{\W_1 \cap \W_2^* })|_{\W_2}\|_2 \lesssim
    \|g|_{\W_1 \cap \W_2^* }\|_2.
\end{equation}
On the other hand, by the proof of Lemma \ref{1861}, we see that \begin{equation} 
\|(g|_{\W_1 \setminus \W_2^* })|_{\W_2}\|_2 = \mathrm{RapDec}(r)\|g\|_2.
\end{equation}
This completes the proof.
\end{proof}

\begin{lem}\label{09lem96}
Let $Z=Z(P_1,\ldots,P_{n-m})$ be a transverse complete intersection, with $\mathrm{deg}P_j \leq d$. Let $b \in \R^n$ with $|b| \lesssim r^{1/2+\delta_m}$. Suppose that $B(\widetilde{\bx}_0,\rho) \subset B(\bx_0,r)$. If $g$ is concentrated on wave packets from $\T_Z[B(\bx_0,r)]$ and $\widetilde{\mathbb{W}} \subset \widetilde{\T}_{b}[B(\widetilde{\bx}_0,\rho)]$, then
\begin{equation}
    \|{g}|_{\widetilde{\mathbb{W}}}\|_2^2 \lesssim_d r^{O(\delta_m)}\Big(\frac{r}{\rho}\Big)^{-\frac{n-m}{2}}\|g|_{\uparrow \widetilde{\mathbb{W}}}\|_2^2+\mathrm{RapDec}(r)\|g\|_2^2.
\end{equation}
\end{lem}

\noindent The above lemma is a corollary of Lemma \ref{transverse-equidistribution} and \ref{intersectionofsets}. The same lemma for the Fourier extension operator for paraboloid is stated and proved in \cite{hickman2020note}. We refer to Lemma 5.3 of \cite{hickman2020note} for the details of the proof.

\begin{lem}\label{probab}
Let $Z=Z(P_1,\ldots,P_{n-m})$ be a transverse complete intersection, with $\mathrm{deg}P_j \leq d$.
Suppose that $B(\widetilde{\bx}_0,\rho) \subset B(\bx_0,r)$ and $B(\widetilde{\bx}_0,\rho) \cap N_{\rho^{1/2+\delta_m}}Z \neq \varnothing$. Let $g$ be concentrated on wave packets from $\T_Z[B(\bx_0,r)]$. Then there are a set of translates $\mathcal{B} \subset B(0,2r^{1/2+\delta_m})$ and functions $\{g_b\}_{b \in \mathcal{B}}$ such that
\begin{equation}\label{first96}
    \|H^{\lambda}g\|_{\mathrm{BL}_{k,A}^p(B(\tilde{\bx}_0,\rho))}^p \lesssim_d (\log{r})^{2p}
    \sum_{b \in \mathcal{B} }\|{H}^{\lambda}{g}_b\|_{\mathrm{BL}_{k,A}^p(B({\tilde{\bx}_0},\rho) \cap N_{\rho^{1/2+\delta_m}}(Z+b) )}^p+\mathrm{RapDec}(r)\|g\|_2^p,
\end{equation}
\begin{equation}\label{second97}
    \sum_{b \in \mathcal{B} }
    \|{g}_b\|_2^2 \lesssim_d \|g\|_2^2,
\end{equation}
\begin{equation}
   \# \mathcal{B} \lesssim_d (r/\rho)^{n(1/2+\delta_m)},
\end{equation}
and
\begin{equation}
    g_b:=g|_{\widetilde{\T}'_b[B(\widetilde{\bx}_0,\rho)]} \text{ \, for some set \, $\widetilde{\T}'_b[B(\widetilde{\bx}_0,\rho)] \subset \widetilde{\T}_b$}.\footnote{We refer to \eqref{0224718} for the explicit definition of $\widetilde{\T}_b'[B(\widetilde{\bx}_0,\rho)]$}
\end{equation}
\end{lem}

The proof of Lemma \ref{probab} requires the lemma below. 

\begin{lem}\label{01017lem81}
Let $\rho\leq r/2$ and $Z\subset\ZR^n$ be a transverse complete intersection.
Let $T_{\theta,v}\in\ZT_Z[B(\bx_0,r)]$ and $b\in B(0,2r^{1/2+\de_m})$. If $\widetilde{T}_{\ti\theta,\ti v}(\widetilde{\bx}_0)\in \widetilde{\ZT}_{\theta,v}[B(\widetilde{\bx}_0,\rho)]$ satisfies
    \begin{equation}
        \wt T_{\ti\theta,\ti v}(\widetilde{\bx}_0)\cap N_{\rho^{1/2+\de_m}/2}(Z+b)\neq \varnothing,
    \end{equation}
    then $\widetilde{T}_{\ti\theta,\ti v}\in\widetilde{\ZT}_{b}[B(\widetilde{\bx}_0,\rho)]$.
\end{lem}

\noindent The rigorous proof of Lemma  \ref{01017lem81} is quite technical, and we refer to the proof of Proposition 9.2 of \cite{MR4047925} for the details. We do not reproduce the proof here.

\begin{proof}[Proof of Lemma \ref{probab}]
We first apply
Lemma 10.5 of \cite{MR4047925}\footnote{We cannot directly apply the lemma because our operator is not of the normal form. However, one can prove the lemma for our operator by following the same argument. We leave out the details here.}. Then there exist a finite set $\mathcal{B} \subset B(0,2r^{1/2+\delta_m})$ with cardinality at most $O((r/\rho)^{n(1/2+\delta_m)})$ and a collection $\mathcal{B}'$ of finitely overlapping $K^2$-balls $B_{K^2}$ intersecting $B(\widetilde{\bx}_0,\rho)$ such that
\begin{equation}\label{010798}
    \|H^{\lambda}g\|_{\mathrm{BL}_{k,A}^p(B(
    \widetilde{\bx}_0,\rho))}^p \lesssim (\log{r})^{2p} \sum_{B_{K^2} \in \mathcal{B}' }
    \mu_{H^{\lambda}g }(B_{K^2}),
\end{equation}
and for each $B_{K^2} \in \mathcal{B}'$ the following holds: there exists some $b \in \mathcal{B}$ such that 
\begin{equation}\label{010799}
    B_{K^2} \subset N_{\rho^{1/2+\delta_m }/2}(Z+b),
\end{equation}
and there exist at most $O(1)$ vectors $b\in \mathcal{B}$ for which  
\begin{equation}\label{0110910}
    B_{K^2} \cap N_{\rho^{1/2+\delta_m }}(Z+b) \neq \varnothing.
\end{equation}
For every $b \in \mathcal{B}$, we let $\mathcal{B}_b'$ denote the collection of all the balls $B_{K^2}$ satisfying \eqref{010799}. By \eqref{010798} and \eqref{010799}, we know that
\begin{equation}\label{middle911}
    \|H^{\lambda}g\|_{\mathrm{BL}_{k,A}^p(B(
    \widetilde{\bx}_0,\rho))}^p \lesssim (\log{r})^{2p} 
    \sum_{b \in \mathcal{B}}
    \sum_{B_{K^2} \in \mathcal{B}_b' }
    \mu_{H^{\lambda}g }(B_{K^2}).
\end{equation}
We define
\begin{equation}\label{0224718}
    \widetilde{\T}_b'[B(\widetilde{\bx}_0,\rho)]:=\{\widetilde{T} \in \bigcup_{{T}_{\theta,v} \in \T_Z[B(\bx_0,r)] } \widetilde{\T}_{\theta,v}[B(\widetilde{\bx}_0,\rho)] : \widetilde{T} \cap \bigcup_{B_{K^2} \in \mathcal{B}_b' }B_{K^2} \neq \varnothing   \}.
\end{equation}
Since every element of $\widetilde{\T}_b'[B(\widetilde{\bx}_0,\rho)]$ intersects $N_{\rho^{1/2+\delta_m}/2}(Z+b)$, by Lemma \ref{01017lem81}, we know that $\widetilde{\T}_b'[B(\widetilde{\bx}_0,\rho)] \subset \widetilde{\T}_{b}[B(\widetilde{\bx}_0,\rho)]$. Let us define $g_b:=g|_{\widetilde{\T}_b'[B(\widetilde{\bx}_0,\rho)]}$. By the proof of Lemma \ref{1861}, the construction of the collection $\widetilde{\T}_b'[B(\widetilde{\bx}_0,\rho)]$, and the triangle inequality of the broad norm (Lemma 6.2 of \cite{MR4047925}), we also know that 
\begin{equation}
\begin{split}
    \mu_{H^{\lambda}g }(B_{K^2}) \lesssim 
    \mu_{ H^{\lambda}g_{b} }(B_{K^2}) +\mathrm{RapDec}(r)\|g\|_2^p
\end{split}
\end{equation}
for every $B_{K^2} \in \mathcal{B}_b'$. Therefore, by combining the above inequality with \eqref{middle911} and the cardinality condition on $\mathcal{B}$, we obtain the first inequality \eqref{first96}.

Let us now show the second inequality \eqref{second97}. By Lemma \ref{210103lem5_1} and \eqref{0110910}, we obtain \begin{equation}
\begin{split}
    \sum_{b \in \mathcal{B}}\|g_b\|_2^2 &= \sum_{b \in \mathcal{B}} \big\|\sum_{\widetilde{T}\in \widetilde{\T}_{b}'[B(\widetilde{\bx}_0,\rho)] }g_{\widetilde{T}}     \big\|_2^2 \\&
    \lesssim 
     \sum_{b \in \mathcal{B}}
     \sum_{\widetilde{T}\in \widetilde{\T}_{b}'[B(\widetilde{\bx}_0,\rho)] }
     \|g_{\widetilde{T}}\|_2^2
     \lesssim
     \sum_{\widetilde{T} \in \widetilde{\T}[B(\widetilde{\bx}_0,\rho)] }\|g_{\widetilde{T}}\|_2^2  \lesssim \|g\|_2^2.
\end{split}
\end{equation}
This completes the proof.
\end{proof}

\section{Finding polynomial structures}\label{f_section8}

In this and the next section, we build a big algorithm to break down the quantity $\|H^{\lambda} g\|_{\BLka^p(B_{R})}$ in our main estimate \eqref{main-esti}, and use it to prove our main result, Theorem \ref{201204thm5_1}. The whole algorithm is involved, so we split it into two smaller ones. This section is devoted for the first one, and we follow Section 6.1 in \cite{hickman2020note}. See also Section 9 in \cite{HR2019}.
\medskip

Recall the following admissible parameters introduced in \eqref{constants_z} 
\begin{equation}
    \epsilon^C \ll \delta \ll \delta_{n} \ll \delta_{n-1} \ll \cdots \ll \delta_1 \ll \epsilon_{\circ} \ll \epsilon.
\end{equation}
We define $\tilde{\delta}_{m-1}$ to be
\begin{equation}
    (1-\tilde{\delta}_{m-1})(1/2+\delta_{m-1})=1/2+\delta_m.
\end{equation}
Notice that $\delta_{m-1}/2 \leq \tilde{\delta}_{m-1} \leq 2 \delta_{m-1}$.
This constant $\tilde{\delta}_{m-1}$ is introduced for some technical reasons and plays a minor role in the proof.

\subsection{Polynomial partitioning lemma}\label{1221.131}

For a polynomial $P$, we denote by $\mathrm{cell}(P)$ a collection of the connected components of $\R^n \setminus Z(P)$. We state the polynoial partitioning lemma used in \cite{guth2018}.

\begin{lem}[\cite{guth2018}]\label{polynomiallemma}
Let $r \gg 1$ and $d,m$ be positive integers. Let $0<\delta_m \ll 1$. Suppose that
 $Z=Z(P_1,\ldots,P_{n-m})$ is a $m$-dimensional transverse complete intersection, with $\mathrm{deg}P_j \leq d$. Suppose that $F \in L^1(\R^n)$ is non-negative and supported on $B_r \cap N_{r^{1/2+\delta_m}}Z$. Then at least one of the following cases hold true:
\begin{enumerate}
    \item (Cellular case) There exists a polynomial $P$ of degree at most $O(d)$ with the following properties:
    \begin{enumerate}
        \item $\#\mathrm{cell}(P) \sim d^m$
        \item For each $O' \in \mathrm{cell}(P)$, define the shrunken cells $O:=O' \setminus N_{R^{1/2+\delta_m}}(Z(P))$. Then for every $O$
          \begin{equation*}
        \int_{O}F \sim d^{-m}\int_{\R^n}F.
    \end{equation*}
    Moreover, the number of the shrunken cells $O$ is comparable to $d^m$, and the diameter of $O$ is smaller than $r/2$.
    \end{enumerate}

    \item (Algebraic case) There exists $(m-1)$-dimensional transverse complete intersection $Y$, defined using polynomials of degree at most $O(d)$ such that
    \begin{equation*}
        \int_{B_r \cap N_{r^{1/2+\delta_m}}(Z) }F \lesssim 
        \int_{B_r \cap N_{r^{1/2+\delta_m}}(Y) }F.
    \end{equation*}
\end{enumerate}
\end{lem}

Suppose that a function $g$ and a grain $(Z,B_r)$ of dimension $m$ are given. We will apply the above lemma to the function 
\begin{equation}
    F(x)=\sum_{B_{K^2}} \mu_{H^{\lambda}g}(B_{K^2})\frac{1}{|B_{K^2}|} 1_{{B_{K^2} \cap B_r \cap N_{r^{1/2+\delta_m} }(Z) } }(x)
\end{equation}
with a number $d$. This number $d$ will be much larger than the degrees of polynomial defining $Z$. If the cellular case holds true, then
\begin{equation}
    \|H^{\lambda}g\|_{\mathrm{BL}_{k,A}^p(B_r \cap N_{r^{1/2+\delta_m}}(Z))}^p \lesssim d^m \|H^{\lambda}g\|_{\mathrm{BL}_{k,A}^p( O  )}^p
\end{equation}
for all $O \in \mathcal{O}$. Here, every $O \in \mathcal{O}$ has a diameter at most $r/2$ and $\# \mathcal{O} \sim d^m$ for some sufficiently large number $d$, which will be determined later.
If the algebraic case holds true, then
\begin{equation}
    \|H^{\lambda}g\|_{\mathrm{BL}_{k,A}^p(B_r \cap N_{r^{1/2+\delta_m}}(Z))} \lesssim\|H^{\lambda}g\|_{\mathrm{BL}_{k,A}^p(B_r \cap N_{r^{1/2+\delta_m}}(Y))}
\end{equation}
for some $(m-1)$-dimensional transverse complete intersection $Y$, defined using polynomials of degree at most $d$. 

\subsection{The first algorithm}

Let us illustrate the first algorithm. This algorithm is the counterpart of the first algorithm in Section 6.1 of \cite{hickman2020note}. Let $1\leq m \leq n$.
\medskip

\noindent\underline{\texttt{Input}} 
The algorithm \texttt{[alg 1]} takes as its input: 
\begin{itemize}
    \item A grain $\big(Z, B(\bx_0,r)\big)$ of dimension $m$ with $B(\bx_0,r) \subset [-3C_n\lambda,3C_n\lambda]^{n-1} \times [R/C_n,C_n\lambda]$.
    \item A function $g \in L^1$ which is tangent to the grain $\big(Z, B(\bx_0,r)\big)$.
    \item An admissible large integer $A \in \N$.
\end{itemize}
\medskip

\noindent\underline{\texttt{Output}}
The $j$th stage of \texttt{[alg 1]} outputs:
\begin{itemize}
    \item A choice of spatial scale $\rho_j \geq 1$ 
    \item Certain integers $\#_{\bta}(j), \#_{\btc}(j) \in \N_0$ satisfying $\#_{\bta}(j) + \#_{\btc}(j) = j$.\footnote{The integers $\#_{\bta}(j)$ and $ \#_{\btc}(j)$ indicate the number of occurrences of algebraic cases and cellular cases, respectively.}
  \item A family of subsets $\mathcal{O}_j$ of $\R^n$ referred to as \emph{cells}. Each cell $O_j \in \mathcal{O}_j$ is contained in some $\rho_j$-ball $B_{O_j} := B(\bx_{O_j}, \rho_j)$. 
\item A collection of functions $\{ g_{O_j}\}_{O_j \in \mathcal{O}_j}$. For each cell $O_j$ there is a translate $Z_{O_j} := Z + \by_{O_j}$ such that  $g_{O_j}$ is tangent to the grain $(Z_{O_j},B_{O_j})$. 
  \item A large integer $d \in \mathbb{N}$ which depends only on the admissible parameters and $\deg Z$.
  \end{itemize}

All the outputs will be constructed so that for some constants
\begin{equation}\label{1221.136}
    C^{\mathrm{I}}_{j,\delta}(d,r), \; C^{\mathrm{II}}_{j,\delta}(d), \; C^{\mathrm{III}}_{j,\delta}(d,r), \; C^{\mathrm{IV}}_{j,\delta}(d,r) \lesssim_{d,\delta} r^{\epsilon_{\circ}} d^{\#_{\stc}(j)\delta},
\end{equation}
which will be defined explicitly in \eqref{HZconstants},
and $A_j=2^{-\#_{\sta}(j)}A$ the following properties hold true:
\medskip

\paragraph*{\underline{Property I}} For some fixed $N \in \mathbb{N}$
\begin{equation}
      \|H^{\lambda}g\|_{\mathrm{BL}_{k,A}^p(B(\bx_0,r))}^p \leq C^{\mathrm{I}}_{j,\delta}(d,r) \sum_{O_{j} \in \mathcal{O}_{j}} \|H^{\lambda}g_{O_j}\|_{\mathrm{BL}_{k,A_j}^p(O_{j})}^p + j r^{-N} \|g\|_{L^2(B^{n-1})}^p
\end{equation}

\paragraph*{\underline{Property II}} 
\begin{equation}
    \sum_{O_{j} \in \mathcal{O}_{j}} \|g_{O_j}\|_2^2 \leq C^{\mathrm{II}}_{j,\delta}(d)d^{\#_{\stc}(j)}  \|g\|_2^2.
\end{equation}

\paragraph*{\underline{Property III}} 
\begin{equation}
    \|g_{O_{j}}\|_{L^2(B^{n-1})}^2 \leq C_{j,\delta}^{\mathrm{III}}(d, r) \Big(\frac{r}{\rho_{j}}\Big)^{-\frac{n-m}{2}} d^{-\#_{\stc}(j)(m-1)}  \|g\|_{L^2(B^{n-1})}^2.
\end{equation}

To state the last property, we need to introduce some notations.
For $\widetilde{\W} \subseteq \widetilde{\T}[B({\bx}_{O_j},\rho_j)]$ let $\uparrow^j \! \widetilde{\W}$ denote the set of wave packets $T_{\theta,v}(\bx_0) \in \T[B(\bx_0,r)]$ satisfying 
\begin{equation}\label{uparrowj}
    \dist(\theta,\widetilde{\theta}_j) \leq c_j \rho_j^{-1/2} \quad \textrm{and} \quad \dist\big(\widetilde{T}_{\tilde{\theta}_j,\tilde{v}_j}({\bx}_{O_j}),\, T_{\theta,v}(\bx_0) \cap B_{O_j}\big) \leq c_j r^{1/2 + \delta}
\end{equation}
for some $\widetilde{T}_{\tilde{\theta}_j,\tilde{v}_j}({\bx}_{O_j}) \in \widetilde{\W}$. Here $\{c_j\}_{j=0}^{\infty}$ is a increasing positive sequence that each of which is bounded above by an absolute constant $C_{\circ}$. In fact, one can take $c_j=C_\circ(1-2^{-j/2})$ for some big constant $C_\circ$. We introduce this sequence only for the sake of rigorousness. Heuristically, one can take $c_j=1$.

\paragraph*{\underline{Property IV}}\label{1219P4} For any $\widetilde{\W} \subseteq \widetilde{\T}[B(\bx_{O_j}, \rho_j)]$, each $g_{O_j}$ satisfies
\begin{equation}
 \|{g}_{O_{j}}|_{\widetilde{\W}}\|_{2}^2 \leq C_{j,\delta}^{\mathrm{IV}}(d, r) \Big(\frac{r}{\rho_{j}}\Big)^{-\frac{n-m}{2}}\|g|_{\uparrow^j\widetilde{\W}}\|_{2}^2 + \mathrm{RapDec}(r)\|g\|_{L^2}^2.
\end{equation}
\hfill

\paragraph*{\underline{\texttt{Stopping conditions}}} 

Suppose that we have the $j$th stage of \texttt{[alg 1]} outputs. We stop our algorithm if one of \texttt{[tiny]} and \texttt{[tang]} holds true. 
\begin{itemize}
\item[\texttt{Stop:[tiny]}] The algorithm terminates if $\rho_j \leq r^{\tilde{\delta}_{m-1}}$.
\end{itemize}
\begin{itemize}
\item[\texttt{Stop:[tang]}]Let $C_{\textrm{\texttt{tang}}}$ and $C_{\mathrm{alg}}$ be large, fixed dimensional constants and $\tilde{\rho} := \rho_j^{1 - \tilde{\delta}_{m-1}}$. The algorithm terminates if there exist
\end{itemize}
\begin{itemize}
    \item $\mathcal{S}$ a collection of grains $(S,B_{\tilde{\rho}})$ of dimension $m-1$, scale $\tilde{\rho}$ and degree at most $C_{\mathrm{alg}}d$;
    \item An assignment of a function $g_{(S,B_{\tilde{\rho}})}$ to each $(S,B_{\tilde{\rho}}) \in \mathcal{S}$ which is tangent to $(S,B_{\tilde{\rho}})$
\end{itemize}
such that the following four conditions hold:  \medskip

\paragraph*{\underline{Condition I}}
\begin{equation*}
    \sum_{O_j \in \mathcal{O}_j} \|H^{\lambda}g_{O_j}\|_{\mathrm{BL}_{k,A_j}^p(O_{j})}^p \leq C_{\textrm{\texttt{tang}}}  \sum_{(S,B_{\tilde{\rho}}) \in \mathcal{S}} \|H^{\lambda}g_{(S,B_{\tilde{\rho}})}\|_{\mathrm{BL}_{k,A_j/2}^p(B_{\tilde{\rho}})}^p.
\end{equation*}

\paragraph*{\underline{Condition II}}

\begin{equation*}
   \sum_{(S,B_{\tilde{\rho}}) \in \mathcal{S}}\|g_{(S,B_{\tilde{\rho}})}\|_{L^2(B^{n-1})}^2  \leq C_{\textrm{\texttt{tang}}}r^{n\tilde{\delta}_m}\sum_{O_j \in \mathcal{O}_j}\|g_{O_j}\|_{L^2(B^{n-1})}^2.
\end{equation*}

\paragraph*{\underline{Condition III}} 
\begin{equation*}
    \max_{(S,B_{\tilde{\rho}}) \in \mathcal{S}}\|g_{(S,B_{\tilde{\rho}})}\|_{2}^2 \leq C_{\textrm{\texttt{tang}}}\max_{O_{j} \in \mathcal{O}_{j} }\|g_{O_j}\|_{2}^2.
\end{equation*}

To state the last condition, we need to introduce some notations. Let us denote by $\tilde{\bx}$ the center of $B_{\tilde{\rho}}$. Given $\widetilde{\W} \subset \widetilde{\T}[B(\tilde{\bx},\tilde{\rho}) ]$, we denote by
 $\uparrow\! \widetilde{\W}$ the set of all $T_{\theta, v}(\bx_{O_j}) \in \T[B(\bx_{O_j},\rho_j)]$ for which there exists some $\widetilde{T}_{\tilde{\theta},\tilde{v}}(\tilde{\bx}) \in \widetilde{\W}$ satisfying
\begin{equation}\label{uparrow1}
    \dist(\tilde{\theta},\theta) \lesssim \tilde{\rho}^{-1/2}, \quad \dist(\widetilde{T}_{\tilde{\theta},\tilde{v}}(\tilde{\bx}),\, T_{\theta,v}(\bx_{O_j}) \cap B(\tilde{\bx}, \tilde{\rho})\big) \lesssim \rho_j^{1/2+\delta}.
\end{equation} 

\paragraph*{\underline{Condition IV}} Given $(S, B(\tilde{\bx},\tilde{\rho})) \in \mathcal{S}$ there exists some $O_j \in \mathcal{O}_j$ such that 
\begin{equation*}
    \|{g}_{(S,B_{\tilde{\rho}})}|_{\widetilde{\W}}\|_{2}^2 \leq C_{\textrm{\texttt{tang}}} \|g_{O_j}|_{\uparrow\widetilde{\W}}\|_{2}^2
    \end{equation*}
    holds for all $\widetilde{\W} \subseteq \widetilde{\T}[B(\tilde{\bx},\tilde{\rho})]$.
    \\

\subsection{A construction of outputs in the first algorithm}

In this subsection, we construct outputs and show that they satisfy the desired properties. Let $d$ be a sufficiently large number, which will be determined later.
We first define the constants
\begin{equation}\label{HZconstants}
\begin{split}
    C_{j,\delta}^{\mathrm{I}}(d,r)&:=d^{\#_{\stc}(j)\delta}(\log{r})^{2p\#_{\sta}(j)(1+\delta) },
    \\
    C_{j,\delta}^{\mathrm{II}}(d)&:=
    d^{\#_{\stc}(j)\delta+n\#_{\sta}(j)(1+\delta) },
    \\
    C_{j,\delta}^{\mathrm{III}}(d,r)&:=
    d^{\#_{\stc}(j)\delta+\#_{\sta}(j)\delta }r^{C\#_{\sta}(j)\delta_m },
    \\
    C_{j,\delta}^{\mathrm{IV}}(d,r)&:=d^{j\delta}r^{C\#_{\sta}(j)\delta_m },
\end{split}
\end{equation}
and define the initial outputs
\begin{itemize}
    \item $\rho_0:=r$, $\#_{\bta}(0)=\#_{\btc}(0):=0$.
    \item $\mathcal{O}_0:=\{\R^n \}$ and $Z_{O_0}=Z$.
    \item $g_{\R^n}:=g$.
\end{itemize}
Note that \eqref{1221.136} holds true. Properties I, II, III, and IV are also vacuously true with the initial outputs.

Let us now assume that we have the outputs of $j$th stage and the stopping conditions fail. We need to construct the outputs of $(j+1)$th stage satisfying all the desired properties. For each function $g_{O_j}$, by the tangency assumption and Property III of $g_{O_j}$, we obtain
\begin{equation}
\label{tangency-assumption}
    \|H^{\lambda}g_{O_j}\|_{\mathrm{BL}_{k,A_j}^p(O_{j})}
    =
    \|H^{\lambda}g_{O_j}\|_{\mathrm{BL}_{k,A_j}^p(O_{j} \cap N_{\rho_j^{1/2+\delta_m}(Z+\by_{O_j} )})}+ \mathrm{RapDec}(r)\|g\|_2^p.
\end{equation}
Hence, we can apply the polynomial partitioning lemma to the first term on the right hand side as in the discussion of Section \ref{1221.131}. Let us denote by $\mathcal{O}_{j,\mathrm{cell}}$ the subcollection of $\mathcal{O}_j$ consisting of all the cells for which the cellular case holds. We define $\mathcal{O}_{j,\mathrm{alg}}:= \mathcal{O}_j \setminus \mathcal{O}_{j,\mathrm{cell}}$. By Property I of the $j$th outputs, we have
\begin{equation}\label{1221.1310}
\begin{split}
    \|H^{\lambda}g\|_{\mathrm{BL}_{k,A_j}^p(B(\bx_0,r) )}^p
    &\leq C_{j,\delta}^{\mathrm{I}}(d,r)
    \Big( \sum_{O_j \in \mathcal{O}_{j, \mathrm{cell} }} \|H^{\lambda}g_{O_j}\|_{\mathrm{BL}_{k,A_j}^p(O_{j})}^p+ 
    \sum_{O_j \in \mathcal{O}_{j, \mathrm{alg} }} \|H^{\lambda}g_{O_j}\|_{\mathrm{BL}_{k,A_j}^p(O_{j})}^p
    \Big)
    \\&
    +\big(jr^{-N}+\mathrm{RapDec}(r)\big)\|g\|_2^p.
\end{split}
\end{equation}
There are two cases depending on which term  dominates.
\subsubsection*{Cellular-dominant case}

Consider the case that the first term on the right hand side of \eqref{1221.1310} dominates the second term. In this case, we have
\begin{equation}\label{0120915}
    \|H^{\lambda}g\|_{\mathrm{BL}_{k,A_j}^p(B(\bx_0,r))}^p
    \leq 2C_{j,\delta}^{\mathrm{I}}(d,r)
    \sum_{O_j \in \mathcal{O}_{j, \mathrm{cell} }} \|H^{\lambda}g_{O_j}\|_{\mathrm{BL}_{k,A_j}^p(O_{j})}^p+\big(jr^{-N}+\mathrm{RapDec}(r)\big)\|g\|_2^p.
\end{equation}
We take the outputs 
\begin{equation}
    \rho_{j+1}:=\rho_j/2, \;\;\; \#_{\btc}(j+1)=\#_{\btc}(j)+1, \;\;\; \#_{\bta}(j+1)=\#_{\bta}(j).
\end{equation}

Let us now construct the collection of step $j+1$ cells $\mathcal{O}_{j+1}$ and $\{ g_{{O}_{j+1}}\}_{O_{j+1} \in \mathcal{O}_{j+1}}$ so that they satisfy the desired properties mentioned in the algorithm.  Since we are in the cellular case, for every $O_j \in \mathcal{O}_{j,\mathrm{cell}}$, we have a polynomial $P$, depending on the choice of $O_j$, of degree at most $d$ and we have the following properties:
\begin{itemize}
    \item The number of connected components of $\R^n \setminus Z(P)$ is comparable to $d^m$.
    \item For each connected component $O'$ of $\R^n \setminus Z(P)$, define the shrunken cell $O=O'\setminus N_{R^{1/2+\delta_m}}(Z(P))$. Then
    \begin{equation}
            \|H^{\lambda}g_{O_j}\|_{\mathrm{BL}_{k,A_j}^p(O_j)}^p \lesssim d^m \|H^{\lambda}g_{O_j}\|_{\mathrm{BL}_{k,A_j}^p( O  )}^p.
    \end{equation}
    \item The number of the shrunken cells $O$ is comparable to $d^m$.
    \item The diameter of $O$ is at most $r/2$.
\end{itemize}
 We denote by
$\mathcal{O}(O_j)$ the collection
of the shrunken cells. As mentioned above, we know that $\# \mathcal{O}(O_j) \sim d^m$ and each
$O \in \mathcal{O}(O_j)$ has diameter at most $\rho_{j+1}$, and
\begin{equation}
    \|H^{\lambda}g_{O_j}\|_{\mathrm{BL}_{k,A_j}^p(O_j)}^p \lesssim d^m \|H^{\lambda}g_{O_j}\|_{\mathrm{BL}_{k,A_j}^p( O  )}^p,
\end{equation}
and thus,
\begin{equation}\label{021820}
    \|H^{\lambda}g_{O_j}\|_{\mathrm{BL}_{k,A_j}^p(O_j)}^p \lesssim \sum_{O \in \mathcal{O}} \|H^{\lambda}g_{O_j}\|_{\mathrm{BL}_{k,A_j}^p( O  )}^p.
\end{equation}
Let us denote by $B(\bx_{O},\rho_{j+1})$ the ball containing $O$. Define
\begin{equation}
    \T_{O}
    := \T_{O}[B({\bx_{O_j}},\rho_{j})]:=\{T \in \T_{Z+\by_{O_j} }[B(\bx_{O_j},\rho_j)]: T \cap O \neq \varnothing   \},
\end{equation}
\begin{equation}
    \widetilde{\T}_{O}:= \widetilde{\T}_{O}[B(\bx_{O},\rho_{j+1})]:=\bigcup_{T_{\theta,v} \in \T_{O} }\widetilde{\T}_{\theta,v}[B(\bx_{O},\rho_{j+1})],
\end{equation}
where $\widetilde{\T}_{\theta,v}[B(\bx_{O},\rho_{j+1})]$ was defined in \eqref{smalltubesinlargetube}. By Lemma \ref{1861}, it holds that
\begin{equation}\label{21021801}
    \|H^{\lambda}g_{O_j}\|_{\mathrm{BL}_{k,A_j}^p( O  )} \lesssim 
    \|H^{\lambda}(( g_{O_j}|_{\T_{O}})|_{\widetilde{\T}_{O} })\|_{\mathrm{BL}_{k,A_j}^p( O  )}+\mathrm{RapDec}(r)\|g\|_2.
\end{equation}
By Lemma \ref{0128lemma} and a simple calculation, we know that
\begin{equation}
\begin{split}
    \bigcup_{\widetilde{T} \in \widetilde{\T}_{O} } \widetilde{T} &\subset N_{2\rho_j^{1/2+\delta_m}}(Z+\by_{O_j})
    \\&
    \subset \bigcup_{b \in \Z^{n}: |b| \lesssim 1}
    N_{\rho_{j+1}^{1/2+\delta_m}/2}\big(Z+\by_{O_j}+({\rho_{j+1}^{1/2+\delta_m}}/{2}) b \big).
    \end{split}
\end{equation}
Thus, every tube $\widetilde{T} \in \widetilde{\T}_{O}$ intersects $N_{\rho_{j+1}^{1/2+\delta_m}/2}\big(Z+\by_{O_j}+({\rho_{j+1}^{1/2+\delta_m}}/{2}) b \big)$ for some $b$ depending on the choice of $\widetilde{T}$. We now write 
\begin{equation}\label{21021802}
    \widetilde{\T}_{O}=\bigsqcup_{b \in \Z^n:|b| \lesssim 1}\widetilde{\T}_{O,b}
\end{equation}
where $\widetilde{\T}_{O,b}$ is some sub-collection of $\widetilde{\T}_{O}$ satisfying
\begin{equation}
    \widetilde{\T}_{O,b} \subset \{ \widetilde{T} \in \widetilde{\T}_{O}: \widetilde{T} \cap N_{\rho_{j+1}^{1/2+\delta_m}/2}\big(Z+\by_{O_j}+({\rho_{j+1}^{1/2+\delta_m}}/{2}) b \big) \neq \varnothing \}.\footnote{Possibly a tube $\widetilde{T}$ intersects $N_{\rho_{j+1}^{1/2+\delta_m}/2}\big(Z+\by_{O_j}+({\rho_{j+1}^{1/2+\delta_m}}/{2}) b \big)$ for many $b$. We simply choose one out of them.  }
\end{equation}
 Define
\begin{equation}
    \mathcal{O}_{j+1}(O_j):=\Big\{O \cap N_{\rho_{j+1}^{1/2+\delta_m}/2}\big(Z+\by_{O_j}+(\rho_{j+1}^{1/2+\delta_m}/2)b \big): O \in \mathcal{O}(O_j),\;\; b\in\ZZ^n~\text{with}~ |b| \lesssim 1 \Big\},
\end{equation}
\begin{equation}
    \mathcal{O}_{j+1}:= \bigcup_{O_j \in \mathcal{O}_{j,\mathrm{cell}} }\mathcal{O}_{j+1}( O_j).
\end{equation}
For every $O_{j+1} \in \mathcal{O}_{j+1}$, there exist corresponding  $O_j$, $b$, and $O$. Define a translate of the variety
\begin{equation}
    Z_{O_{j+1}}:=Z+\by_{O_{j+1}}:=Z+\by_{O_j}+(\rho_{j+1}^{1/2+\delta_m}/2)b
\end{equation} 
and a function
\begin{equation}\label{21021803}
    g_{O_{j+1}}:=\big( (g_{O_j})|_{\T_{O}}\big)\big|_{\widetilde{\T}_{O,b}}.
\end{equation}
Let us define $\bx_{O_{j+1}}:=\bx_{O}$. By Lemma \ref{01017lem81}, we have $\widetilde{\T}_{O,b} \subset \widetilde{\T}_{Z_{O_{j+1}} }[B({\bx}_{O_{j+1}},\rho_{j+1})]$. 

Via Lemma \ref{210103lem5_1} (the $L^2$ orthogonality), one has
\begin{equation}
\begin{split}
    \sum_{O_{j+1} \in \mathcal{O}_{j+1}(O_j)} \|g_{O_{j+1}}\|_2^2
        &\lesssim \sum_{O \in \mathcal{O}(O_j) }\sum_{b: |b| \lesssim 1}\| (g_{O_j}|_{\T_{O}})|_{\widetilde{\T}_{O,b} }\|_2^2
        \\  & \lesssim 
        \sum_{O\in \mathcal{O}(O_j) }
       \| g_{O_j}|_{\T_{O}}\|_2^2
       \\  & \lesssim \sum_{O \in \mathcal{O}(O_j) }\sum_{T \in \T_{O} }\|(g_{O_j})_T\|_2^2.
\end{split}
\end{equation}
By the fundamental theorem of algebra, we know that each tube $T_{\theta,v} \in \T[B(\bx_{O_j},\rho_j)]$ can intersect at most $O(d)$ cells $O_{j+1} \in \mathcal{O}_{j+1}(O_j)$. Therefore, we further have
\begin{equation}
\begin{split}
     \sum_{O \in \mathcal{O}(O_j) }\sum_{T \in \T_{O} }\|(g_{O_j})_T\|_2^2 &\lesssim
        \sum_{O \in \mathcal{O}(O_j) }
        \sum_{T \in \T[B(\bx_{O_j},\rho_j)] : T \cap O \neq \varnothing }
        \|(g_{O_j})_T\|_2^2
        \\  &\lesssim
        \sum_{T \in \T[B(\bx_{O_j},\rho_j)] }
        \sum_{O \in \mathcal{O}(O_j): T \cap O \neq \varnothing  }
        \|(g_{O_j})_T\|_2^2
        \lesssim d\|g_{O_j}\|_2^2.
\end{split}
\end{equation}
The above two estimates lead to 
\begin{equation}
\label{0218831}
    \sum_{O_{j+1} \in \mathcal{O}_{j+1}(O_j)} \|g_{O_{j+1}}\|_2^2\lesssim d\|g_{O_j}\|_2^2.
\end{equation}
Since we know that $\#\mathcal{O}_{j+1}(O_j) \sim d^m$, by pigeonholing argument, we can take a subcollection of $\mathcal{O}_{j+1}(O_j)$ such that the cardinality is similar to $\mathcal{O}_{j+1}(O_j)$ and every element $O_{j+1}$ of the subcollection satisfies
\begin{equation}\label{0120925}
    \|g_{O_{j+1}}\|_2^2 \lesssim d^{-(m-1)}\|g_{O_j}\|_2^2.   
\end{equation}
By abusing the notation, we still call  such subcollection $\mathcal{O}_{j+1}(O_j)$ and their union $\mathcal{O}_{j+1}$. This completes the construction of our outputs. 
    \medskip
    
It remains to show that our outputs satisfy the desired properties. The function $g_{O_{j+1}}$ is tangent to the grain $(Z_{O_{j+1}},B_{O_{j+1}})$ because of  $\widetilde{\T}_{O,b} \subset \widetilde{\T}_{Z_{O_{j+1}} }[B({\bx_{O_{j+1}}},\rho_{j+1})]$. Property I, II, and III follow from the arguments in page 254--256 of \cite{HR2019}. We follow pages 19--20 of \cite{hickman2020note} for Property IV.
Our proofs are very similar to theirs, so we only give a sketch of the proof.
\medskip

Let us start with Property I. By \eqref{0120915}, \eqref{021820}, \eqref{21021801}, \eqref{21021802} and \eqref{21021803} with the triangle and H\"{o}lder's inequality, we obtain
\begin{equation}
    \|H^{\lambda}g\|_{\mathrm{BL}_{k,A_j}^p(B(\bx_0,r))}^p
    \leq C C_{j,\delta}^{\mathrm{I}}(d,r)
    \sum_{O_{j+1} \in \mathcal{O}_{j+1 }} \|H^{\lambda}g_{O_{j+1}}\|_{\mathrm{BL}_{k,A_{j+1}}^p(O_{j+1})}^p+(j+1)r^{-N}\|g\|_2^p.
\end{equation}
Property I follows by taking $d$ sufficiently large so that $CC_{j,\delta}^{\mathrm{I}}(d,r) \leq C_{j+1,\delta}^{\mathrm{I}}(d,r)$. To prove Property II, we take the sum over $O_j \in \mathcal{O}_j$ to the inequality \eqref{0218831} and obtain
\begin{equation}
    \sum_{O_{j+1} \in \mathcal{O}_{j+1} }\|g_{O_{j+1}}\|_2^2 \lesssim d \sum_{O_j \in \mathcal{O}_j }\|g_{O_j}\|_2^2 \lesssim C_{j,\delta}^{\mathrm{II}}(d)d^{1+\#_{\stc}(j)}\|g\|_2^2,
\end{equation}
where the second inequality follows from Property II of the $j$th outputs. Property II follows by taking $d$ sufficiently large so that $CC_{j,\delta}^{\mathrm{II}}(d) \leq C_{j+1,\delta}^{\mathrm{II}}(d)$.
Similarly, Property III follows from the inequality \eqref{0120925} and Property II of the $j$th outputs. We leave out the details. 

Finally, let us prove Property IV. Recall that the notations $\uparrow^{j}\!\!\widetilde{\mathbb{W}}$ and $\uparrow\!\!\widetilde{\mathbb{W}}$  were introduced in \eqref{uparrowj} and in \eqref{uparrow1}. By an application of Lemma \ref{intersectionofsets}, \eqref{reverselemma}, and Property IV of the $j$th outputs, we obtain
\begin{equation}
\begin{split}
    \|g_{O_{j+1}}|_{\widetilde{\mathbb{W}}}\|_2^2 &\lesssim
    \|g_{O_{j}}|_{\uparrow\widetilde{\mathbb{W}}}\|_2^2+\mathrm{RapDec}(r)\|g\|_2^2
    \\&
    \lesssim 
    C_{j,\delta}^{\mathrm{IV}}(d, r) \Big(\frac{r}{\rho_{j}}\Big)^{-\frac{n-m}{2}}\|g|_{\uparrow^j(\uparrow\widetilde{\W})}\|_{2}^2 + \mathrm{RapDec}(r)\|g\|_{L^2}^2.
\end{split}
\end{equation}
It is proved that $\uparrow^j\!\!(\uparrow\!\! \widetilde{\mathbb{W}}) \subset \uparrow^{j+1}\!\!\widetilde{\mathbb{W}}$ on page 20 of \cite{hickman2020note}. Therefore, by arguing similarly as in the proof of Lemma \ref{intersectionofsets}, and by modifying the constant $c_{j+1}$ appropriately and taking $d$ sufficiently large, we obtain Property IV. This finishes the discussion on the cellular-dominant case.

\subsubsection*{Algebraic-dominant case}
Consider the case that the second term on the right hand side of \eqref{1221.1310} dominates the first term. Recall \eqref{tangency-assumption}: 
\begin{equation}
    \|H^{\lambda}g_{O_j}\|_{\mathrm{BL}_{k,A_j}^p(O_{j})}
    =
    \|H^{\lambda}g_{O_j}\|_{\mathrm{BL}_{k,A_j}^p(O_{j} \cap N_{\rho_j^{1/2+\delta_m}(Z+\by_{O_j} )})}+ \mathrm{RapDec}(r)\|g\|_2^p.
\end{equation}
Since we are in the algebraic case, we have
\begin{equation}\label{091024}
\begin{split}
    \|H^{\lambda}g\|_{\mathrm{BL}_{k,A_j}^p(B_r)}^p
    \leq 2 C_{j,\delta}^{\mathrm{I}}(d,r)
    \sum_{O_j \in \mathcal{O}_{j, \mathrm{alg} }} \|H^{\lambda}g_{O_j}\|_{\mathrm{BL}_{k,A_j}^p(O_{j} \cap N_{\rho^{1/2+\delta_m}}(Y(O_j)) )}^p+(jr^{-N}+\mathrm{RapDec}(r))\|g\|_2^p
\end{split}
\end{equation}
for some $(m-1)$-dimensional transverse complete intersection $Y(O_j)$ that is defined using polynomials of degree at most $d$. Abbreviate $Y(O_j)$ to $Y$ for simplicity. We take the outputs 
\begin{equation}
    \rho_{j+1}:=\rho_j^{1-\tilde{\delta}_{m-1}}, \;\;\; \#_{\btc}(j+1)=\#_{\btc}(j), \;\;\; \#_{\bta}(j+1)=\#_{\bta}(j)+1.
\end{equation}

It remains to define $\mathcal{O}_{j+1}$ and $\{ g_{{O}_{j+1}}\}_{O_{j+1} \in \mathcal{O}_{j+1}}$, and prove the desired properties. Given $O_j \in \mathcal{O}_{j,\mathrm{alg}}$,
we take a collection $\mathcal{B}(O_j)$ of finitely overlapping balls of $\rho_{j+1}$ covering $O_j \cap N_{\rho_j^{1/2+\delta_m}}Y$.
For every $B \in \mathcal{B}(O_j)$, we record tubes intersecting $B$;
\begin{equation}
    \T_B:=\{T \in \T_{Z_{O_j}}[B(\bx_{O_j},\rho_j)]: T \cap B  \neq \varnothing \}.
\end{equation}
Let us define tangent and transverse tubes.

\begin{defi}
We denote by $\T_{B,\mathrm{tang}}$ the collection of all the tubes $T_{\theta,v} \in \T_B$ such that
\begin{enumerate}
    \item $T_{\theta,v} \cap 2B \subset N_{2\rho_{j+1}^{1/2+\delta_{m-1}}}(Y)$;
    \item For every $x \in T_{\theta,v}$ and $y\in Y\cap 2B$ satisfying $|y-x| \lesssim \rho_{j+1}^{1/2+\delta_{m-1}}$, it holds that
    \begin{equation}
        \ang(G^{\lambda}(\omega_{\theta}),T_yY) \lesssim \rho_{j+1}^{-1/2+\delta_{m-1}}.
    \end{equation}
\end{enumerate}
\end{defi}

We define a collection of transverse tubes by $\T_{B,\mathrm{trans}}:=\T_{B}\, \setminus \, \T_{B,\mathrm{tang}}$, and define functions
\begin{equation}\label{841}
    g_{B,\mathrm{tang}}:= (g_{O_j})|_{\T_{B,\mathrm{tang}}},\;\;\;\;
    g_{B,\mathrm{trans}}:=
    (g_{O_j})|_{\T_{B,\mathrm{trans}}}.
\end{equation}
Since for any ball $B \in \mathcal{B}(O_j)$ the wave packets not intersecting $B$ are negligible, we have 
\begin{equation}
    H^{\lambda}g_{O_j}(x)=H^{\lambda}g_{B,\mathrm{tang}}(x)+H^{\lambda}g_{B,\mathrm{trans}}(x)+ \mathrm{RapDec}(r)\|g\|_2, \;\;\;\; x \in B.
\end{equation}
By the finite sub-additivity of the broad norm (See Lemma 6.1 of \cite{MR4047925}), we obtain
\begin{equation}\label{091030}
\begin{split}
&\sum_{O_j \in \mathcal{O}_{j,\mathrm{alg} }}
    \|H^{\lambda}g_{O_j}\|_{\mathrm{BL}_{k,A_j}^p(O_{j} \cap N_{\rho^{1/2+\delta_m}}(Y) )}^p
    \\&
    \sim
    \sum_{O_j \in \mathcal{O}_{j,\mathrm{alg} }} 
    \sum_{B \in \mathcal{B}(O_j) } \|H^{\lambda}g_{O_j}\|_{\mathrm{BL}_{k,A_j}^p(B)}^p
    \\&
    \lesssim
\sum_{O_j \in \mathcal{O}_{j,\mathrm{alg} }}
    \sum_{B \in \mathcal{B}(O_j)}
    \Big(
    \|H^{\lambda}g_{B,\mathrm{tang}}\|_{\mathrm{BL}_{k,A_{j+1}}^p(B)}^p+
     \|H^{\lambda}g_{B,\mathrm{trans}}\|_{\mathrm{BL}_{k,A_{j+1}}^p(B)}^p
    \Big)
    +\mathrm{RapDec}(r)\|g\|_2^p
    .
\end{split}
\end{equation}

Note that by the failure of the stopping condition of \texttt{[tang]}, it holds that
\begin{equation}\label{844}
    \sum_{O_j}
    \sum_{B \in \mathcal{B}(O_j)} \|H^{\lambda}g_{B,\mathrm{tang}}\|_{\mathrm{BL}_{k,A_{j+1}}^p(B)}^p \leq C_{\mathrm{tang} }^{-1}\sum_{O_j \in \mathcal{O}_j}\|H^{\lambda}g_{O_j}\|_{\mathrm{BL}_{k,A_j}^p(O_j)}^p.
\end{equation}
Indeed, we take collections of grains
$
    \mathcal{S}(O_j):=\{ (Y(O_j),B): B \in \mathcal{B}(O_j) \}
$ and $\mathcal{S}:=\bigcup_{O_j}\mathcal{S}(O_j)$, and take functions $g_{(Y,B)}:=g_{B,\mathrm{tang}}$. By Lemma \ref{0128lemma} and \eqref{841}, the function $g_{B,\mathrm{tang}}$ is tangent to $(Y(O_j),B)$. Also, one may see that Condition II, III and IV of the stopping condition are satisfied. By this and the failure of the stopping condition, Condition I fails and it gives \eqref{844}. 

Since we are in the algebraic case and the constant $C_{\mathrm{tang}}$ is sufficiently large, the contribution from the tangential wave packets can be absorbed into the left hand side of \eqref{091030} and we obtain
\begin{equation}\label{0220843}
\sum_{O_j \in \mathcal{O}_{j,\mathrm{alg} }}
    \|H^{\lambda}g_{O_j}\|_{\mathrm{BL}_{k,A_j}^p(O_{j} \cap N_{\rho^{1/2+\delta_m}}(Y) )}^p
    \lesssim
   \sum_{O_j \in \mathcal{O}_{j,\mathrm{alg} }}
    \sum_{B \in \mathcal{B}(O_j)} \|H^{\lambda}g_{B,\mathrm{trans}}\|_{\mathrm{BL}_{k,A_{j+1}}^p(B)}^p+\mathrm{RapDec}(r)\|g\|_2^p.
\end{equation}
We apply Lemma \ref{probab} to each $g_{B,\mathrm{trans}}$ and obtain a collection $\mathcal{B}$ of points such that
\begin{equation}\label{091031}
    \|H^{\lambda}g_{B,\mathrm{trans}}\|_{\mathrm{BL}_{k,A_{j+1}}^p(B)}^p \lesssim
    (\log{r})^{2p}
    \sum_{b \in \mathcal{B} }
    \|H^{\lambda}{g}_{B,\mathrm{trans},b}\|_{\mathrm{BL}_{k,A_{k+1} }^p(B \cap N_{\rho_{j+1}^{1/2+\delta_m}}(Z_{O_j}+b) )}^p 
\end{equation}
and
\begin{equation}
    \sum_{b \in \mathcal{B} }
    \|{g}_{B,\mathrm{trans},b}\|_2^2 \lesssim \|g_{B,\mathrm{trans}}\|_2^2,
\end{equation}
where ${g}_{B,\mathrm{trans},b}:= (g_{B,\mathrm{trans}})_b$. 

We define the collection of step $j+1$ cells by
\begin{equation}
    \mathcal{O}_{j+1}(O_j):=\{B \cap N_{\rho_{j+1}^{1/2+\delta_m} }(Z+\by_{O_j}+b): B \in \mathcal{B}(O_j) \;\; \mathrm{and} \;\; b \in \mathcal{B}   \},
\end{equation}
\begin{equation}
    \mathcal{O}_{j+1}:= \bigcup_{O_j \in \mathcal{O}_{j,\mathrm{alg} }}\mathcal{O}_{j+1}(O_j).
\end{equation}
For every $O_{j+1} \in \mathcal{O}_{j+1}$, we define a function $g_{O_{j+1}}:= {g}_{B, \mathrm{trans},b}$ and a point $\by_{O_{j+1}}:=\by_{O_j}+b$. 
\medskip

It remains to show that all these outputs satisfy the desired properties. We will follow arguments in pages 259--261 of \cite{HR2019} for Property I, II, and III, and pages 19--20 of \cite{hickman2020note} for Property IV. These arguments are also similar to the counterpart of the cellular-dominant case.  Hence, we give only a sketch here.
First, it is straightforward to see that Property I follows from \eqref{091024}, \eqref{0220843} and \eqref{091031} together with Property 1 of the $j$th outputs.
Before proving Property II, we recall that our tubes are straight. We apply Lemma 5.7 of \cite{guth2018} to the wave packets at scale $\rho_{j+1}$, and by the $L^2$-orthogonality, we obtain
\begin{equation}
    \sum_{O_{j+1}} \|g_{O_{j+1}}\|_2^2 \lesssim d^n \sum_{O_j}\|g_{O_j}\|_2^2.
\end{equation}
Property II now follows from the above inequality with Property II of the $j$th outputs. To prove Property III, we apply Lemma \ref{09lem96} to the function $g_{B,\mathrm{trans}}$ with $\widetilde{\W}=\widetilde{\T}_{b}'$ in defined \eqref{0224718} and by the $L^2$-orthogonality, we obtain
\begin{equation}
\begin{split}
    \|g_{O_{j+1}}\|_2^2 \lesssim r^{O(\delta_m)}(\frac{\rho_{j}}{\rho_{j+1}})^{-\frac{n-m}{2}}\|g_{O_j}\|_2^2+\mathrm{RapDec}(r)\|g\|_2^2.
\end{split}
\end{equation}
It suffices to combine this inequality with Property III of the $j$th outputs. Property IV also
follows from an application of Lemma \ref{09lem96}. We leave out the details.

\section{A proof of Theorem \ref{201204thm5_1} }\label{f_section9}
In this section, we prove Theorem \ref{201204thm5_1}, that is, we prove
\begin{equation}\label{1219.14.1}
    \|H^{\lambda}g\|_{\mathrm{BL}_{k,A}^p(B_{R})}
    \lesssim_{\epsilon} R^{\epsilon}\|g\|_{L^2}^{2/p} \|g\|_{L^{\infty}}^{1-2/p}
\end{equation}
for every $1 \leq R \leq \lambda$ and ball $B_R \subset [-3C_n\lambda,3C_n\lambda]^{n-1} \times [R/C_n,C_n\lambda]$. We follow the proofs in Section 4, 6.3, and 6.4 of \cite{hickman2020note}. 
Note that \eqref{1219.14.1} is obviously true unless $g$ satisfies the \textit{non-degenerate} hypothesis:
\begin{equation}\label{nondegenerate}
    \|H^{\lambda}g\|_{\mathrm{BL}_{k,A}^p(B_R) } \gtrsim R^{\epsilon}\|g\|_{L^2}.
    \end{equation}
Hence, in this section, we always assume that $g$ satisfies the non-degenerate hypothesis.
\medskip

Consider a family of Lebesgue exponents $p_i$ for $k \leq i \leq n$ satisfying 
\begin{equation*}
p_{k} \geq p_{k+1} \geq \dots \geq p_n =: p \geq 2
\end{equation*}
and define $0 \leq \alpha_i, \beta_i \leq 1$ in terms of the $p_i$ by
\begin{equation*}
    \alpha_{i} := \Big(\frac{1}{2} - \frac{1}{p_i}\Big)^{-1}\Big(\frac{1}{2} - \frac{1}{p_{i+1}}\Big) \quad \textrm{and} \quad \beta_{i} := \Big(\frac{1}{2} - \frac{1}{p_i}\Big)^{-1}\Big(\frac{1}{2} - \frac{1}{p_n}\Big)
\end{equation*}
for $k \leq i \leq n - 1$ and $\alpha_n :=: \beta_n:=: \beta_{n+1} := 1$. All the exponents $p_i$ will be determined later.

\subsection{The second algorithm}

Let us explain the second algorithm \texttt{[alg 2]}.
\medskip

\noindent\underline{\texttt{Input}} 
The algorithm takes as its input:
\begin{itemize}
    \item A ball $B(\bx_0,R) \subset [-3C_n\lambda,3C_n\lambda]^{n-1} \times [R/C_n,C_n\lambda]$.
    \item An admissible large integer $A \in \N$. 
    \item A function $g \in L^1$ satisfying the  {non-degenerate hypothesis} \eqref{nondegenerate}.
\end{itemize}

\noindent\underline{\texttt{Output}} 
The $(n+1-l)$th step of the recursion will produce:
\begin{itemize}
    \item An $(n+1-l)$-tuple of:
    \begin{itemize}
        \item Scales $\vec{r}_l = (r_n, \dots, r_l)$ satisfying $R = r_n > r_{n-1} > \dots > r_{l}$;
        \item  Large and (in general) non-admissible parameters $\vec{D}_l = (D_n, \dots, D_{l})$;
        \item Integers $\vec{A}_l=(A_n,\ldots,A_l)$ such that $A=A_n > A_{n-1} > \cdots > A_l$.
    \end{itemize}
    \item For $l \leq l' \leq n$ a family $\vec{\mathcal{S}}_{l'}$ of level $n-l'$ multigrains. Each $\vec{S}_{l'} \in \vec{\mathcal{S}}_{l'}$ has multiscale $\vec{r}_{l'} = (r_n, \cdots, r_{l'})$ and complexity $O_{\epsilon}(1)$. The families have a nested structure in the sense that for each $l \leq l' < n$ and each $\vec{S}_{l'} \in \vec{\mathcal{S}}_{l'}$,  there exists some $\vec{S}_{l'+1} \in \vec{\mathcal{S}}_{l'+1}$ such that $\vec{S}_{l'} \preceq \vec{S}_{l'+1}$.
    \item For $l \leq l' \leq n$  an assignment of a function $g_{\vec{S}_{l'}}$ to each $\vec{S}_{l'} \in \vec{\mathcal{S}}_{l'}$. Each $g_{\vec{S}_{l'}}$ is tangent to $(S_{l'}, B_{r_{l'}})$, the final component of $\vec{S}_{l'}$. Moreover, $S_{l'}$ is of dimension $l'$.
\end{itemize}

All these outputs will be chosen so that the following properties hold true.
\medskip

\paragraph*{\underline{Property I}} The inequality holds true
\begin{equation}
\|H^{\lambda}g\|_{\mathrm{BL}_{k,A}^p(B_{R})} \lesssim  M(\vec{r}_l, \vec{D}_l)R^{O(\epsilon_{\circ})} \|g\|_{L^2(B^{n-1})}^{1-\beta_l} \Big( \sum_{\vec{S}_l \in \mathcal{\vec{S}}_l} \|H^{\lambda}g_{\vec{S}_l}\|_{\mathrm{BL}_{k,A_{l}}^{p_l}(B_{r_l})}^{p_l}\Big)^{\frac{\beta_l}{p_l}},
\end{equation}
where the pair $(S_{l},B_{r_l})$ is the last component of the multigrain $\vec{S}_l$ and
\begin{equation*}
    M(\vec{r}_l, \vec{D}_l) := \Big(\prod_{i=l}^{n}D_i\Big)^{(n-l)\delta}\Big(\prod_{i=l}^{n} r_i^{\frac12(\beta_{i+1}-\beta_i)}D_i^{\frac12(\beta_{i+1} - \beta_{l})}\Big).
\end{equation*}

\paragraph*{\underline{Property II}} 
\begin{equation}
 \sum_{\vec{S}_l \in \mathcal{\vec{S}}_l} \|g_{\vec{S}_l}\|_2^{2} \lesssim \Big(\prod_{i = l }^{n} D_{i}^{1 + \delta}\Big) R^{O(\epsilon_{\circ})}\|g\|_{L^2(B^{n-1})}^{2}. 
\end{equation}

\paragraph*{\underline{Property III}}   For $l'$ with $l+1 \leq l' \leq n$,
\begin{equation}
 \max_{\vec{S}_l \in \mathcal{\vec{S}}_l}\|g_{\vec{S}_l}\|_2^2 \lesssim \Big(\prod_{i= l}^{l'-1}\big(\frac{r_{i+1}}{r_i} \big)^{-\frac{n-i}{2}} D_{i}^{-i + \delta}\Big) R^{O(\epsilon_{\circ})}\max_{\vec{S}_{l'} \in \vec{\mathcal{S}}_{l'}} \|g_{\vec{S}_{l'}}\|_2^2.
\end{equation}

To state the last property, we need to introduce some notations. Let us consider a multigrain $\vec{S}_{l}=(\mathcal{G}_n,\ldots,\mathcal{G}_{l}) \in \vec{\mathcal{S}}_{l}$ and denote   $\mathcal{G}_i=(S_{i}, B(\bx_{n-i},r_{i}))$.
Given $\widetilde{\W} \subseteq \widetilde{\T}[B(\bx_{n-l},r_{l})]$, let $\uparrow\uparrow_{l} \!\!\widetilde{\W}$ denote the set of wave packets $T_{\theta, v} \in \T[B(\bx_{n-l-1},r_{l+1})]$ for which there exists some $\widetilde{T}_{\tilde{\theta}, \tilde{v}} \in \widetilde{\W}$ satisfying
\begin{equation}
\label{two-upper-arrow}
    \dist(\tilde{\theta},\theta) \leq C_{\circ} r_{l}^{-1/2} \quad \textrm{and} \quad \dist\big(\widetilde{T}_{\tilde{\theta},\tilde{v}}(\bx_{n-l}),\, T_{\theta,v}(\bx_{n-l-1}) \cap B(\bx_{n-l},r_{l})\big) \leq C_{\circ} r_{l+1}^{1/2+\delta}.
\end{equation}
Here, the constant $C_{\circ}$ is the constant mentioned in the discussion below the inequality \eqref{uparrowj}.
\medskip

\paragraph*{\underline{Property IV}} 
Let $l \leq l' \leq n-1$. 
For every $\vec{S}_{l'}
\in \vec{\mathcal{S}}_{l'}$ 
, $\vec{S}_{l'+1} \in \vec{\mathcal{S}}_{l'+1}$ with $\vec{S}_{l'} \preceq \vec{S}_{l'+1}$, and  $\widetilde{\W} \subset \widetilde{\T}[B(\bx_{n-l'},r_{l'})]$, it holds that 
\begin{equation}
    \|{g}_{\vec{S}_{l'}}|_{\widetilde{\W}}\|_{L^2(B^{n-1})}^2 \lesssim \big( \frac{r_{l'+1}}{r_{l'}} \big)^{-\frac{n-l'}{2}} (D_{l'})^{\delta}  R^{O(\epsilon_{\circ})}\|g_{\vec{S}_{l'+1}}|_{(\uparrow\uparrow_{l'} \widetilde{\W})^{*}} \|_{2}^2
    +\mathrm{RapDec}(R)\|g\|_{L^2}^2.
\end{equation}
Here, the notation $*$ was introduced in \eqref{020184}

\paragraph*{\underline{\texttt{Stopping conditions}}} 

Suppose that we have the outputs of the $(n+1-l)$th stage of \texttt{[alg 2]}. We terminate our algorithm if the following condition \texttt{[tiny-dom]} is satisfied.
\begin{itemize}
\item[\texttt{Stop:[tiny-dom]}] 
The following inequality holds true:
\begin{equation}
    \sum_{\vec{S}_l \in \vec{\mathcal{S}}_l }\|H^{\lambda}g_{\vec{S}_l }\|_{\mathrm{BL}_{k,A_l}^{p_l}(B_{r_l}) }^{p_l} \leq 2 \sum_{\vec{S}_l \in \vec{\mathcal{S}}_{l,tiny}}\|H^{\lambda}g_{\vec{S}_l}\|_{\mathrm{BL}_{k,A_l}^{p_l}(B_{r_l}) }^{p_l},
\end{equation}
where the right-hand summation is restricted to those $\vec{S_l} \in \vec{\mathcal{S}}_l$ for which \texttt{[alg 1]} terminates due to the stopping condition \texttt{[tiny]}.
\end{itemize}

\subsection{A construction of outputs of the second algorithm} In this subsection, we construct the outputs and show that they satisfy the desired properties in the algorithm. We first define the initial outputs as follows.
\begin{itemize}
    \item $r_n:=R,\;\; D_n:=1, \;\; A_n:=A$.
    \item $\vec{\mathcal{S}}_n:=\{\vec{S}_n \}$ and $\vec{S}_n:=(\R^n,B(\bx_0,R))$.
    \item $g_{\vec{S}_n}:=g$.
\end{itemize}
These outputs vacuously satisfy the desired properties.

Let us now assume that we have the outputs of the $(n+1-l)$th step for some $1 \leq l \leq n$ and the stopping condition \texttt{[tiny-dom]} fails. We need to construct the outputs of the $(n+2-l)$th step so that they satisfy the desired properties. By the failure of the stopping condition \texttt{[tiny-dom]}, we know that \texttt{[alg 1]} stops due to \texttt{[tang]}. Hence,
\begin{equation}
    \sum_{\vec{S}_l \in \vec{\mathcal{S}}_l }\|H^{\lambda}g_{\vec{S}_l }\|_{\mathrm{BL}_{k,A_l}^{p_l}(B_{r_l}) }^{p_l} \leq 2 \sum_{\vec{S}_l \in \vec{\mathcal{S}}_{l,\mathrm{tang}}}\|H^{\lambda}g_{\vec{S}_l}\|_{\mathrm{BL}_{k,A_l}^{p_l}(B_{r_l}) }^{p_l},
\end{equation}
where the right-hand summation is restricted to those $\vec{S}_l \in \vec{\mathcal{S}}_{l}$ for which \texttt{[alg 1]} stops due to \texttt{[tang]}. By the definition of \texttt{[tang]} and the first three properties of \texttt{[alg 1]}, for each multigrain $\vec{S}_l \in \vec{\mathcal{S}}_{l,\mathrm{tang}}$, there exist 
\begin{itemize}
    \item a collection $\mathcal{S}_{l-1}[\vec{S}_l]$ of grains of dimension $l-1$, some scale $r_{l-1}$  and degree  $O(1)$
    \item an assignment of a function $g_{(S,B_{r_{l-1}})}$ to each  $(S,B_{r_{l-1}}) \in \mathcal{S}_{l-1}[\vec{S}_l]$ tangent to the grain $(S,B_{r_{l-1}})$
    \item some parameter $D_{l-1}:=d^{\#_{\stc}(j_0)}$ where \texttt{[alg 1]} terminates at the $j_0$th stage
    \item some parameter $A_{l-1}=A_l/2^{\#_{\sta}(j_0)}$
\end{itemize}
such that 
\begin{align} \label{1218.143}
     \|H^{\lambda}g_{\vec{S}_l }\|_{\mathrm{BL}_{k,A_l}^p(B_{r_l} )}^p  &\lesssim   R^{O(\epsilon_{\circ})}\!\!\!\!\!\!\!\!\!
     \sum_{(S,B_{r_{l-1}}) \in \mathcal{S}_{l-1}[\vec{S}_l] }\!\!\!\!\!\!\!\!\!\! \|H^{\lambda}\big((g_{\vec{S}_l})_{(S,B_{r_{l-1}})}\big)\|_{\mathrm{BL}_{k,A_{l-1}}^p(B_{r_{l-1}})}^p+\mathrm{RapDec}(R)\|g_{\vec{S}_l}\|_2^p,
     \\ \label{1218.144}
   \sum_{(S,B_{r_{l-1}}) \in \mathcal{S}_{l-1}[\vec{S}_l] }\|&(g_{\vec{S}_l})_{(S,B_{r_{l-1}})}\|_2^2  \lesssim 
   D_{l-1}^{1+\delta} R^{O(\epsilon_{\circ})} 
   \|g_{\vec{S}_l}\|_{2}^2,
   \\ \label{1218.145}
    \max_{(S,B_{r_{l-1}}) \in \mathcal{S }_{l-1}[\vec{S}_l]} \|&(g_{\vec{S}_l})_{(S,B_{r_{l-1}})}\|_{2}^2 \lesssim \Big(\frac{r_l}{r_{l-1}}\Big)^{-\frac{n-l}{2}}
    D_{l-1}^{-(l-1)+\delta}R^{O(\epsilon_{\circ})}
    \|g_{\vec{S}_{l}}\|_{2}^2,
    \\ \label{1219.146}
    \|({g}_{\vec{S}_{l}})_{(S,B_{r_{l-1}})}|_{\widetilde{\W}}&\|_2^2
    \lesssim
    \big( \frac{r_{l}}{r_{l-1}} \big)^{-\frac{n-l}{2}} D_{l-1}^{\delta}  R^{O(\epsilon_{\circ})}    
    \|g_{\vec{S}_l} |_{\uparrow^l (\uparrow \widetilde{\W})} \|_{2}^2+\mathrm{RapDec}(R)\|g_{\vec{S}_l}\|_2^2,
\end{align}
for every sub-collection $\widetilde{\W}$ of the wave packets at scale $r_{l-1}$. Here, we recall that the notations $\uparrow^l$ and $\uparrow$ are defined in $\eqref{uparrowj}$ and $\eqref{uparrow1}$, respectively.

Note that the parameters $r_{l-1}$, $D_{l-1}$, and $A_{l-1}$ may depend on the choice of the multigrain $\vec{S}_{l}$. To take uniform parameters independent of the choice, we apply a pigeonholing argument by losing some $(\log{R})^C$. By taking a sub-collection of $\vec{S}_{l,\mathrm{tang}}$ and abusing the notation, we may say that for every $\vec{S}_l \in \vec{S}_{l,\mathrm{tang}}$ the collection    $\mathcal{S}_{l-1}[\vec{S}_l]$ has the uniform parameters $r_{l-1}$, $D_{l-1}$ and $A_{l-1}$.

We define a family of multigrains by
\begin{equation}
    \vec{\mathcal{S}}_{l-1}:=\big\{\vec{S}_{l-1}:=
    (\vec{S}_{l},(S,B_{r_{l-1}})): (S,B_{r_{l-1}}) \in \mathcal{S}_{l-1}[\vec{S}_l]
    \big\}
\end{equation}
and the function $g_{\vec{S}_{l-1}}:=(g_{\vec{S}_l})_{(S,B_{r_{l-1}})}$. It is elementary to see that Property II follows from \eqref{1218.144} and Property III follows from \eqref{1218.145}. Also, by using H\"{o}lder's inequality (see pages 266--267 of the paper of \cite{HR2019}), Property I follows from \eqref{1218.144}. We leave out the details.
It remains to show Property IV. Since our outputs at the $(n+1-l)$th step satisfy Property IV, it suffices to prove for the case that $l'=l-1$. By \eqref{1219.146} and Property III, it suffices to prove that
\begin{equation}
    \|g_{\vec{S}_l} |_{\uparrow^l (\uparrow \widetilde{\W})} \|_{L^2}^2
    \lesssim \|g_{\vec{S}_l} |_{(\uparrow \uparrow_{l-1} \widetilde{\W})^{*} } \|_{L^2}^2 + \mathrm{RapDec}(R)\|g\|_{L^2}^2,
\end{equation}
where the notation $\uparrow \uparrow_{l} \! \! \widetilde{\W}$ was introduced in \eqref{two-upper-arrow} and $\W^\ast$ was defined in \eqref{020184}. If we can show $\uparrow^l \! \!(\uparrow \! \! \widetilde{\W}) \subset \uparrow \uparrow_{l-1} \! \! \widetilde{\W}$, then the estimate above follows from Lemma \ref{210103lem5_1} and Lemma \ref{intersectionofsets}. In fact, on page 20 of \cite{hickman2020note}, it is proved that $\uparrow^l \! \!(\uparrow \! \! \widetilde{\W}) \subset \uparrow^{l-1}\widetilde{\W}$. By the definition of $\uparrow \uparrow_{l-1} \! \! \widetilde{\W}$, we also know that $\uparrow^{l-1}\widetilde{\W} \subset \uparrow \uparrow_{l-1} \! \! \widetilde{\W}$, hence, $\uparrow^l \! \!(\uparrow \! \! \widetilde{\W}) \subset \uparrow \uparrow_{l-1} \! \! \widetilde{\W}$. This gives Property IV and finishes the proof.


\subsection{A proof of Theorem \ref{201204thm5_1} }

Let us prove the broad norm estimate \eqref{1219.14.1}.
We first state the vanishing property of the broad norm. 

\begin{lem}\label{vanishinglemma}
Let $1< r \leq R$, and let $1 \leq m <k \leq n$, and let ${Z}$ be an $m$-dimensional transverse complete intersection. Suppose that $g$ is concentrated on wave packets from $\mathbb{T}_{{Z}}[B(\widetilde{\bx}_0,r)]$. Then for every ball $B(\widetilde{\bx}_0,r) \subset [-3C_n \lambda, 3 C_n \lambda]^{n-1}\times [R/C_n, C_n\lambda]$, it holds that
\begin{equation}
    \|H^{\lambda}g\|_{\BLka^p(B(\widetilde{\bx}_0,r))}=\rm{RapDec}(r)\|g\|_{2}.
\end{equation}
\end{lem}

The proof is straightforward and we leave out the details here. The interested
reader should consult the proof of the analogous result in Page 339 of \cite{MR4047925}.
\medskip

 As a consequence of \texttt{[alg 2]} and Lemma \ref{vanishinglemma}, we will obtain the following \textit{multiscale grains decomposition}, which is the counterpart of that in Section 4.1 of \cite{hickman2020note}.
\medskip

\noindent\underline{\texttt{Input}} 
The algorithm takes as its input:
\begin{itemize}
    \item A ball $B(\bx_0,R) \subset [-3C_n\lambda,3C_n\lambda]^{n-1} \times [R/C_n,C_n\lambda]$.
    \item An admissible large integer $A \in \N$.
    \item A function $g \in L^1$ satisfying the  {non-degenerate hypothesis} \eqref{nondegenerate}.
\end{itemize}
\medskip

\noindent\underline{\texttt{Output}} The algorithm produces
\begin{itemize}
    \item $\mathcal{O}$ a finite collection of open subsets of $\R^n$ of diameter at most $R^{\epsilon_{\circ} }$
    \item A dimension $m$ with $k \leq m \leq n$ and an integer parameter $1 \leq A_{m-1} \leq A$
        \item Scales $\vec{r}_m = (r_n, \dots, r_m)$ satisfying $R = r_n > r_{n-1} > \dots > r_{m}$
        \item  Large  non-admissible parameters $\vec{D}_{m-1} = (D_n, \dots, D_{m-1})$.
      \item For $m \leq l \leq n$ a family $\vec{\mathcal{S}}_{m}$ of level $n-l$ multigrains. Each $\vec{S}_{l} \in \vec{\mathcal{S}}_{l}$ has multiscale $\vec{r}_{l} = (r_n, \cdots, r_{l})$ and complexity $O_{\epsilon}(1)$. The families have a nested structure in the sense that for each $m \leq l \leq n$ and each $\vec{S}_{l} \in \vec{\mathcal{S}}_{l}$,  there exists some $\vec{S}_{l+1} \in \vec{\mathcal{S}}_{l+1}$ such that $\vec{S}_{l} \preceq \vec{S}_{l+1}$.
    \item For $m \leq l \leq n$  an assignment of a function $g_{\vec{S}_{l}}$ to each $\vec{S}_{l} \in \vec{\mathcal{S}}_{l}$. Each $g_{\vec{S}_{l}}$ is tangent to $(S_{l}, B_{r_{l}})$, the final component of $\vec{S}_{l}$. Moreover, $S_l$ is of dimension $l$.
\end{itemize}

All these outputs will be chosen so that the following properties hold true.
\medskip

\paragraph*{\underline{Property I}} The inequality holds true
\begin{equation}
\|H^{\lambda}g\|_{\mathrm{BL}_{k,A}^p(B_{R})} \lesssim  M(\vec{r}_m, \vec{D}_m)R^{O(\epsilon_{\circ})} \|g\|_{L^2(B^{n-1})}^{1-\beta_m} \Big( \sum_{O \in \mathcal{O} } \|H^{\lambda}g_{O}\|_{\mathrm{BL}_{k,A_{m}}^{p_m}({O})}^{p_m}\Big)^{\frac{\beta_m}{p_m}},
\end{equation}
where
\begin{equation*}
    M(\vec{r}_m, \vec{D}_m) := \Big(\prod_{i=m}^{n}D_i\Big)^{(n-m)\delta}\Big(\prod_{i=m}^{n} r_i^{\frac12(\beta_{i+1}-\beta_i)}D_i^{\frac12(\beta_{i+1} - \beta_{m})}\Big).
\end{equation*}

\paragraph*{\underline{Property II}} 
\begin{equation}\label{1220.1417}
 \sum_{O \in \mathcal{O} } \|g_{O}\|_2^{2} \lesssim \Big(\prod_{i = m-1 }^{n} D_{i}^{1 + \delta}\Big) R^{O(\epsilon_{\circ})}\|g\|_{L^2(B^{n-1})}^{2}. 
\end{equation}

\paragraph*{\underline{Property III}}   For $m \leq l \leq n$
\begin{equation}\label{1220.1418}
 \max_{O \in \mathcal{O} }\|g_{O}\|_2^2 \lesssim r_{l}^{-\frac{n-l}{2}}\Big(\prod_{i= m-1}^{l-1}r_{i}^{-1/2} D_{i}^{-i + \delta}\Big) R^{O(\epsilon_{\circ})}\max_{\vec{S}_{l} \in \vec{\mathcal{S}}_{l}} \|g_{\vec{S}_{l}}\|_2^2.
\end{equation}

\paragraph*{\underline{Property IV}}   For $m \leq l \leq n-1$
\begin{equation}\label{1220.1419}
    \|g_{\vec{S}_{l}}\|_{L^2(B^{n-1})}^2 \lesssim r_{l}^{\frac{n-l}{2}}\Big(\prod_{i = l}^{n-1}r_i^{-1/2}D_i^{\delta} \Big)  R^{O(\epsilon_{\circ})}\|g_{\vec{S}_{l}}^{\#}\|_{2}^2,
\end{equation}
where
\begin{equation}\label{sharpnotation}
   g_{\vec{S}_{l}}^{\#} :=  g|_{\T[\vec{S}_l]}.
\end{equation}
Here, $\T[\vec{S}_l]$ is defined in Definition \ref{nestedtube}.
\\

Let us explain how we obtain the multiscale grains decomposition. We first note that \texttt{[alg 2]} terminates at $(n+1-m)$th step with $m \geq k$. Otherwise, \texttt{[alg 2]} does not terminate at $(n+1-k)$th step, and as an application of Lemma \ref{vanishinglemma}, we  obtain that
$\|H^{\lambda}g\|_{\BLka^p(B_R)} =\mathrm{RapDec}(R)\|g\|_2$. However, since we are assuming that $g$ satisfies the non-degenerate hypothesis \eqref{nondegenerate}, this does not take place.  

By the stopping condition of \texttt{[alg 2]}, we know that \texttt{[alg 1]} terminates due to \texttt{[tiny]}. Let us denote by $\mathcal{O}$ a final collection of cells and define $A_{m-1}$ and $D_{m-1}$ as in pages 267--268 of \cite{HR2019}. Then it is straightforward to see that Property I, II and III follow from those of \texttt{[alg 2]}. Let us explain how Property IV can be deduced from that of \texttt{[alg 2]}. Let us fix a $n-l$ level multigrain $\vec{S}_l$. Since $\vec{S}_l$ has a nested structure, we can take $n-i$ level multigrains $\vec{S}_i$ such that
\begin{equation}
    \vec{S}_l \preceq \vec{S}_{l+1} \preceq \cdots \preceq \vec{S}_n.
\end{equation}
For $l \leq i \leq n$, if $(S_i,B(\bx_{n-i},r_i))$ denotes the  $(n-i+1)$th component of $\vec{S}_l$, then let $\T_{\mathrm{tang}}[S_i]$ denote the set of all wave packets of scale $r_i$  tangent to $S_i$ in $B(\bx_{n-i},r_i)$.
We construct sets $\W_i \subset \T_{\mathrm{tang}}[S_i]$ for $l \leq i \leq n $ as follows; We first set
\begin{equation}
    \W_l:=\T_{\mathrm{tang}}[S_l],
\end{equation}
and define recursively
\begin{equation}\label{1220.1423}
\W_{i+1}:=\T_{\mathrm{tang}}[S_l] \cap (\uparrow\uparrow_i \!\! \W_i)^*.
\end{equation}
By Property IV of \texttt{[alg 2]}, we know that
\begin{equation}\label{1220.1424}
    \|{g}_{\vec{S}_{i}}|_{\W_i}\|_{L^2(B^{n-1})}^2 \lesssim \big( \frac{r_{i+1}}{r_{i}} \big)^{-\frac{n-i}{2}} D_{i}^{\delta}  R^{O(\epsilon_{\circ})}\|g_{\vec{S}_{i+1}}|_{(\uparrow\uparrow_{i} \W_i)^*} \|_{2}^2
    +\mathrm{RapDec}(R)\|g\|_{L^2}^2
\end{equation}
for $l \leq i \leq n$.
Recall that $g_{\vec{S}_{i+1}}$ is concentrated on wave packets belonging to $\T_{\mathrm{tang}}[S_{i+1}]$, and thus, we know that
\begin{equation}\label{1220.1425}
    g_{\vec{S}_{i+1}}=g_{\vec{S}_{i+1} }|_{\T_{\mathrm{tang}[S_{i+1}] }}+\mathrm{RapDec}(R)\|g_{\vec{S}_{i+1}}\|_2.
\end{equation}
By \eqref{1220.1423}, \eqref{1220.1424}, \eqref{1220.1425}, Lemma \ref{intersectionofsets}, and Property III, we obtain that
\begin{equation}
    \|{g}_{\vec{S}_{i}}|_{\W_i}\|_{L^2(B^{n-1})}^2 \lesssim \big( \frac{r_{i+1}}{r_{i}} \big)^{-\frac{n-i}{2}} D_{i}^{\delta}  R^{O(\epsilon_{\circ})}\|g_{\vec{S}_{i+1}}|_{\W_{i+1}} \|_{2}^2
    +\mathrm{RapDec}(R)\|g\|_{L^2}^2.
\end{equation}
We iterate this inequality, and it suffices to set the function
$g^{\#}_{\vec{S}_l}:=g|_{\W_0}$ and notice that this function satisfies the nested tube hypothesis in Definition \ref{nestedtube}. We refer to pages 22-23 of \cite{hickman2020note} for more details. This completes the proof of Property IV.
\\

Let us now see how the multiscale grains decomposition can be applied to prove Theorem \ref{201204thm5_1}. Recall that \eqref{1219.14.1} is vacuously true unless $g$ satisfies the non-degenerate hypothesis \eqref{nondegenerate}. Hence, we may assume that $g$ satisfies the non-degenerate hypothesis. We now apply the above multiscale grains decomposition and obtain
\begin{equation}
\|H^{\lambda}g\|_{\mathrm{BL}_{k,A}^p(B_{R})} \lesssim  M(\vec{r}_m, \vec{D}_m)R^{O(\epsilon_{\circ})} \|g\|_{L^2(B^{n-1})}^{1-\beta_m} \Big( \sum_{O \in \mathcal{O} } \|H^{\lambda}g_{O}\|_{\mathrm{BL}_{k,A_{m}}^{p_m}({O})}^{p_m}\Big)^{\frac{\beta_m}{p_m}}.
\end{equation}
Since each element of $\mathcal{O}$ has a diameter at most $R^{\epsilon_{\circ}}$, we obtain
\begin{equation}
    \|H^{\lambda}g_{O}\|_{\mathrm{BL}_{k,A_{m}}^{p_m}({O})}
    \lesssim R^{C\epsilon_{\circ}}\|g_{O}\|_2.
\end{equation}
By combining these two inequalities, we obtain
\begin{equation}
    \begin{split}
        \|H^{\lambda}g\|_{\mathrm{BL}_{k,A}^p(B_{R})} &\lesssim R^{C\epsilon_{\circ} }  M(\vec{r}_m, \vec{D}_m) \|g\|_{L^2(B^{n-1})}^{1-\beta_m} \Big( \sum_{O \in \mathcal{O} } \|g_{O}\|_{L^2}^{p_m}\Big)^{\frac{\beta_m}{p_m}}
        \\&
        \lesssim R^{C\epsilon_{\circ} }  M(\vec{r}_m, \vec{D}_m) \|g\|_{L^2(B^{n-1})}^{1-\beta_m} \Big( \underbrace{\sum_{O \in \mathcal{O} } \|g_{O}\|_{L^2}^{2}}_{\mathrm{Part\; I} }\Big)^{\frac{\beta_m}{p_m}}
        \big(\underbrace{\sup_{O \in \mathcal{O} }
        \|g_O\|^2_{{2}}}_{\mathrm{Part \; II}} \big)^{(\frac{1}{2}-\frac{1}{p_m}){\beta_m}}.
    \end{split}
\end{equation}
We apply \eqref{1220.1417} to Part I. Then by the definition of $\beta_m$, the right hand side is bounded by
\begin{equation}\label{1220.1430}
    \|H^{\lambda}g\|_{\mathrm{BL}_{k,A}^p(B_{R})} \lesssim 
    \Big(\prod_{i=m-1}^n r_i^{\frac12(\beta_{i+1}-\beta_i)}D_i^{\frac{\beta_{i+1}}{2}-(\frac12-\frac{1}{p_n})+C\delta }  \Big) R^{C\epsilon_{\circ}}
    \|g\|_2^{\frac{2}{p_n}}
        \big(\underbrace{\sup_{O \in \mathcal{O} }
        \|g_O\|^2_{{2}}}_{\mathrm{Part \; II}} \big)^{\frac12-\frac{1}{p_n}}
\end{equation}
where $r_{m-1}:=1$.
This inequality is the counterpart of the last inequality of Section 4.1 of \cite{hickman2020note}.
\medskip

To deal with Part II, we need the following lemma, which is the counterpart of Lemma 4.3 of \cite{hickman2020note}. This lemma is a corollary of the nested polynomial Wolff axioms.

\begin{lem}[cf. Lemma 4.3 of \cite{hickman2020note}]\label{lem92}
For $m \leq l \leq n$, and the functions $g^{\#}_{\vec{S}_l}$ defined in \eqref{sharpnotation}, it holds that
\begin{equation}
    \max_{\vec{S}_l \in \vec{\mathcal{{S}}}_l }
    \|g_{\vec{S}_{l}}^{\#}\|_{L^2}^2
    \lesssim
    \Big(\prod_{i=l}^{n-1}r_i^{-1/2} \Big)
    R^{\epsilon_{\circ}}
    \|g\|_{L^\infty}^2.
\end{equation}
\end{lem}

\begin{proof}
By  the $L^2$-orthogonality, we know that
\begin{equation}
    \|g_{\vec{S}_{l}}^{\#}\|_2^2 \lesssim \sum_{\theta \in \Theta[{\vec{S}_{l}}] } \sum_{v:T_{\theta,v} \in \T[\vec{S}_{l}]} \|g_{T_{\theta,v}}\|_2^2.
\end{equation}
Since $g_{T_{\theta,v}}$ is supported on $\theta$, by taking the maximum over $\theta$, this is bounded by
\begin{equation}
    \# \Theta[\vec{S}_{l}]\max_{\theta \in \Theta_R } \Big( \sum_{v \in R^{1/2}\Z^{n-1}}\|g_{T_{\theta,v}}\|_{L^2(\theta)}^2 \Big).
\end{equation}
By the $L^2$-orthogonality and replacing the $L^2$-norm by the $L^{\infty}$-norm, it is further bounded by
\begin{equation}
    \#\Theta[\vec{S}_{l}]R^{-(n-1)/2} \|g\|_{\infty}^2.
\end{equation}
It suffices now to apply Lemma \ref{nestedpoly} and bound $\#\Theta[\vec{S}_l]$.
\end{proof}

After applying \eqref{1220.1418}, \eqref{1220.1419}, and this lemma  to Part II, we obtain
\begin{equation}\label{1220.1431}
    \max_{O \in \mathcal{O}}\|g_{O}\|_{L^2}^2 \lesssim_{\epsilon}
    \Big(\prod_{i=m-1}^{n-1}r_i^{-1/2}D_i^{\delta} \Big)
    \Big(\prod_{i=l}^{n-1}r_i^{-1/2} \Big)
    \Big(\prod_{i=m-1}^{l-1}D_i^{-i} \Big)R^{C\epsilon_{\circ}}
    \|g\|_{L^\infty}^2
\end{equation}
for every $m \leq l \leq n$.
This inequality is the counterpart of the inequality (4.2) of \cite{hickman2020note}. 
We now combine \eqref{1220.1430} and \eqref{1220.1431} as in Section 4.3 of \cite{hickman2020note} and obtain the obtain the desired estimate \eqref{1219.14.1}. Since the proof is identical to theirs, we leave out the details.

\section{Appendix: a transversality lemma}\label{f_section10}
The appendix is devoted to a transversality lemma . First, let us introduce one more definition.
\begin{defi}
Let $M,N$ be smooth manifolds. Assume that $f:M\to N$ and $A\subset N$ a submanifold. Then $f$ is said to be transverse to $A$, which is denoted by $f\pitchfork A$, if for any $x\in M$ with $f(x)=y\in A$, the tangent space $T_yN$ is spanned by $T_yA$ and the image $Df_x(T_xM)$. 
\end{defi}
Now recall the follow version of the transversality theorem.
\begin{thm}[\cite{Guillemin-Pollack}, Page 68]
\label{transversality-thm}
Let $X,S,Y$ be smooth manifolds without boundary and $A\subset N$ a smooth submanifold. Let $F:V\to C^\infty(M,N)$ satisfy the following condition:
\begin{enumerate}
    \item the evaluation map $F^{ev}:X\times S\to Y,(x,s)\to F_s(x)$ is $C^\infty$.
    \item $F^{ev}$ is transverse to A. 
\end{enumerate}
Then the complement of the set
\begin{equation}
    \pitchfork(F;A):=\{s\in S:F_s\pitchfork A\}
\end{equation}
in $V$ has measure zero.
\end{thm}

Let $\Phi:\ZR^n\to\ZR^n$ be a smooth map and let $Z\subset\ZR^n$ be a smooth submanifold. Fix $k$ vectors $ m_1,\ldots, m_k\in\ZR^n$, $1\leq k\leq n$. Consider a family of parallel affine subspaces $\{\Pi_c\}_{{c}\in\ZR^k}$ in $\ZR^n$ formed by $(m_1,\ldots, m_k)$:
\begin{equation}
    \Pi_c:\left\{ \begin{array}{c}
    m_1\cdot x-c_1=0,\\
    \ldots\\
    m_k\cdot x-c_k=0.
    \end{array}\right.
\end{equation}
We show that $\Phi^{-1}(\Pi_c)$ is transvserse to $Z$ for generic $\Pi_c$.
\begin{lem}
\label{transverse-lemma}
The complement of the
\begin{equation}
    \cC_\Phi:=\{c\in\ZR^k:\Phi^{-1}(\Pi_c){\rm{~is~transverse~to~}}Z\}
\end{equation}
in $\ZR^k$ has measure zero.
\end{lem}
\begin{proof}
We first take $X=Z$, $S=\ZR^k$, $Y=\ZR^k$, $A=0$, and $F^{ev}=(m_1\cdot\Phi(z)-v_1,\ldots, m_k\cdot\Phi(z)-v_k)$ be the evaluation map in Theorem \ref{transversality-thm}. Then the complement of the set
\begin{equation}
    \cC_F:=\{c\in\ZR^k:F_c\pitchfork A\}
\end{equation}
in $\ZR^k$ has measure zero. Then we take $X=\ZR^n$, $S=\ZR^k$, $Y=\ZR^k$, $A=0$, and $G^{ev}=(m_1\cdot\Phi(z)-v_1,\ldots, m_k\cdot\Phi(z)-v_k)$. One can argue similarly to get that the complement of
\begin{equation}
    \cC_G:=\{c\in\ZR^k:G_c\pitchfork A\}
\end{equation}
in $\ZR^k$ has measure zero. Take $\cC=\cC_F\cap\cC_G$. We claim that $\cC\subset\cC_\Phi$, which proves the lemma.

To prove our claim above, first note that $\Phi^{-1}(\Pi_c)\subset\ZR^n$ is a smooth embedding of dimension $n-k$ when $c\in\cC$, since  $G_c\pitchfork A$. Take any $z\in Z\cap \Phi^{-1}(\Pi_c)$. We need to show $T_zZ+T_z\Phi^{-1}(\Pi_c)=\ZR^n$. Notice that
\begin{enumerate}
    \item $G|_Z=F$, so $D_zG|_{T_zZ}=D_zF$.
    \item $D_zF$ is surjective since $F_c\pitchfork A$.
    \item $G(\Phi^{-1}(\Pi_c))=0$, which implies $D_zG|_{T_z\Phi^{-1}(\Pi_c)}=0$.
\end{enumerate}
Thus, we have
\begin{enumerate}
    \item $D_z G$ is surjective on $T_zZ$.
    \item $\ker(D_zG)=T_z\Phi^{-1}(\Pi_c)$, since both $\Phi^{-1}(\Pi_c)$ and $\ker(D_zG)$ have dimension $n-k$. 
\end{enumerate}
Hence $T_zZ$ and $\ker(D_zG)=T_z\Phi^{-1}(\Pi_c)$ span the whole space $\ZR^n$ as desired.
\end{proof}

	\bibliography{reference}{}
	\bibliographystyle{alpha}

\end{document}